\documentclass[11pt]{article}

\usepackage[margin=1in]{geometry}
\usepackage{amsmath,amssymb,amsfonts,amsthm,mathtools}
\usepackage{bm}
\usepackage{enumitem}
\usepackage{xcolor}
\usepackage{float}
\usepackage{tikz}
\usetikzlibrary{arrows.meta}
\usepackage[colorlinks=true,linkcolor=blue,citecolor=blue,urlcolor=blue]{hyperref}
\providecommand{\keywords}[1]{%
  \par\smallskip\noindent\textbf{Keywords: }#1\par}
\providecommand{\subjclass}[2][]{%
  \par\smallskip\noindent\textbf{#1 Mathematics Subject Classification: }#2\par}

\numberwithin{equation}{section}

\theoremstyle{plain}
\newtheorem{theorem}{Theorem}[section]
\newtheorem{proposition}[theorem]{Proposition}
\newtheorem{lemma}[theorem]{Lemma}
\newtheorem{corollary}[theorem]{Corollary}

\theoremstyle{definition}
\newtheorem{assumption}[theorem]{Assumption}
\newtheorem{definition}[theorem]{Definition}

\theoremstyle{remark}
\newtheorem{remark}[theorem]{Remark}

\newcommand{\R}{\mathbb R}
\newcommand{\eps}{\varepsilon}
\newcommand{\D}{\mathsf D}
\newcommand{\Q}{\mathsf Q}
\renewcommand{\Cap}{\operatorname{Cap}}
\newcommand{\supp}{\operatorname{supp}}
\newcommand{\dist}{\operatorname{dist}}
\newcommand{\E}{\mathbb E}
\newcommand{\Pbb}{\mathbb P}
\newcommand{\calD}{\mathcal D}
\newcommand{\calL}{\mathcal L}
\newcommand{\calC}{\mathcal C}
\newcommand{\calK}{\mathcal K}
\newcommand{\bfone}{\mathbf 1}

\title{Weak Equilibrium Measures and Capacity--Hitting Identities\\
for the Hypoelliptic Third-Order Langevin Diffusion}
\author{
Ping He\thanks{School of Mathematics, Shanghai University of Finance and Economics, Shanghai, People's Republic of China; \texttt{pinghe@mail.shufe.edu.cn}}
\and
Xiaodan Li\thanks{School of Mathematics, Shanghai University of Finance and Economics, Shanghai, People's Republic of China; \texttt{lixiaodan@mail.shufe.edu.cn}}
\and
Yingli Wang\thanks{School of Mathematical Sciences, Fudan University, Shanghai, People's Republic of China; \texttt{yingliwang@fudan.edu.cn}}
\and
Lingjiong Zhu\thanks{Department of Mathematics, Florida State University, Tallahassee, Florida, United States of America; \texttt{zhu@math.fsu.edu}}}
\date{\today}

\begin{document}
\maketitle

\begin{abstract}
We construct weak equilibrium measures and weak capacities for the
hypoelliptic third-order Langevin diffusion motivated by an accelerated sampling algorithm (Mou et al.
(2021) \textit{J. Mach. Learn. Res.}, \textbf{22}(42), 1--41).  In this
process, the Brownian noise acts only in the highest-order auxiliary variable
and reaches the physical variables through a third-order H\"ormander chain.
Characteristic points of phase-space balls therefore require an alternative to
the standard uniformly elliptic boundary-flux theory.
We prove an elliptic-regularization stability theorem for the
corresponding hitting laws and then define the weak equilibrium measure and
weak capacity.  
The
proof combines the boundary-hitting stability strategy of Lee--Ramil--Seo
(2026, \textit{arXiv:2503.12610v2}) with localized hypoelliptic heat-kernel
estimates (Pigato (2022) \textit{Stoch. Process. Appl.}, \textbf{145},
117--142) adapted to the third-order chain.  
We obtain the bounded-domain weak
capacity--hitting identity and a Lyapunov drift argument in the
spirit of Lee--Ramil--Seo that yields recurrence and extends the
construction to a whole-space weak equilibrium measure, and whole-space
capacity--hitting identity.
\end{abstract}

\keywords{weak equilibrium measure; weak capacity;
capacity--hitting identity; third-order Langevin diffusion;
hypoellipticity; elliptic regularization;
hitting-law stability; metastability.}

\subjclass[2020]{Primary 60J45;
Secondary 35H10, 60H10, 60J60.}


\section{Introduction}

\textit{Langevin algorithms} are popular Markov Chain Monte Carlo methods to sample from a given density $\pi^{\varepsilon}(\theta)\propto e^{-U(\theta)/\varepsilon}$ of interest
where $\theta\in\mathbb{R}^{d}$.
Langevin algorithms are widely used in \textit{Bayesian learning} problems, such as Bayesian formulations of inverse problems, and Bayesian classification and regression tasks in machine learning \cite{gelman1995bayesian,stuart2010inverse,andrieu2003introduction,teh2016consistency,DistMCMC19,GIWZ2024}.
Langevin algorithms have also been used for solving non-convex optimization problems that arise in machine learning \cite{Raginsky,xu2018global,Chau2019,GGZ,Chau2022,Zhang2019}.
The classical Langevin algorithm is based on the discretization of
the \textit{overdamped Langevin diffusion} \cite{Dalalyan,DM2017,DK2017,Raginsky,Barkhagen2021,Chau2019,EH2021,Zhang2019,BCESZ2022}:
\begin{equation}\label{eq:overdamped-2}
d\theta_{t}=-\nabla U(\theta_{t})dt+\sqrt{2\varepsilon}dB_{t},
\end{equation}
where $U:\mathbb{R}^{d}\rightarrow\mathbb{R}$ is often known as the \textit{potential function}, $\varepsilon>0$ is a scaling parameter, 
and $B_{t}$ is a standard $d$-dimensional Brownian motion with $\theta_{0}\in\mathbb{R}^{d}$. Under some mild assumptions on $U(\cdot)$, the diffusion \eqref{eq:overdamped-2} admits a unique stationary distribution with the density $\pi^{\varepsilon}(\theta)\propto e^{-U(\theta)/\varepsilon}$,
also known as the \emph{Gibbs distribution} \cite{chiang1987diffusion,stroock-langevin-spectrum}. 
In practice, \eqref{eq:overdamped-2} is implemented through its discretizations and one of the most commonly used is the Euler–Maruyama discretization of
\eqref{eq:overdamped-2}, often known as the \textit{unadjusted Langevin algorithm} (ULA) in the literature \cite{DM2017,DM2016}.
When the full gradient is replaced
by a stochastic gradient, the Euler–Maruyama discretization is known 
as the \textit{stochastic gradient Langevin dynamics} (SGLD) \cite{welling2011bayesian,Raginsky}.
When the scaling parameter $\varepsilon>0$ is small, the Gibbs distribution $\pi^{\varepsilon}(\theta)\propto e^{-U(\theta)/\varepsilon}$ will concentrate
around the global minimizer of $U$ \cite{Raginsky}, which is why Langevin algorithms have also been widely used 
to obtain global convergence guarantees for solving non-convex optimization problems that often arise in machine learning \cite{Raginsky,xu2018global,Chau2019,Zhang2019}.

In the literature, many variants
of the overdamped Langevin diffusion and the discretization schemes have been studied.
One popular Langevin dynamics is the {\it underdamped Langevin diffusion}
\cite{mattingly2002ergodicity,Villani2009,cheng2018underdamped,cheng-nonconvex,CLW2020,JianfengLu,dalalyan2018kinetic,Ma2019,GGZ}:
\begin{equation}\label{eqn:underdamped}
\begin{cases}
dr_{t}=-\gamma r_{t}dt-\nabla U(\theta_{t})dt+\sqrt{2\gamma\varepsilon}dB_{t},\\
d\theta_{t}=r_{t}dt,
\end{cases}
\end{equation}
where $B_{t}$ is a standard $d$-dimensional Brownian motion
with $\gamma>0$ being the friction coefficient and $\varepsilon>0$ being a scaling parameter.
Under some mild assumptions on $U$, the diffusion \eqref{eqn:underdamped} admits a unique stationary distribution with the density $\pi^{\varepsilon}(\theta,r) \propto e^{-(U(\theta)+\frac{1}{2}|r|^{2})/\varepsilon}$ \cite{Eberle}, whose $\theta$-marginal distribution coincides 
with the stationary distribution of \eqref{eq:overdamped-2}.
It is known that the underdamped Langevin diffusion \eqref{eqn:underdamped} might converge to the Gibbs distribution faster than the overdamped Langevin diffusion \cite{Eberle,JianfengLu}.
Various discretizations based on the underdamped Langevin diffusion have better iteration complexity in terms
of the dependence on the dimension
and the accuracy level \cite{cheng2018underdamped,GGZ}.
Underdamped Langevin samplers, viewed as lifted MCMC methods, have also been analyzed
through splitting schemes and Wasserstein contraction for the resulting
discrete chains \cite{Monmarche2021Splitting}.  
When the full gradient is replaced by a stochastic gradient, the discretization of \eqref{eqn:underdamped} is known as the \textit{stochastic gradient Hamiltonian Monte Carlo} (SGHMC) \cite{emilyfox-sghmc,carin-2015-langevin-integrators,GGZ}.
When $\varepsilon>0$, the $\theta$-marginal of the Gibbs distribution concentrates around the global minimizer of $U$ \cite{Raginsky,GGZ} and hence SGHMC 
has been used to obtain global convergence guarantees for solving the non-convex optimizations in the literature \cite{GGZ,Chau2022}. It was first shown in \cite{GGZ} that SGHMC can outperform SGLD in the context of non-convex optimization, and hence momentum-based acceleration is achievable. 

In this paper, we study the diffusion process that belongs to the family of high-order Langevin
dynamics introduced for accelerated sampling algorithms
\cite{MouMaWainwrightBartlettJordan2021}. We consider the \textit{third-order Langevin diffusion}
$Z_t=(\theta_t,p_t,r_t)\in\R^{3d}$ \cite{MouMaWainwrightBartlettJordan2021}:
\begin{equation}\label{eq:sde}
\begin{cases}
    d\theta_t=p_t dt,\\[1mm]
    dp_t=-L^{-1}\nabla U(\theta_t) dt+\gamma r_t dt,\\[1mm]
    dr_t=-\gamma p_t dt-\gamma r_t dt+
    \sqrt{2\gamma\eps/L} dB_t,
\end{cases}
\end{equation}
where $\gamma>0$ is the friction coefficient, $L>0$ is a smoothness parameter, $\eps>0$ is a scaling parameter, $B_t$ is a standard Brownian motion
in $\R^d$.  The general third-order family of
\cite{MouMaWainwrightBartlettJordan2021} allows separate coupling and
dissipation parameters.  Following the one-parameter convention for
higher-order Langevin dynamics in
\cite{DangGurbuzbalabanIslamYaoZhu2025}, we set them equal and denote their
common value by $\gamma$.  More explicitly, the coupling parameter
$\gamma$ and damping parameter $\xi$ in the notation of
\cite{MouMaWainwrightBartlettJordan2021} are both equal to the present
$\gamma$ as in \cite{DangGurbuzbalabanIslamYaoZhu2025}.  The infinitesimal generator of \eqref{eq:sde} is
given by
\begin{equation}\label{eq:generator}
    \calL_\eps f
    =p\cdot\nabla_\theta f
    +\left(-\frac1L\nabla U(\theta)+\gamma r\right)\cdot\nabla_p f 
    +(-\gamma p-\gamma r)\cdot\nabla_r f
    +\frac{\gamma\eps}{L}\Delta_r f,
\end{equation}
and its Hamiltonian is
\begin{equation}\label{eq:H}
    H(\theta,p,r):=U(\theta)+\frac L2|p|^2+\frac L2|r|^2,
\end{equation}
and the invariant density of \eqref{eq:sde} is given by
\begin{equation}\label{eq:pi}
    \pi^\eps(dz)=Z_\eps^{-1}e^{-H(z)/\eps} dz,
    \qquad z=(\theta,p,r)\in\R^{3d}.
\end{equation}
Introduce
\begin{equation}\label{eq:DQ}
\D:=
\begin{pmatrix}
0&0&0\\
0&0&0\\
0&0&\dfrac{\gamma}{L}I_{d}
\end{pmatrix},
\qquad
\Q:=
\begin{pmatrix}
0&\dfrac1L I_{d}&0\\
-\dfrac1L I_{d}&0&\dfrac{\gamma}{L}I_{d}\\
0&-\dfrac{\gamma}{L}I_{d}&0
\end{pmatrix}.
\end{equation}
Then $\D=\D^\top\succeq 0$, $\Q^\top=-\Q$, and we can re-write the infinitesimal generator \eqref{eq:generator} as
\begin{equation}\label{eq:div-form}
    \calL_\eps f
    =
    \eps e^{H/\eps}\nabla\cdot
    \left(e^{-H/\eps}(\D+\Q)\nabla f\right),
\end{equation}
and its adjoint in $L^2(\pi^\eps)$ is given by
\begin{equation}\label{eq:adjoint}
    \calL_\eps^* f
    =
    \eps e^{H/\eps}\nabla\cdot
    \left(e^{-H/\eps}(\D-\Q)\nabla f\right).
\end{equation}

The third-order Langevin diffusion \eqref{eq:sde} has recently been extended to general
arbitrary order Langevin Monte Carlo schemes
\cite{DangGurbuzbalabanIslamYaoZhu2025,high-order-Liu-2025}.  
Related generalized Langevin
diffusions have been studied through almost-sure and Wasserstein contraction
methods \cite{Monmarche2023AlmostSure}.  
We study fixed-temperature weak equilibrium measures and capacity identities
for the degenerate continuous-time process.

Our approach is a fixed-temperature weak potential-theoretic construction.
Let $A$ and $B$ be two disjoint phase-space neighborhoods of the metastable
states, and let $D$ be a bounded domain containing
$\overline A\cup\overline B$.  Set
\[
    C:=B\cup\partial D,
    \qquad
    \tau_E:=\inf\{t\geq 0:Z_t\in E\}.
\]

In the third-order Langevin diffusion \eqref{eq:sde}, the Brownian noise is $d$-dimensional while the phase space is
$3d$-dimensional.  It acts directly only in the $r$-coordinate and reaches
$p$ and $\theta$ through the chain
\[
    r\longrightarrow p\longrightarrow \theta .
\]
This degeneracy motivates a path-space construction in place of a classical
normal flux through the boundary $\partial A$ of the phase-space neighborhood
$A$.  Characteristic points lie beyond the direct scope of the standard
uniformly elliptic boundary-flux theory.  The propagation of
noise through the chain is an instance of H\"ormander hypoellipticity
\cite{Hormander1967}.  Related metastability and annealing problems for kinetic
Langevin diffusions exhibit the analytical
difficulties caused by non-reversibility and degenerate noise
\cite{Monmarche2018}; in capacity-based approaches to underdamped Langevin diffusion
metastability, these features lead to nontrivial boundary issues
\cite{LeeRamilSeo2026}.
For elliptic irreversible or non-reversible diffusions, sharp small-noise
transition and exit asymptotics, including Eyring--Kramers prefactors and
principal-eigenvalue/mean-exit relations, have been developed in
\cite{BouchetReygner2016,LePeutrecMichelNectoux2024}.

In the motivating double-well picture (Figure~\ref{fig:bounded-domain-geometry}), the
corresponding phase-space centers are $(m,0,0)$ and $(s,0,0)$, where $m$ and
$s$ are local minima of $U$, and $\sigma$ is a saddle connecting their basins;
Figure~\ref{fig:bounded-domain-geometry}
distinguishes the potential wells from their phase-space neighborhoods.
\begin{figure}[H]
    \centering
    \begin{tikzpicture}[x=1cm,y=1cm,>=Latex]
        \begin{scope}[shift={(-4.1,-0.85)},x=1.05cm,y=1cm]
            \fill[gray!10]
                (-2.35,-0.05) --
                plot[smooth,tension=0.7] coordinates {
                    (-2.35,2.35) (-1.85,0.95) (-1.25,0.20)
                    (-0.65,0.65) (0,1.50) (0.65,0.55)
                    (1.25,0.08) (1.85,0.85) (2.35,2.25)
                }
                -- (2.35,-0.05) -- cycle;
            \draw[very thick]
                plot[smooth,tension=0.7] coordinates {
                    (-2.35,2.35) (-1.85,0.95) (-1.25,0.20)
                    (-0.65,0.65) (0,1.50) (0.65,0.55)
                    (1.25,0.08) (1.85,0.85) (2.35,2.25)
                };
            \draw[-{Latex[length=1.8mm]},thin]
                (-2.55,-0.05) -- (2.55,-0.05) node[right] {$\theta$};
            \draw[-{Latex[length=1.8mm]},thin]
                (-2.55,-0.05) -- (-2.55,2.65) node[above] {$U(\theta)$};

            \fill[blue!65!black] (-1.25,0.20) circle (2pt);
            \fill (0,1.50) circle (2pt);
            \fill[red!70!black] (1.25,0.08) circle (2pt);
            \node[blue!65!black,below=2pt] at (-1.25,0.20) {$m$};
            \node[above=2pt] at (0,1.50) {$\sigma$};
            \node[red!70!black,below=2pt] at (1.25,0.08) {$s$};

            \node[font=\scriptsize,align=center] at (0,-0.92)
                {$m,s$: local minima of $U$;\quad
                 $\sigma$: a connecting saddle};
        \end{scope}

        \begin{scope}[shift={(3.45,0)},x=0.95cm,y=0.95cm]
            \shade[inner color=white,outer color=gray!22,opacity=0.65]
                (0,0) ellipse (3.2 and 1.8);

            \draw[red!40!black,dashed]
                (3.2,0) arc[start angle=0,end angle=180,
                            x radius=3.2,y radius=0.47];
            \draw[red!55!black]
                (-3.2,0) arc[start angle=180,end angle=360,
                             x radius=3.2,y radius=0.47];
            \draw[red!40!black,dashed]
                (0,1.8) arc[start angle=90,end angle=270,
                            x radius=1.0,y radius=1.8];
            \draw[red!55!black]
                (0,-1.8) arc[start angle=270,end angle=450,
                             x radius=1.0,y radius=1.8];

            \shade[ball color=blue!28,opacity=0.9]
                (-1.45,0.12) ellipse (0.72 and 0.58);
            \draw[blue!65!black,thick]
                (-1.45,0.12) ellipse (0.72 and 0.58);
            \shade[ball color=red!30,opacity=0.9]
                (1.55,0.34) ellipse (0.74 and 0.60);
            \draw[red!70!black,thick]
                (1.55,0.34) ellipse (0.74 and 0.60);

            \node[blue!65!black] at (-1.45,0.31) {$A$};
            \node[blue!65!black,font=\scriptsize]
                at (-1.45,-0.12) {$(m,0,0)$};
            \node[red!70!black] at (1.55,0.54) {$B\subset C$};
            \node[red!70!black,font=\scriptsize]
                at (1.55,0.10) {$(s,0,0)$};
            \node[red!70!black,fill=white,fill opacity=0.8,
                  text opacity=1,inner sep=1.2pt]
                at (0,1.83) {$\partial D\subset C$};
            \node at (2.62,-1.02) {$D$};

            \fill (0,-0.16) circle (1.7pt);
            \node[above right=-1pt] at (0,-0.16) {$z$};
            \draw[-{Latex[length=2mm]},blue!65!black,thick]
                (-0.08,-0.11) .. controls (-0.48,0.48) and (-0.8,0.18)
                .. (-0.75,0.16);
            \draw[-{Latex[length=2mm]},red!70!black,thick,dashed]
                (0.12,-0.08) .. controls (0.55,0.7) and (0.85,0.42)
                .. (0.82,0.38);
            \draw[-{Latex[length=2mm]},red!70!black,thick,dashed]
                (0.05,-0.3) .. controls (0.35,-0.9) and (0.75,-1.3)
                .. (1.08,-1.63);

            \coordinate (O) at (-2.55,-1.32);
            \draw[-{Latex[length=1.5mm]},thin] (O) -- ++(0.72,0)
                node[right,font=\scriptsize] {$\theta$};
            \draw[-{Latex[length=1.5mm]},thin] (O) -- ++(-0.38,0.24)
                node[above left=-1pt,font=\scriptsize] {$p$};
            \draw[-{Latex[length=1.5mm]},thin] (O) -- ++(0,0.65)
                node[above,font=\scriptsize] {$r$};

            \draw[very thick,red!70!black]
                (0,0) ellipse (3.2 and 1.8);
        \end{scope}

        \node[font=\small] at (-4.1,-2.18)
            {(a) Motivating double-well landscape};
        \node[font=\small] at (3.45,-2.18)
            {(b) Phase-space neighborhoods};
    \end{tikzpicture}
    \caption{Relation between the motivating energy landscape and the sets in
    the bounded-domain construction.  Panel (a) shows a double-well example;
    the bounded-domain theory allows geometries beyond a unique-saddle Morse
    structure.  Panel (b) shows that $A$ and $B$ are phase-space neighborhoods
    centered at $(m,0,0)$ and $(s,0,0)$, rather than the wells themselves, and
    that the competing target is $C=B\cup\partial D$.
    }
    \label{fig:bounded-domain-geometry}
\end{figure}
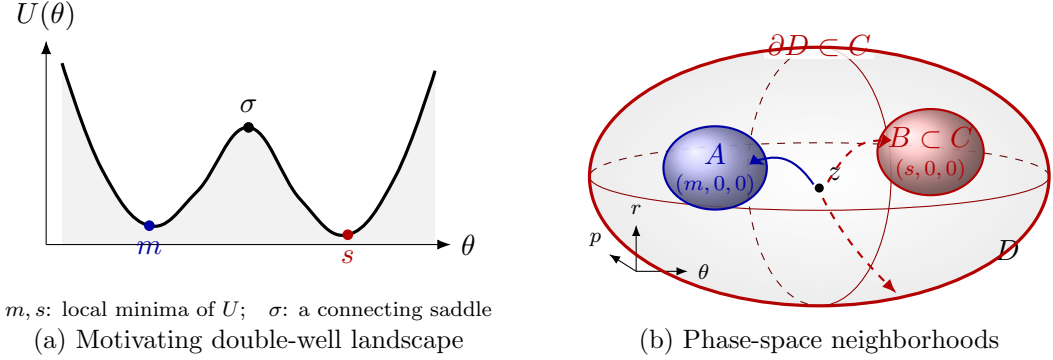
Rather than using a pointwise boundary-normal or conormal flux
representation for the hypoelliptic committor, we begin with the
path-space hitting probabilities
\[
    h(z):=\Pbb_z(\tau_A<\tau_C),
    \qquad
    h^*(z):=\Pbb_z^*(\tau_A<\tau_C),
\]
where $\Pbb_z$ and $\Pbb_z^*$ denote the laws of the forward and
adjoint processes, respectively, starting from $z$.
The equilibrium measure and the capacity are introduced only after these
hitting laws are shown to be stable under the auxiliary regularization.
The parameter $\delta$ serves as a technical device that supplies classical
identities at fixed $\delta>0$; probabilistic stability then transfers these
identities to the degenerate process.  The characteristic set of a phase-space
sphere lies in $\{r=0\}$.  We adapt the boundary analysis in
\cite{LeeRamilSeo2026} to show that boundary hits occur with
non-zero incoming $r$ with high probability and that small-$r$ boundary entries
have uniformly vanishing probability.
Compared with the underdamped case in \cite{LeeRamilSeo2026}, the third-order
extension requires more than replacing velocity by the highest auxiliary
variable: the noise reaches the physical coordinate only after the two
deterministic links $r\to p$ and $p\to\theta$, so the small-boundary-entry
estimate has to be coupled with a third-order density scale and with the
$\delta\downarrow0$ regularization limit.

The density condition needed for this boundary analysis is local in space.  We
replace the polynomial-growth dynamics, on a bounded neighborhood of the domain
and its boundary collars, by a chain-compatible cutoff system.  After the
reordering
\[
    X^1=r,\qquad X^2=p,\qquad X^3=\theta,
\]
the limiting cutoff dynamics form a three-layer chain of the type treated in
\cite{Pigato2022}.  The corresponding density estimates provide the required
local upper bound, and the additional
$\theta$- and $p$-noises in the elliptic regularizations contribute
nonnegative Malliavin covariance. Thus, the bracket chain
$r\to p\to\theta$ already present at $\delta=0$ controls the regularized
family.  This yields the uniform hitting-time stability needed for the weak
potential-theoretic limit.

In the reversible potential-theoretic approach to metastability, equilibrium
potentials, equilibrium measures, and capacities are the basic objects, and
last-exit averaged mean hitting times are related to capacities by the classical
capacity formula \cite[Corollary~7.30]{BovierDenHollander2015}. 
Non-self-adjoint elliptic diffusions admit
Dirichlet--Thomson variational principles for capacity
\cite{LandimMarianiSeo2019}.  In the present hypoelliptic setting, the
additional degeneracy motivates the weak equilibrium measure constructed below.

The auxiliary regularized family is obtained by adding small symmetric
gradient-type diffusions in the $\theta$- and $p$-coordinates.  
For $\varsigma\in\{+1,-1\}$ and $0\le\delta\le1$, define the signed
regularized diffusion $Z_t^{\delta,\varsigma}
=\left(\theta_t^{\delta,\varsigma},p_t^{\delta,\varsigma},r_t^{\delta,\varsigma}\right)$
by
\begin{equation}\label{eq:sde-delta}
\begin{cases}
    d\theta_t^{\delta,\varsigma}
    =\varsigma p_t^{\delta,\varsigma}dt
    -\delta\nabla U\left(\theta_t^{\delta,\varsigma}\right)dt
    +\sqrt{2\eps\delta}\,dW_t^\theta,\\[1mm]
    dp_t^{\delta,\varsigma}
    =-\varsigma L^{-1}\nabla U\left(\theta_t^{\delta,\varsigma}\right)dt
    +\varsigma\gamma r_t^{\delta,\varsigma}dt
    -\delta Lp_t^{\delta,\varsigma}dt
    +\sqrt{2\eps\delta}\,dW_t^p,\\[1mm]
    dr_t^{\delta,\varsigma}
    =-\varsigma\gamma p_t^{\delta,\varsigma}dt
    -\gamma r_t^{\delta,\varsigma}dt
    +\sqrt{2\gamma\eps/L}\,dB_t,
\end{cases}
\end{equation}
where $W^\theta,W^p,B$ are independent standard $d$-dimensional Brownian
motions.  We write $Z^\delta:=Z^{\delta,+1}$ for the forward process and
$Z^{\delta,*}:=Z^{\delta,-1}$ for the adjoint process.  At $\delta=0$ these
are the limiting degenerate forward and adjoint diffusions.
These additional
second-order parts are reversible with respect to the Gibbs weight when viewed
alone, but the full regularized dynamics remain non-reversible because the
antisymmetric transport matrix $\Q$ is unchanged.  This family supplies the
fixed-$\delta$ classical identities used in the probabilistic stability
argument as $\delta\downarrow0$.
For $0<\delta\le1$, the infinitesimal generator of the forward process in
\eqref{eq:sde-delta} is
\begin{equation}\label{eq:Ldelta}
\begin{aligned}
    \calL_{\eps,\delta}
    &:=\calL_\eps
    +\delta(\eps\Delta_\theta-\nabla U(\theta)\cdot\nabla_\theta)
    +\delta(\eps\Delta_p-Lp\cdot\nabla_p).
\end{aligned}
\end{equation}
and the adjoint process in \eqref{eq:sde-delta} has generator
$\calL_{\eps,\delta}^*$ in $L^2(\pi^\eps)$.  We use the shorthand
\begin{align}\label{defn:L:pm}
\calL_{\eps,\delta}^{(+)}:=\calL_{\eps,\delta},
    \qquad
    \calL_{\eps,\delta}^{(-)}:=\calL_{\eps,\delta}^*.
\end{align}
At the endpoint $\delta=0$, this notation means
$\calL_{\eps,0}^{(+)}=\calL_\eps$ and
$\calL_{\eps,0}^{(-)}=\calL_\eps^*$.
Equivalently,
\begin{equation}\label{eq:Ldelta-div}
    \calL_{\eps,\delta}f
    =
    \eps e^{H/\eps}\nabla\cdot
    \left(e^{-H/\eps}(\D_\delta+\Q)\nabla f\right),
\end{equation}
where
\begin{equation}\label{eq:Ddelta}
    \D_\delta:=
    \begin{pmatrix}
        \delta I_{d}&0&0\\
        0&\delta I_{d}&0\\
        0&0&\dfrac{\gamma}{L}I_{d}
    \end{pmatrix}.
\end{equation}
For fixed $0<\delta\le1$, the operator is uniformly elliptic on compact sets and
has the same infinitesimally invariant measure $\pi^\eps$. Indeed, for every
$f\in C_c^\infty(\R^{3d})$, the divergence form \eqref{eq:Ldelta-div}
implies
\[
    \int_{\R^{3d}}\calL_{\eps,\delta}fd\pi^\eps
    =
    \eps Z_\eps^{-1}\int_{\R^{3d}}
    \nabla\cdot\left(e^{-H/\eps}(\D_\delta+\Q)\nabla f\right)dz=0.
\]
Thus, $\pi^\eps$ is infinitesimally invariant for
$\calL_{\eps,\delta}$ on $C_c^\infty(\R^{3d})$.  Under the non-explosion and
well-posedness assumptions used below, the Echeverr\'ia invariant-measure
criterion \cite{Echeverria1982} promotes this identity to invariance for the
corresponding semigroup. Its adjoint is obtained by replacing $\Q$ by $-\Q$.

The related literature can be viewed through three strands.
First, the potential-theoretic approach to metastability, developed systematically in
\cite{BovierDenHollander2015}, expresses transition-time asymptotics through
equilibrium potentials, equilibrium measures, and capacities, and it provides a
standard route to Eyring--Kramers laws for reversible overdamped diffusions.
For irreversible elliptic diffusions, the normal-flux picture is supplemented
by non-self-adjoint capacity identities and
Dirichlet--Thomson variational principles \cite{LandimMarianiSeo2019}, and
sharp exit asymptotics are known in several small-noise regimes
\cite{BouchetReygner2016,LePeutrecMichelNectoux2024}.

Second, for kinetic and hypoelliptic Langevin diffusions, degeneracy of the noise
and non-reversibility create additional boundary and regularity issues.  We
follow the methodology in Lee--Ramil--Seo~\cite{LeeRamilSeo2026}, where a weak
equilibrium measure handles the characteristic boundary points lying beyond
the standard uniformly elliptic normal-flux theory.  Our
adaptation replaces their two-layer small-momentum analysis by a small-$r$
boundary-stability argument for the third-order chain $r\to p\to\theta$.

Third, high-order Langevin Monte Carlo algorithms are designed to improve sampling
behavior, while rare transitions and metastable exits remain governed by
phase-space potential theory rather than by mixing estimates alone.  The
third-order chain forces us to combine the weak-capacity viewpoint with
localized chain-type heat-kernel estimates in the sense of \cite{Pigato2022}.
A weak capacity theory for the third-order Langevin diffusion therefore gives a
fixed-temperature entrance law and capacity--hitting identity that can serve as
a basis for later low-temperature Eyring--Kramers asymptotics and for comparing
metastable transition mechanisms across overdamped, underdamped, and higher-order
Langevin samplers.

The contributions of this paper can be summarized as follows.
\begin{enumerate}[label=(\roman*)]
    \item We construct weak equilibrium measures and weak capacities for the
    hypoelliptic third-order Langevin diffusion in bounded phase-space ball
    geometries from path-space committors and hitting identities.
    \item We prove a regularization-uniform boundary-stability theorem for the
    third-order chain $r\to p\to\theta$, combining anisotropic density estimates
    with uniform small-$r$ boundary-entry bounds at the characteristic parts of
    phase-space spheres.
    \item We identify the limiting weak equilibrium measure through the adjoint
    committor and prove the corresponding capacity--hitting identity, together
    with a normalized version.
\end{enumerate}

The main technical novelty is a boundary-stability theory that is uniform
along the elliptic regularization for the third-order chain
$r\longrightarrow p\longrightarrow\theta$.
In the underdamped two-layer setting of \cite{LeeRamilSeo2026}, the directly
forced momentum reaches the physical coordinate after one deterministic link.
Here, the boundary-normal physical coordinate is two links away from the
$r$-noise, producing the anisotropic scales $t^{1/2}$, $t^{3/2}$, and
$t^{5/2}$.  Consequently, the small-boundary-entry argument requires a
three-layer density estimate whose constants remain uniform as the auxiliary
$p$- and $\theta$-noises vanish with $\delta\downarrow0$.  We obtain this
estimate through chain-compatible localization and uniform inverse-Malliavin-
covariance bounds.  The passage of the fixed-$\delta$ Green identities to the
degenerate process is carried out through pathwise hitting-law stability and
small-$r$ non-grazing estimates.

The rest of the paper is organized around these points: the next section (Section~\ref{sec:main}) states
the main assumptions and stability result, where Section~\ref{sec:small} proves the
small-$r$ boundary stability theorem, 
Section~\ref{sec:weak} constructs the
weak measure, capacity, and hitting identities, and Section~\ref{sec:passage} 
passes to the whole-space identity. Finally, we conclude in Section~\ref{sec:conclusion}. 
In Appendix~\ref{app:uniform-perturbation}, we provide uniform perturbation estimates for the cutoff covariance.

\paragraph{Notation.}
We write $z=(\theta,p,r)\in\R^{3d}$, where each component belongs to
$\R^d$; in arguments using the chain order, we set
$(X^1,X^2,X^3)=(r,p,\theta)$.  The open Euclidean ball with center $x$ and
radius $\rho$ is denoted by $B(x,\rho)$.  For a set $E$, we write
$\overline E$, $\partial E$, $E^c$, and $\bfone_E$ for its closure, boundary,
complement, and indicator, respectively, and use
$\dist(x,E)$ and $\dist(E,F)$ for Euclidean distances.  The symbol $|\cdot|$
denotes the Euclidean norm, while $\|\cdot\|$ denotes the operator norm for
matrices; $\|\cdot\|_F$ denotes the Frobenius norm, $M^\top$ the transpose of
$M$, and $a\wedge b=\min\{a,b\}$.  For a continuous process $Y$ and a Borel
set $E$, let
\[
    \tau_E(Y):=\inf\{t\ge0:Y_t\in E\},
    \qquad \inf\varnothing:=\infty,
\]
and suppress $Y$ when it is clear from the probability law; a superscript
$\delta$ indicates the regularized process.  The symbols
$\Pbb_z^\delta,\E_z^\delta$ and
$\Pbb_z^{\delta,*},\E_z^{\delta,*}$ denote probability and expectation for
the forward and adjoint regularized processes started at $z$, respectively.
At $\delta=0$ we use $\Pbb_z,\E_z$ and $\Pbb_z^*,\E_z^*$.
We use $C_c^\infty(G)$ for smooth compactly supported functions on $G$ and
$C_b^k(G)$ for $C^k$ functions whose derivatives through order $k$ are
bounded; $\supp f$ and $\supp\mu$ denote the supports of a function $f$ and
a measure $\mu$, respectively.  Weak
convergence of finite Borel measures is denoted by
$\mu_n\Rightarrow\mu$.  Constants denoted by $C,c$, with or without
subscripts, are finite and positive and may change from line to line; any
required uniformity is stated explicitly.

\section{Main Results}\label{sec:main}

\begin{assumption}[Potential]\label{ass:U}
The potential $U\in C^\infty(\R^d)$ is bounded from below.  After adding a
constant, assume $U\ge0$.  The sublevel sets of $U$ are compact.  Moreover,
there exist constants $c_0>0$ and $R_0>0$ such that
\begin{equation}\label{eq:radial-growth}
    \theta\cdot\nabla U(\theta)
    \ge c_0\left(|\theta|^2+U(\theta)\right),
    \qquad |\theta|\ge R_0.
\end{equation}
\end{assumption}

\begin{assumption}[Controlled derivatives and local moments]
\label{ass:controlled-coefficients}
We assume the following.
\begin{enumerate}[label=(\alph*)]
\item \emph{Controlled derivatives.}  The force and all its derivatives have at
most polynomial growth: for every multi-index $\alpha$ with
$|\alpha|\ge0$, there exist constants $C_\alpha,m_\alpha<\infty$ such that
\[
    |\partial^\alpha \nabla U(\theta)|
    \le C_\alpha(1+|\theta|^{m_\alpha}),
    \qquad \theta\in\R^d .
\]
\item \emph{Uniform finite-time well-posedness and moments.}  For every
$0\le\delta\le1$, the forward and adjoint stochastic differential equations
in \eqref{eq:sde-delta} admit unique non-explosive strong solutions.  In this
assumption, only $0<\delta\le1$ is used for the
regularized processes, while the endpoint $\delta=0$ denotes the unregularized
degenerate process. Moreover, for every $T<\infty$, every $q\ge1$, and
every compact
$K\subset\R^{3d}$,
the forward processes satisfy
\[
    \sup_{0\le\delta\le1}\sup_{z\in K}
    \E_z^\delta\left[
        \sup_{0\le t\le T}\left(1+|Z_t^\delta|\right)^q
    \right]<\infty,
\]
and the adjoint processes satisfy
\[
    \sup_{0\le\delta\le1}\sup_{z\in K}
    \E_z^{\delta,*}\left[
        \sup_{0\le t\le T}\left(1+|Z_t^{\delta,*}|\right)^q
    \right]<\infty .
\]
Thus the uniform bounds include both the regularized family
$0<\delta\le1$ and its limiting degenerate endpoint $\delta=0$.
\end{enumerate}
\end{assumption}

\begin{remark}[Role of Assumption~\ref{ass:controlled-coefficients}]
Part (a) is a structural condition on the potential.  Part (b) is imposed for
the general class considered in the main theorems; Proposition~\ref{prop:basic}
proves non-explosion of the unregularized forward process from
Assumption~\ref{ass:U}.  For polynomial double-well examples, the remaining
assertions in part (b) follow from the Hamiltonian estimates
\eqref{eq:polynomial-H-drift}--\eqref{eq:polynomial-moment-bound} in
Remark~\ref{rem:polynomial-double-well-moments}.  The boundary-layer argument
uses localized heat-kernel estimates for a chain-compatible cutoff process that
agrees with the original polynomial-growth dynamics on the bounded region
visited in the boundary argument.
\end{remark}

\begin{remark}[Polynomial double-well examples]
\label{rem:polynomial-double-well-moments}
For example, in one dimension,
\[
    U(x)=\frac14(x^2-1)^2
\]
satisfies the growth and coercivity parts of
Assumptions~\ref{ass:U} and \ref{ass:controlled-coefficients}.  It is
non-negative, has compact sublevel sets, and all derivatives of its force have
polynomial growth.  Moreover, $U'(x)=x(x^2-1)$ and
$x^2+U(x)=\frac14(x^2+1)^2$.
Hence, for $|x|\ge\sqrt2$,
\[
    xU'(x)=x^2(x^2-1)
    \ge \frac12\left(x^2+U(x)\right),
\]
which verifies the radial growth condition in Assumption~\ref{ass:U}, for
instance with $c_0=1/2$ and $R_0=\sqrt2$.  The finite-time local moment bounds
can be checked by applying the regularized forward and adjoint generators
$\calL_{\eps,\delta}^{(\pm)}$, defined in \eqref{defn:L:pm}, to
$V_m=(1+H)^m$, where
$H=U+L|p|^2/2+L|r|^2/2$.  The antisymmetric Hamiltonian part annihilates $H$, so
the forward and adjoint computations give the same identity.  For the quartic
example above, $d=1$.  We retain the dimension $d$ in the following formula
(because the same computation applies to multidimensional coercive polynomial
potentials; in the one-dimensional case, $\Delta U=U''$,
$\gamma d\eps=\gamma\eps$, and $\delta\eps Ld=\delta\eps L$):
\begin{equation}\label{eq:polynomial-H-drift}
    \calL_{\eps,\delta}^{(\pm)}H
    =
    -\gamma L|r|^2
    -\delta|\nabla U(\theta)|^2
    -\delta L^2|p|^2
    +\gamma d\eps
    +\delta\eps\Delta U(\theta)
    +\delta\eps Ld .
\end{equation}
For the polynomial double wells considered here,
$|\Delta U(\theta)|\le C(1+U(\theta))\le C(1+H)$.  Therefore the trace term in
$\calL_{\eps,\delta}^{(\pm)}V_m$ is bounded by $C_mV_m$, and, for each $m\ge1$,
\begin{equation}\label{eq:polynomial-Vm-drift}
    \sup_{0\le\delta\le1}\calL_{\eps,\delta}^{(\pm)}V_m\le C_mV_m .
\end{equation}
The corresponding carr\'e-du-champ terms satisfy the polynomial bounds
needed in the Burkholder-Davis-Gundy inequality.  It\^o's formula, localization, Burkholder-Davis-Gundy inequality, and
Gr\"onwall's lemma yield
\begin{equation}\label{eq:polynomial-moment-bound}
    \sup_{\substack{0\le\delta\le1\\z\in K}}
    \E_z^{\delta,\pm}\left[\sup_{0\le s\le T}V_m(Z_s^\delta)\right]
    \le C_{m,T,K}.
\end{equation}
Here the signs $+$ and $-$ refer to the forward and adjoint laws,
respectively.
Since $V_m$ dominates $(1+|\theta|+|p|+|r|)^q$ for $m$ sufficiently large, this
verifies Assumption~\ref{ass:controlled-coefficients}(b) for this example.
The same argument applies to coercive multidimensional polynomial double wells
with the corresponding radial growth.
\end{remark}

\begin{assumption}[Two metastable phase-space balls]\label{ass:two-balls}
Let $m,s\in\R^d$ be two distinguished local minima of $U$.  Fix $\rho>0$ such
that, with $B(x,r)$ denoting the open Euclidean ball, the phase-space balls
\[
    A=B((m,0,0),\rho),
    \qquad
    B=B((s,0,0),\rho)
\]
have disjoint closures.
\end{assumption}

\begin{assumption}[Whole-space growth]\label{ass:whole-space-growth}
For the whole-space exhaustion argument, assume in addition that
\begin{equation}\label{eq:whole-space-growth}
    \frac{|\nabla U(\theta)|^2}
    {(1+U(\theta))(1+\theta\cdot\nabla U(\theta))}
    \longrightarrow 0,
    \qquad |\theta|\to\infty .
\end{equation}
Equivalently, in view of \eqref{eq:radial-growth},
\begin{equation}\label{eq:whole-space-growth-equivalent}
    \frac{|\nabla U(\theta)|^2}{1+U(\theta)}
    =
    o\!\left(\theta\cdot\nabla U(\theta)\right),
    \qquad |\theta|\to\infty.
\end{equation}
\end{assumption}

For $R>0$, define
\begin{equation}\label{eq:whole-space-lyapunov-set}
    \Psi(\theta,p,r)
    :=U(\theta)+|p|^2+|r|^2+H(\theta,p,r)|r|^2,
    \qquad
    \mathsf K_R:=\{z\in\R^{3d}:\Psi(z)\le R\}.
\end{equation}
Assumption~\ref{ass:U} implies that every $\mathsf K_R$ is compact.

\begin{remark}[Polynomial potentials and the whole-space growth condition]\label{rem:quartic-lyapunov}
Condition~\eqref{eq:whole-space-growth} is the structural growth needed to
absorb the third-order chain term $r\cdot\nabla U(\theta)$ in the Lyapunov
estimate below.  It directly controls the gradient ratio in the form used by
that estimate.  It is enough for a coercive
polynomial potential to satisfy, outside a compact set,
\[
    U(\theta)\asymp |\theta|^m,\qquad
    \theta\cdot\nabla U(\theta)\asymp |\theta|^m,\qquad
    |\nabla U(\theta)|\lesssim |\theta|^{m-1},
\]
with $m\ge2$.  Then, as $|\theta|\rightarrow\infty$, 
\[
    \frac{|\nabla U(\theta)|^2}{1+U(\theta)}
    \lesssim |\theta|^{m-2}
    =
    o(|\theta|^m)
    =
    o\!\left(\theta\cdot\nabla U(\theta)\right),
\]
so that \eqref{eq:whole-space-growth} holds.  This covers standard even-degree
coercive double-well potentials, including separable multidimensional quartic
examples; the one-dimensional quartic double well is the case $m=4$.
\end{remark}

\begin{remark}[Comparison with \cite{LeeRamilSeo2026} assumptions]
The potential assumptions used for the bounded-domain weak-capacity identities
differ from those in Lee--Ramil--Seo~\cite[Assumptions~2.1
and~2.3]{LeeRamilSeo2026}.  The radial condition
\eqref{eq:radial-growth} is the analogue of the first growth condition in
\cite[Assumption~2.3]{LeeRamilSeo2026}.  Our bounded-domain setting allows
geometries beyond the Morse double-well structure with a unique saddle from
\cite[Assumption~2.1]{LeeRamilSeo2026}.  Polynomial derivative bounds, local
moment control, and chain-compatible localization take the place of the second
growth condition
\[
    \liminf_{|\theta|\to\infty}
    \left(|\nabla U(\theta)|-\beta\Delta U(\theta)\right)>0
\]
appearing in \cite[Assumption~2.3]{LeeRamilSeo2026}.  Thus the bounded-domain
construction is geometrically less restrictive than the low-temperature
Eyring--Kramers setting of \cite{LeeRamilSeo2026}.  The whole-space exhaustion
additionally uses Assumption~\ref{ass:whole-space-growth}, which supplies the
Lyapunov drift needed for the third-order chain.  The required accessibility is
established below by an explicit controlled-path argument.
\end{remark}

\paragraph{Bounded-domain setup.}\phantomsection\label{setup:domains}
Under Assumption~\ref{ass:two-balls}, fix $N>0$ sufficiently large that the compact set
$\overline A\cup\overline B$ satisfies
\[
    \overline A\cup\overline B\subset \calD:=B(0,N)\subset\R^{3d}.
\]

The three compact boundary components are pairwise disjoint.  Consequently,
\[
    d_\partial
    :=\min_{\substack{\Sigma,\Sigma'\in
        \{\partial A,\partial B,\partial\calD\}\\ \Sigma\ne\Sigma'}}
        \dist(\Sigma,\Sigma')>0.
\]
Fix $0<\eta_{\rm tub}<d_\partial/3$ and, for each boundary component, define
the distance collar
\[
    \mathsf T_{\eta_{\rm tub}}(\Sigma)
    :=\{z\in\R^{3d}:\dist(z,\Sigma)<\eta_{\rm tub}\}.
\]
The closures of these three collars are pairwise disjoint.

Set
\[
    \Omega:=\calD\setminus\left(\overline A\cup\overline B\right),
    \qquad
    \calC:=B\cup\partial\calD.
\]

\begin{remark}[Open-ball convention]\label{rem:open-ball-convention}
The hitting times $\tau_A$ and $\tau_B$ always refer to entry into the open
balls $A$ and $B$, as in the convention of Lee--Ramil--Seo.  The closures are
removed when defining the smooth PDE domain $\Omega$.  For every fixed
$\delta>0$, regular boundary points for the uniformly elliptic process identify
the open-set hitting time with the first contact time with the corresponding
sphere.  For the limiting process, the immediate-crossing and small-$r$
estimates proved below give the same identification in every hitting-law
argument used in the paper.  Values prescribed on $\partial A$ and $\partial B$
are boundary extensions of the committors.
\end{remark}

\begin{remark}[Component dimension]\label{rem:dimension}
Throughout the paper, $d\ge1$ denotes the dimension of each component variable
$\theta,p,r\in\mathbb R^d$.  The full phase space has dimension $3d$.
\end{remark}

The forward and adjoint regularized processes used below are the signed
diffusions in \eqref{eq:sde-delta}, with $\varsigma=+1$ and $\varsigma=-1$,
respectively.

For $z\in\Omega$, define the regularized and limiting committors by
\begin{equation}\label{eq:committors} h_\delta(z):=\Pbb_z^\delta\left(\tau_A^\delta<\tau_\calC^\delta\right),
    \qquad
    h_\delta^*(z):=\Pbb_z^{\delta,*}\left(\tau_A^\delta<\tau_\calC^\delta\right),
\end{equation}
and
\begin{equation}\label{eq:limiting-committors}
    h(z):=\Pbb_z(\tau_A<\tau_\calC),
    \qquad
    h^*(z):=\Pbb_z^*(\tau_A<\tau_\calC).
\end{equation}
They are extended by the boundary values $1$ on $\overline A$ and $0$ on
$\overline B\cup\partial\calD$.
For $M<\infty$, set
\begin{equation}\label{eq:gF-defs}
    g_{\delta,M}(z)=\E_z^\delta\left[\tau_\calC^\delta\wedge M\right],
    \qquad
    g_M(z)=\E_z\left[\tau_\calC\wedge M\right],
\end{equation}
and
\begin{equation}\label{eq:F-defs}
    F_{\delta,M}(z)=\Pbb_z^\delta\left(\tau_\calC^\delta\le M\right),
    \qquad
    F_M(z)=\Pbb_z(\tau_\calC\le M).
\end{equation}

The following theorem collects the regularization stability statements needed
to pass from the auxiliary regularized processes to the degenerate third-order
process in the bounded phase-space ball geometry.

\begin{theorem}[$\delta\downarrow0$ stability of hitting laws for phase-space balls]
\label{thm:ball-stability}
Suppose Assumptions~\ref{ass:U} and~\ref{ass:controlled-coefficients} hold,
and adopt the \hyperref[setup:domains]{bounded-domain setup}.  Then, for the
forward and adjoint processes in \eqref{eq:sde-delta}, the following holds.
\begin{enumerate}[label=(\roman*)]
    \item
    As $\delta\downarrow0$, 
    \[
        h_\delta^*\to h^*
        \qquad\text{in }L^1(\calD,\pi^\eps).
    \]

    \item 
     As $\delta\downarrow0$, 
    \[
        h_\delta^*F_{\delta,M}\to h^*F_M
        \qquad\text{in }L^1(\calD,\pi^\eps),
    \]
    where $F_{\delta,M},F_{M}$ are defined in \eqref{eq:F-defs} for $M<\infty$.

    \item 
     As $\delta\downarrow0$, 
    \[
        \sup_{z\in\partial A}|g_{\delta,M}(z)-g_M(z)|\to0,
    \]
    where $g_{\delta,M},g_{M}$ are defined in \eqref{eq:gF-defs} for $M<\infty$.
    Moreover $g_M\in C(\partial A)$.

    \item Let
    $\sigma^\delta=\tau_A^\delta\wedge\tau_\calC^\delta$ and
    $\sigma=\tau_A\wedge\tau_\calC$.  For $\pi^\eps$-almost every
    $z\in\Omega$ and every $M<\infty$,
    \[
        \sigma^\delta\wedge M
        \Longrightarrow
        \sigma\wedge M
    \]
    under $\Pbb_z$ as $\delta\downarrow0$, and the first-hit labels satisfy
    \[
        \bfone_{\{\tau_A^\delta<\tau_\calC^\delta\}}
        \Longrightarrow
        \bfone_{\{\tau_A<\tau_\calC\}} .
    \]
    Under the explicit SDE coupling of Theorem~\ref{thm:delta-coupling}, these
    convergences hold in
    probability at every starting point for which the limiting first killed hit
    is a non-characteristic crossing with a unique first-hit boundary component.

    \item The limiting killed process satisfies
    \[
        \Pbb_z(\tau_\calC<\infty)=1
        \qquad\text{for every }z\in\calD\setminus B.
    \]
\end{enumerate}
The analogous statements obtained by replacing each forward hitting functional
with its adjoint counterpart also hold.
\end{theorem}

\begin{remark}[Characteristic boundary points]\label{rem:ball-characteristic}
For the limiting diffusion, noise acts directly only in the $r$-variable.
Hence, for a phase-space sphere with defining function
\[
    \rho_m(\theta,p,r)=|\theta-m|^2+|p|^2+|r|^2-\rho^2,
\]
the diffusion-normal component is $\nabla_r\rho_m=2r$ and vanishes on
the subset $\{r=0\}$ of the boundary.  Theorem~\ref{thm:ball-stability}
is therefore proved through pathwise non-characteristic crossing and the
uniform small-$r$ boundary-entry estimate, rather than through uniform
elliptic boundary regularity as $\delta\downarrow0$.
\end{remark}

The next proposition gathers the basic analytic properties of the
unregularized generator that will be used throughout the weak potential theory.

\begin{proposition}[Basic properties]\label{prop:basic}
Under Assumption~\ref{ass:U}, the following statements hold.
\begin{enumerate}[label=(\roman*)]
    \item For every $\eps>0$,
    \[
        Z_\eps=\int_{\R^{3d}}e^{-H(z)/\eps} dz<\infty.
    \]
    Hence $\pi^\eps$ defined in \eqref{eq:pi} is a probability measure.

    \item The solution of \eqref{eq:sde} is non-explosive.

    \item For every $f\in C_c^\infty(\R^{3d})$,
    \[
        \int_{\R^{3d}}\calL_\eps f d\pi^\eps=0.
    \]

    \item The operator $\calL_\eps$ satisfies the H\"ormander bracket
    condition at every point of $\R^{3d}$.
\end{enumerate}
\end{proposition}

\begin{proof}
The $p$- and $r$-integrals in $Z_\eps$ are Gaussian, so that
\[
    Z_\eps=\left(\frac{2\pi\eps}{L}\right)^d
    \int_{\R^d}e^{-U(\theta)/\eps} d\theta .
\]
For $\theta=\rho\omega$, $\rho=|\theta|$, $|\omega|=1$, and $\rho\ge R_0$,
\eqref{eq:radial-growth} gives
\[
    \rho\frac{d}{d\rho}U(\rho\omega)
    =\rho \omega\cdot\nabla U(\rho\omega)
    \ge c_0\rho^2.
\]
Integrating along the ray yields $U(\rho\omega)\ge c\rho^2-C$ for suitable
constants $c,C>0$. Thus, $e^{-U/\eps}$ is dominated at infinity by a Gaussian
density, which proves (i).

For (ii), we can compute that
\begin{align*}
    \calL_\eps H
    =p\cdot\nabla U
    +\left(-\frac1L\nabla U+\gamma r\right)\cdot Lp
    +(-\gamma p-\gamma r)\cdot Lr
    +\frac{\gamma\eps}{L}\Delta_r\left(\frac L2|r|^2\right)
    =-\gamma L|r|^2+\gamma d\eps.
\end{align*}
In particular $\calL_\eps H\le\gamma d\eps$.  Since $H$ has compact sublevel
sets, let $\tau_n=\inf\{t\ge0:H(Z_t)\ge n\}$. It\^o's formula gives
\[
    \E_z \left[H(Z_{t\wedge\tau_n})\right]
    \le H(z)+\gamma d\eps t.
\]
On $\{\tau_n\le t\}$, $H(Z_{t\wedge\tau_n})\ge n$, and therefore
\[
    \Pbb_z(\tau_n\le t)
    \le \frac{H(z)+\gamma d\eps t}{n}.
\]
Letting $n\to\infty$ proves non-explosion.

For (iii), the divergence form \eqref{eq:div-form} gives
\[
    \int_{\mathbb{R}^{3d}} \calL_\eps f d\pi^\eps
    =\eps Z_\eps^{-1}\int_{\mathbb{R}^{3d}}
    \nabla\cdot\left(e^{-H/\eps}(\D+\Q)\nabla f\right)dz=0.
\]

For (iv), set $X_i=\partial_{r_i}$, $1\le i\le d$, and let $X_0$ be the
first-order drift field in \eqref{eq:generator}.  Then
\[
    [\partial_{r_i},X_0]
    =\gamma\partial_{p_i}-\gamma\partial_{r_i},
    \qquad
    [\partial_{p_i},X_0]
    =\partial_{\theta_i}-\gamma\partial_{r_i}.
\]
Thus, the diffusion directions generate the $r$-directions directly, the
$p$-directions after one commutator, and the $\theta$-directions after a
second commutator.
\end{proof}

The heat-kernel estimate used later is local, so we first build cutoff dynamics
that agree with the original dynamics on the relevant bounded set while
preserving the third-order chain structure.

\begin{lemma}[Chain-compatible cutoffs]
\label{lem:chain-compatible-cutoffs}
Let $\mathsf O\subset\R^{3d}$ be a bounded smooth open set.  After the
reordering
\[
    x^1=r,\qquad x^2=p,\qquad x^3=\theta,
\]
there exist cutoff drifts, indexed by $0\le\delta\le1$,
\[
    B_1^{\mathsf O},B_2^{\delta,\mathsf O}
        \in C_b^\infty\left(\R^{3d};\R^d\right),
    \qquad
    B_3^{\delta,\mathsf O}
        \in C_b^\infty\left(\R^{2d};\R^d\right),
\]
with
\[
    B^{\delta,\mathsf O}(x)
    =\left(
        B_1^{\mathsf O}(x),
        B_2^{\delta,\mathsf O}(x),
        B_3^{\delta,\mathsf O}\left(x^2,x^3\right)
    \right),
\]
such that, for every $x\in\mathsf O$,
\begin{equation}\label{eq:cutoff-agreement}
    B^{\delta,\mathsf O}(x)
    =
    \begin{pmatrix}
        -\gamma x^2-\gamma x^1\\
        -L^{-1}\nabla U(x^3)+\gamma x^1-\delta Lx^2\\
        x^2-\delta\nabla U(x^3)
    \end{pmatrix}.
\end{equation}
Moreover, for every $k\ge0$,
\begin{equation}\label{eq:cutoff-uniform-Ck}
    \sup_{0\le\delta\le1}
    \left(
        \left\|B_1^{\mathsf O}\right\|_{C_b^k}
        +\left\|B_2^{\delta,\mathsf O}\right\|_{C_b^k}
        +\left\|B_3^{\delta,\mathsf O}\right\|_{C_b^k}
    \right)<\infty.
\end{equation}
On $\mathsf O$,
\begin{equation}\label{eq:cutoff-chain-links}
\begin{aligned}
    J_{x^1}B_2^{\delta,\mathsf O}=\gamma I_{d},
    \qquad
    J_{x^2}B_3^{\delta,\mathsf O}=I_{d}.
\end{aligned}
\end{equation}
The same assertions hold for cutoff drifts $B^{\delta,\mathsf O,*}$ of the
adjoint dynamics, with
\begin{equation}\label{eq:adjoint-cutoff-agreement}
    B^{\delta,\mathsf O,*}(x)
    =
    \begin{pmatrix}
        \gamma x^2-\gamma x^1\\
        L^{-1}\nabla U(x^3)-\gamma x^1-\delta Lx^2\\
        -x^2-\delta\nabla U(x^3)
    \end{pmatrix},
    \qquad x\in\mathsf O,
\end{equation}
and hence with chain links $-\gamma I_d$ and $-I_d$.
\end{lemma}

\begin{proof}
In the reordered variables $x=(x^1,x^2,x^3)=(r,p,\theta)$, the forward
regularized drift is
\[
\begin{aligned}
    b_1(x)
    &=-\gamma x^2-\gamma x^1,\\
    b_2^\delta(x)
    &=-L^{-1}\nabla U(x^3)+\gamma x^1-\delta Lx^2,\\
    b_3^\delta(x^2,x^3)
    &=x^2-\delta\nabla U(x^3).
\end{aligned}
\]
Choose $R$ so large that $\overline{\mathsf O}\subset B(0,R)$, and choose smooth bounded
maps $\kappa_{123}:\R^{3d}\to\R^{3d}$ and
$\kappa_{23}:\R^{2d}\to\R^{2d}$ with bounded derivatives of all orders, equal
to the identity on neighborhoods of $B(0,R)$ and of the projection of
$B(0,R)$ onto the $(x^2,x^3)$ variables, respectively.  Set
\[
    B_1^{\mathsf O}(x)
    =b_1\left(\kappa_{123}(x)\right),\quad
    B_2^{\delta,\mathsf O}(x)
    =b_2^\delta\left(\kappa_{123}(x)\right),\quad
    B_3^{\delta,\mathsf O}(x)
    =b_3^\delta\left(\kappa_{23}\left(x^2,x^3\right)\right).
\]
The third identity makes $B_3^{\delta,\mathsf O}$ independent of $x^1$.
Since both cutoff maps are the identity near $\overline{\mathsf O}$,
$B_i^{\delta,\mathsf O}=b_i^\delta$
on $\mathsf O$,
$i=1,2,3$,
where $b_1^\delta:=b_1$.  Therefore, on $\mathsf O$,
\[
    J_{x^1}B_2^{\delta,\mathsf O}=\gamma I_d,
    \qquad
    J_{x^2}B_3^{\delta,\mathsf O}=I_d.
\]
The boundedness of the cutoff images and the smoothness of $U$ give, for every
$k\ge0$,
\[
    \sup_{0\le\delta\le1}
    \max_{1\le i\le3}
    \left\|D^kB_i^{\delta,\mathsf O}\right\|_\infty<\infty.
\]
For the adjoint drift, the two chain links have the opposite signs; applying
the same cutoff maps gives the same triangular dependence and the same
derivative bounds.
\end{proof}

\begin{remark}[Density estimates and comparison with Lee--Ramil--Seo]
The role of Proposition~\ref{prop:localized-cutoff-heat-kernel} is analogous
to that of the localized density bound in Lee--Ramil--Seo's boundary-layer
analysis.  For the underdamped Langevin chain $p\to q$, they modify the
coefficients outside the bounded pre-killing region and apply the explicit
parametrix Gaussian upper bound of Konakov--Menozzi--Molchanov
\cite[Theorem~2.1]{KonakovMenozziMolchanov2010} to prove
\cite[Lemma~5.5]{LeeRamilSeo2026}.  Its boundary-layer consequence
\cite[Corollary~5.6]{LeeRamilSeo2026} is then used to control small-velocity
boundary entries in their regularization-stability argument.

The present third-order Langevin diffusion \eqref{eq:sde} has the longer chain
\[
    r\longrightarrow p\longrightarrow \theta ,
\]
The analogous boundary-layer step therefore requires a density estimate that
is uniform in $\delta$ and respects the three-layer triangular H\"ormander
structure.  We obtain this estimate through a chain-compatible cutoff and a
$\delta$-uniform adaptation of Pigato's Malliavin-covariance argument; the
details are given in Lemma~\ref{lem:uniform-pigato-covariance} and
Appendix~\ref{app:uniform-perturbation}.
\end{remark}

For $k\in\mathbb N_0$ and $p\ge1$, let $\mathbb D^{k,p}$ denote the standard
Malliavin--Sobolev space: a random variable belongs to $\mathbb D^{k,p}$ when
its Malliavin derivatives up to order $k$ have finite $L^p$ moments.  
Thus, $k$ is the Malliavin differentiability
order and $p$ is the integrability exponent.
We write
$\|\cdot\|_{k,p}$ for the corresponding norm, defined componentwise for
vector- and matrix-valued random variables; see
\cite[Section~1.2]{Nualart2006}.

\begin{lemma}[Uniform Pigato-type inverse covariance estimate]
\label{lem:uniform-pigato-covariance}
Fix $\eps>0$.  For the cutoff regularized family of
Lemma~\ref{lem:chain-compatible-cutoffs}, use the block order
\[
    X^1=r,\qquad X^2=p,\qquad X^3=\theta .
\]
Let $J_{t,s}^{\delta,\mathsf O}$ be the Jacobian flow from time $s$ to time
$t$ along the full cutoff regularized process.  The three constant noise
matrices in the block order $(r,p,\theta)$ are
\begin{equation}\label{eq:regularized-noise-blocks}
    \Sigma_r=
    \begin{pmatrix}
        \sqrt{2\gamma\eps/L}\,I_d\\ 0\\ 0
    \end{pmatrix},
    \qquad
    \Sigma_p^\delta=
    \begin{pmatrix}
        0\\ \sqrt{2\eps\delta}\,I_d\\ 0
    \end{pmatrix},
    \qquad
    \Sigma_\theta^\delta=
    \begin{pmatrix}
        0\\ 0\\ \sqrt{2\eps\delta}\,I_d
    \end{pmatrix}.
\end{equation}
Each matrix in \eqref{eq:regularized-noise-blocks} belongs to
$\R^{3d\times d}$.  Define the Malliavin covariance contribution generated
only by the $r$-Brownian motion by
\[
    \mathcal M_{t,r}^{\delta,\mathsf O}
    =
    \int_0^t
    J_{t,s}^{\delta,\mathsf O}
    \Sigma_r\Sigma_r^\top
    \left(J_{t,s}^{\delta,\mathsf O}\right)^\top\,ds.
\]
Set
\[
    T_t=\operatorname{diag}\left(t^{1/2}I_d,t^{3/2}I_d,t^{5/2}I_d\right).
\]
Then, for every $q\ge1$, every $M<\infty$, and every compact
$K_0\subset\mathsf O$,
\[
    \sup_{\substack{
        0\le\delta\le1\\
        z\in K_0\\
        0<t\le M}}
    \mathbb E_z\left[
        \lambda_{\min}\left(
            T_t^{-1}\mathcal M_{t,r}^{\delta,\mathsf O}T_t^{-1}
        \right)^{-q}
    \right]
    <\infty .
\]
The same estimate holds for the adjoint cutoff family.  All constants are for
fixed $\eps$ and may depend on $\eps$.
\end{lemma}

\begin{proof}
In the reordered cutoff system, the diffusion matrices are constant in space:
\[
    dX_t
    =
    B^{\delta,\mathsf O}(X_t)\,dt
    +\Sigma_r\,dW_t^r
    +\Sigma_p^\delta\,dW_t^p
    +\Sigma_\theta^\delta\,dW_t^\theta,
\]
where $W^r$, $W^p$, and $W^\theta$ are independent standard
$d$-dimensional Brownian motions, with $W^r=B$ in
\eqref{eq:sde-delta}.
The matrices $\Sigma_p^\delta$ and $\Sigma_\theta^\delta$ are the constant
regularizing diffusion matrices in the $p$- and $\theta$-coordinates given in
\eqref{eq:regularized-noise-blocks}.
Thus, the $r$-noise Malliavin derivative is
$D_s^rX_t^{\delta,\mathsf O}=J_{t,s}^{\delta,\mathsf O}\Sigma_r$, computed
along the full regularized trajectory.  The extra $p$- and $\theta$-noises
remain in this trajectory and enter the estimate through the semimartingale
remainders of the coefficient expansions.

Set
\[
    \overline{\mathcal M}_{t,r}^{\delta,\mathsf O}
    :=T_t^{-1}\mathcal M_{t,r}^{\delta,\mathsf O}T_t^{-1}.
\]
The normalized inverse-flow expansion in
Lemma~\ref{lem:appendix-chain-remainder}, the remainder-covariance estimate in
Lemma~\ref{lem:appendix-remainder-covariance}, and the stopping argument in
Lemma~\ref{lem:appendix-pigato-stopping} give negative moments of every order
for the reduced covariance.  Lemma~\ref{lem:appendix-scaled-inverse-jacobian}
then restores the endpoint Jacobian factor, and
Lemma~\ref{lem:appendix-negative-moments} gives the required short-time bound.
Finally, Lemma~\ref{lem:appendix-finite-time-extension} transports that bound
to every fixed interval $0<t\le M$ and verifies the same conclusion for the
adjoint cutoff family.  This proves the claim.
\end{proof}

The inverse-covariance estimate above provides the uniform non-degeneracy
input for the Malliavin density criterion.  Combining it with uniform
Malliavin derivative bounds for the cutoff flow yields the local heat-kernel
estimate needed in the boundary-layer argument.

\begin{proposition}[Localized cutoff heat-kernel bound]
\label{prop:localized-cutoff-heat-kernel}
Let $\mathsf O\subset\R^{3d}$ be a bounded smooth open set, and let
$Z^{\delta,\mathsf O}$ denote the cutoff process constructed in
Lemma~\ref{lem:chain-compatible-cutoffs}.  After the reordering
\[
    X^1=r,\qquad X^2=p,\qquad X^3=\theta,
\]
the cutoff drifts preserve Pigato's triangular dependence structure
\[
    B_j=B_j\left(x^{j-1},\ldots,x^3\right),\qquad j=2,3.
\]
Then, for every $0<t_0<M<\infty$ and compact sets
$K_0,K_1\subset\mathsf O$, the transition laws of
$Z^{\delta,\mathsf O}$ admit densities
$p_\delta^{\mathsf O}(t,z,y)$ satisfying
\[
    \sup_{\substack{
        0\le\delta\le1\\
        t\in[t_0,M]\\
        z\in K_0,\ y\in K_1}}
    p_\delta^{\mathsf O}(t,z,y)<\infty .
\]
The same estimate holds for the adjoint cutoff family.
\end{proposition}

\begin{proof}
For the cutoff family all coefficients have bounded derivatives, uniformly for
$0\le\delta\le1$, by Lemma~\ref{lem:chain-compatible-cutoffs}.  At
$\delta=0$ the process is hypoelliptic and its Malliavin non-degeneracy is
supplied by the chain $r\to p\to\theta$.  At $\delta>0$ the process is
elliptic, but the ellipticity constants degenerate as $\delta\downarrow0$.
Therefore the proof uses the same uniform chain estimate for every
$0\le\delta\le1$, rather than any lower ellipticity bound from the added
$p$- and $\theta$-noises.

Let $\mathcal M_t^{\delta,\mathsf O}$ denote the full Malliavin covariance of
the cutoff regularized process.  With the block order $(r,p,\theta)$ and the
constant matrices $\Sigma_r$, $\Sigma_p^\delta$, and $\Sigma_\theta^\delta$
from Lemma~\ref{lem:uniform-pigato-covariance},
\[
    \begin{aligned}
    \mathcal M_t^{\delta,\mathsf O}
    &=
    \int_0^t
    J_{t,s}^{\delta,\mathsf O}
    \left(
        \Sigma_r\Sigma_r^\top
        +\Sigma_p^\delta\left(\Sigma_p^\delta\right)^\top
        +\Sigma_\theta^\delta\left(\Sigma_\theta^\delta\right)^\top
    \right)
    \left(J_{t,s}^{\delta,\mathsf O}\right)^\top\,ds  \\
    &=
    \mathcal M_{t,r}^{\delta,\mathsf O}
    +\mathcal M_{t,p}^{\delta,\mathsf O}
    +\mathcal M_{t,\theta}^{\delta,\mathsf O}.
    \end{aligned}
\]
Here $\mathcal M_{t,r}^{\delta,\mathsf O}$ is the $r$-noise contribution
computed with the Jacobian of the full regularized cutoff process; the influence
of the $p$- and $\theta$-noises is retained through that Jacobian.  Since the
last two summands are positive semidefinite,
\[
    \mathcal M_t^{\delta,\mathsf O}
    \succeq
    \mathcal M_{t,r}^{\delta,\mathsf O}.
\]
Consequently,
\[
    \lambda_{\min}\left(
        T_t^{-1}\mathcal M_t^{\delta,\mathsf O}T_t^{-1}
    \right)
    \ge
    \lambda_{\min}\left(
        T_t^{-1}\mathcal M_{t,r}^{\delta,\mathsf O}T_t^{-1}
    \right),
\]
and Lemma~\ref{lem:uniform-pigato-covariance} gives uniform negative moments
for the full rescaled covariance on $0<t\le M$.

Put
\[
    \overline{\mathcal M}_t^{\delta,\mathsf O}
    =T_t^{-1}\mathcal M_t^{\delta,\mathsf O}T_t^{-1}.
\]
Since
\[
    \mathcal M_t^{\delta,\mathsf O}
    =T_t\overline{\mathcal M}_t^{\delta,\mathsf O}T_t,
\]
and
\[
    c_{t_0,M}:=
    \inf_{t\in[t_0,M]}\lambda_{\min}(T_t)>0,
\]
we have
\[
    \lambda_{\min}\left(\mathcal M_t^{\delta,\mathsf O}\right)
    \ge
    c_{t_0,M}^{2}
    \lambda_{\min}\left(
        \overline{\mathcal M}_t^{\delta,\mathsf O}
    \right),
    \qquad t\in[t_0,M].
\]
Consequently, Lemma~\ref{lem:uniform-pigato-covariance} gives, for every
$q\ge1$,
\begin{equation}\label{eq:uniform-smallest-eigenvalue-density}
    \sup_{\substack{
        0\le\delta\le1,\ z\in K_0\\
        t\in[t_0,M]}}
    \E_z\left[
        \lambda_{\min}\left(
            \mathcal M_t^{\delta,\mathsf O}
        \right)^{-q}
    \right]
    \le
    c_{t_0,M}^{-2q}
    \sup_{\substack{
        0\le\delta\le1,\ z\in K_0\\
        t\in[t_0,M]}}
    \E_z\left[
        \lambda_{\min}
        \left(\overline{\mathcal M}_t^{\delta,\mathsf O}\right)^{-q}
    \right]
    <\infty .
\end{equation}

The remaining Malliavin--Sobolev estimates needed for the density criterion are
also uniform.  Define the horizontal block concatenation
\[
    \Sigma^\delta
    :=\left[\,\Sigma_r\ \ \Sigma_p^\delta\ \ \Sigma_\theta^\delta\,\right]
    =
    \begin{pmatrix}
        \sqrt{2\gamma\eps/L}\,I_d&0&0\\
        0&\sqrt{2\eps\delta}\,I_d&0\\
        0&0&\sqrt{2\eps\delta}\,I_d
    \end{pmatrix}
    \in\R^{3d\times3d},
\]
where the three block columns are defined in
\eqref{eq:regularized-noise-blocks}.
Since the diffusion matrices are constant, the first Malliavin derivative
satisfies
\[
    D_sZ_t^{\delta,\mathsf O,z}
    =J_{t,s}^{\delta,\mathsf O,z}\Sigma^\delta,
    \qquad 0\le s\le t,
\]
and higher Malliavin derivatives satisfy the corresponding iterated
variational equations.  Their inhomogeneous terms are finite sums of products
of lower-order Malliavin derivatives and derivatives of
$B^{\delta,\mathsf O}$.  After cutoff, these coefficient derivatives are
bounded uniformly for
$0\le\delta\le1$, and $\Sigma^\delta$ is constant in space and uniformly
bounded.  Therefore,
for every Malliavin differentiability order $k$ and integrability exponent
$p$, repeated applications of the
Burkholder-Davis-Gundy inequality and Gr\"onwall's lemma to the variational
equations and their Malliavin derivatives yield
\[
    \sup_{\substack{
        0\le\delta\le1,\ z\in K_0\\
        t\in[t_0,M]}}
    \left\|Z_t^{\delta,\mathsf O,z}\right\|_{\mathbb D^{k,p}}
    <\infty .
\]
The constants depend only on the cutoff derivative bounds, $K_0$, $M$,
$k,p$, and the model parameters, and are uniform in $\delta$.

For $F=Z_t^{\delta,\mathsf O,z}$, its Malliavin covariance is
$\gamma_F=\mathcal M_t^{\delta,\mathsf O}$.  Thus
\eqref{eq:uniform-smallest-eigenvalue-density} is precisely the uniform bound
on
\[
    \Gamma_F(q)
    :=1+\E_z\left[\lambda_{\min}(\gamma_F)^{-q}\right]
\]
required in Pigato's density criterion.  Apply the empty-multiindex (density)
case of \cite[Lemma~A.1]{Pigato2022}, with the integrability exponent chosen as
in that lemma.  The preceding Malliavin--Sobolev estimates are uniform, and the
tail factor appearing there satisfies
$\Pbb_z(|F|>|y|/2)^b\le1$.  Hence the criterion yields
\[
    \sup_{\substack{
        0\le\delta\le1\\
        t\in[t_0,M]\\
        z\in K_0,\ y\in K_1}}
    p_\delta^{\mathsf O}(t,z,y)<\infty .
\]
The restriction $t\ge t_0>0$ removes the explicit short-time singular factor in
the density bound; the short-time rescaling has already been accounted for by
$T_t$ in Lemma~\ref{lem:uniform-pigato-covariance}.  For the adjoint cutoff
family the triangular chain is the same up to signs, so the relevant chain
matrices have the same singular values.  The preceding argument and constants
therefore apply unchanged to the adjoint densities.
\end{proof}

The next lemma records that, before exiting a fixed compact set, the auxiliary
regularized process is close to the degenerate process on finite time
intervals under the explicit SDE coupling.

\begin{lemma}[Compact-time pathwise stability]\label{lem:pathwise}
Fix $T<\infty$ and a compact set $K\subset\R^{3d}$.  Let $\calK$ be a compact
neighborhood of $K$, and let
\[
    \sigma_\delta:=\inf\left\{t\ge0:Z_t^\delta\notin\calK
    \text{ or } Z_t\notin\calK\right\}.
\]
If $Z_t^\delta$ and $Z_t$ start from the same point $z\in K$ and are coupled
by using the same Brownian motion $B$ in the $r$-coordinate, then, for every
$\eta>0$,
\[
    \sup_{z\in K}
    \Pbb_z\left(
        \sup_{0\le t\le T\wedge\sigma_\delta}\left|Z_t^\delta-Z_t\right|>\eta
    \right)
    \longrightarrow0
    \qquad\text{as }\delta\downarrow0.
\]
Moreover,
\begin{equation}\label{eq:pathwise-L2}
    \sup_{z\in K}
    \E_z\left[
        \sup_{0\le t\le T\wedge\sigma_\delta}
        \left|Z_t^\delta-Z_t\right|^2
    \right]
    \le C_{T,\calK}\delta .
\end{equation}
The same statement holds for the adjoint regularized processes.
\end{lemma}

\begin{proof}
On $\calK$, the vector field $\nabla U$ is bounded and Lipschitz.  Write
$\Delta\theta_t=\theta_t^\delta-\theta_t$,
$\Delta p_t=p_t^\delta-p_t$, and $\Delta r_t=r_t^\delta-r_t$.  Up to time
$T\wedge\sigma_\delta$,
\begin{align}
    \Delta\theta_t
    &=\int_0^t\Delta p_s ds
      -\delta\int_0^t\nabla U\left(\theta_s^\delta\right) ds
      +\sqrt{2\eps\delta} W_t^\theta,
      \label{eq:pathwise-difference-theta}\\
    \Delta p_t
    &=-\frac1L\int_0^t
      \left(\nabla U\left(\theta_s^\delta\right)-\nabla U(\theta_s)\right) ds
      +\gamma\int_0^t\Delta r_s ds
      -\delta L\int_0^t p_s^\delta ds
      +\sqrt{2\eps\delta} W_t^p,
      \label{eq:pathwise-difference-p}\\
    \Delta r_t
    &=-\gamma\int_0^t\Delta p_s ds
      -\gamma\int_0^t\Delta r_s ds.
      \label{eq:pathwise-difference-r}
\end{align}
Therefore, for a constant $C=C(T,\calK,L,\gamma,U)$,
\[
    \sup_{0\le s\le t\wedge\sigma_\delta}\left|Z_s^\delta-Z_s\right|
    \le
    C\int_0^t\sup_{0\le u\le s\wedge\sigma_\delta}\left|Z_u^\delta-Z_u\right| ds
    +C\delta
    +C\sqrt\delta\sup_{0\le s\le T}\left(\left|W_s^\theta\right|+\left|W_s^p\right|\right).
\]
Gr\"{o}nwall's lemma gives
\begin{equation}\label{eq:pathwise-gronwall}
    \sup_{0\le s\le T\wedge\sigma_\delta}\left|Z_s^\delta-Z_s\right|
    \le
    C\delta
    +C\sqrt\delta\sup_{0\le s\le T}\left(\left|W_s^\theta\right|+\left|W_s^p\right|\right).
\end{equation}
Taking second moments in \eqref{eq:pathwise-gronwall} gives
\eqref{eq:pathwise-L2}, since the Brownian
suprema have finite second moments on $[0,T]$.  In particular, the right-hand
side converges to zero in probability, uniformly in $z\in K$.
The adjoint case is identical after replacing the drift by the adjoint drift.
\end{proof}

Lemma~\ref{lem:pathwise} gives convergence only before exit from a prescribed
compact set.  The uniform finite-time moment bounds now allow us to remove
this localization and obtain global finite-horizon convergence.

\begin{theorem}[$\delta$-coupling of the dynamics]\label{thm:delta-coupling}
Fix $T<\infty$ and a compact set $K\subset\R^{3d}$.  Couple
$Z^\delta$ and $Z$ by using the same Brownian motion $B$ in the $r$-coordinate
and the additional Brownian motions $W^\theta,W^p$ only in the regularized
equations \eqref{eq:sde-delta}.  Then, for every $\eta>0$,
\[
    \sup_{z\in K}
    \Pbb_z\left(
        \sup_{0\le t\le T}\left|Z_t^\delta-Z_t\right|>\eta
    \right)
    \longrightarrow0
    \qquad\text{as }\delta\downarrow0.
\]
The same statement holds for the adjoint regularized processes.
\end{theorem}

\begin{proof}
Let $\calK_R:=\{|z|\le R\}$.  By Assumption~\ref{ass:controlled-coefficients},
for every $q\ge1$,
\[
    \sup_{0\le\delta\le1}\sup_{z\in K}
    \E_z^\delta\left[\sup_{0\le t\le T}\left|Z_t^\delta\right|^q\right]
    +
    \sup_{z\in K}
    \E_z\left[\sup_{0\le t\le T}|Z_t|^q\right]
    <\infty .
\]
Therefore
\[
    \lim_{R\to\infty}\sup_{0\le\delta\le1}\sup_{z\in K}
    \Pbb_z\left(
        Z^\delta \text{ or } Z \text{ exits } \calK_R
        \text{ before }T
    \right)=0 .
\]
On the complementary event, Lemma~\ref{lem:pathwise} applies with
$\calK=\calK_R$.  Hence, for every fixed $R$,
\[
    \sup_{z\in K}
    \Pbb_z\left(
        \sup_{0\le t\le T}\left|Z_t^\delta-Z_t\right|>\eta,\ 
        \sup_{0\le t\le T}\left|Z_t^\delta\right|\le R,\ 
        \sup_{0\le t\le T}|Z_t|\le R
    \right)\longrightarrow0
    \qquad\text{as }\delta\downarrow0 .
\]
Letting first $\delta\downarrow0$ and then $R\to\infty$ proves the claim.  The
adjoint proof is the same, using the adjoint moment bound in
Assumption~\ref{ass:controlled-coefficients}.
\end{proof}

We also need a bounded-domain killing estimate to remove finite-time
truncations in the limiting hitting identities.

\begin{lemma}[Almost-sure killing in bounded domains]\label{lem:finite-killing}
Under Assumptions~\ref{ass:U} and~\ref{ass:controlled-coefficients}, and in
the \hyperref[setup:domains]{bounded-domain setup}, for every
$z\in\calD\setminus B$,
\[
    \Pbb_z(\tau_\calC<\infty)=1.
\]
Moreover, $\sup_{z\in\overline\calD\setminus B}\E_z[\tau_\calC]<\infty$.
The analogous statement holds for the adjoint process.
\end{lemma}

\begin{proof}
It suffices to prove a uniform geometric tail on the compact set
$\overline\calD\setminus B$ and then restrict to
$\calD\setminus B$.  The boundary points
$\partial B\cup\partial\calD$ are included for compactness of the finite
cover.  Points of $\partial\calD$ are already killed, whereas points of
$\partial B$ are handled by the same controlled-exit construction as the
interior starting points.

Since $\calD$ is bounded, choose $R>0$ such that
$\overline\calD\subset\{|z|<R\}$.  Consider the controlled system associated
with \eqref{eq:sde}, obtained by replacing the Brownian increment in the
$r$-equation by a deterministic Cameron--Martin control.
Set
\[
    a_r=\sqrt{2\gamma\eps/L},
\]
and, for $T>0$, let
\[
    \mathcal H_T
    :=
    \left\{
        h\in\mathcal{AC}\left([0,T];\R^d\right):
        h(0)=0,\ \dot h\in L^2\left([0,T];\R^d\right)
    \right\}
\]
be the Cameron--Martin space of the $d$-dimensional Brownian motion on
$[0,T]$; see \cite[Section~1.1]{Nualart2006}.  Its norm is
\begin{equation}\label{eq:Cameron-Martin-norm}
    \|h\|_{\mathcal H_T}^2
    =\int_0^T|\dot h(t)|^2\,dt.
\end{equation}
For $h\in\mathcal H_T$, replacing $dB_t$ by $\dot h(t)dt$ gives the
controlled system
\[
    \frac{d\theta}{dt}=p,
    \qquad
    \frac{dp}{dt}=-L^{-1}\nabla U(\theta)+\gamma r,
    \qquad
    \frac{dr}{dt}=-\gamma p-\gamma r+a_r\dot h(t).
\]
Fix $z_i\in\overline\calD\setminus B$ and write
$z_i=(\theta_i,p_i,r_i)$.  Since $\overline\calD$ is compact, its
$\theta$-projection
\[
    \Pi_\theta\overline\calD
    =
    \left\{\theta:\text{ there are }p,r\text{ with }(\theta,p,r)\in\overline\calD\right\}
\]
is compact.  Choose $\theta_{\rm out}$ outside this projection and set
\[
    m_i=\frac14\dist\left(\theta_{\rm out},\Pi_\theta\overline\calD\right)>0 .
\]
Pick $T_i>0$ and a smooth polynomial interpolation
$\Theta_i:[0,T_i]\to\R^d$ satisfying
\[
    \Theta_i(0)=\theta_i,\qquad
    \dot\Theta_i(0)=p_i,\qquad
    \ddot\Theta_i(0)=-L^{-1}\nabla U(\theta_i)+\gamma r_i,
\]
and $\Theta_i(T_i)=\theta_{\rm out}$.  Such a polynomial is obtained by the
following explicit construction.  Set
\[
    a_i:=-L^{-1}\nabla U(\theta_i)+\gamma r_i
\]
and take
\[
    \Theta_i(t)
    =\theta_i+p_it+\frac12a_it^2
    +\frac{t^3}{T_i^3}
    \left(
        \theta_{\rm out}-\theta_i-p_iT_i-\frac12a_iT_i^2
    \right).
\]
These four interpolation conditions uniquely determine a vector-valued
polynomial of degree at most three.  Define
\[
    P_i(t)=\dot\Theta_i(t),\qquad
    R_i(t)=\gamma^{-1}
    \left(\ddot\Theta_i(t)+L^{-1}\nabla U(\Theta_i(t))\right),
\]
and set
\[
    v_i(t)=\dot R_i(t)+\gamma P_i(t)+\gamma R_i(t),
    \qquad
    h_i(t)=a_r^{-1}\int_0^t v_i(s)\,ds.
\]
Then $a_r\dot h_i=v_i$, so that $(\Theta_i,P_i,R_i)$ solves the controlled system
with initial condition $z_i$ and control $h_i$.  Since $v_i$ is smooth,
\[
    \|h_i\|_{\mathcal H_{T_i}}^2
    =a_r^{-2}\int_0^{T_i}|v_i(t)|^2\,dt<\infty.
\]
The path
\[
    K_i^0=\{(\Theta_i(t),P_i(t),R_i(t)):0\le t\le T_i\}
\]
is compact, and the terminal point has a fixed exit margin:
\[
\dist\left((\Theta_i(T_i),P_i(T_i),R_i(T_i)),\calD\right)
    \ge
    \dist\left(\theta_{\rm out},\Pi_\theta\overline\calD\right)
    =4m_i .
\]
Thus the compact set and the exit margin are fixed before any support-theorem
or stochastic-tube argument is invoked.

By continuous dependence of the controlled ODE on initial conditions, there are
an open neighborhood $U_i$ of $z_i$, a number $\eta_i>0$, and a compact set
$K_i$ such that, for every $z\in U_i$,
\[
    \operatorname{dist}\left(\phi_{T_i}^{z,h_i},\calD\right)\ge3\eta_i,
    \qquad
    \phi_s^{z,h_i}\in K_i^\circ,\quad 0\le s\le T_i .
\]
Choose a cutoff system which agrees with the original coefficients on the
$2\eta_i$-neighborhood of $K_i$.  For this cutoff system, the localized It\^o
map is uniformly continuous in the driving path and in the initial condition on
$U_i$.  Indeed, the noise is additive; after subtracting the corresponding
Brownian path from the stochastic equation, the localized system becomes a
random ODE with globally Lipschitz coefficients, whose solution map is uniformly
continuous in the initial point and in the driving path on bounded tubes.
Hence, there exists some $\varepsilon_i>0$ such that the Brownian tube event
\[
    E_i=
    \left\{
        \sup_{0\le s\le T_i}|B_s-h_i(s)|<\varepsilon_i
    \right\}
\]
implies, simultaneously for all $z\in U_i$,
\[
    \sup_{0\le s\le T_i}
    \left|Z_s^z-\phi_s^{z,h_i}\right|<\eta_i .
\]
On $E_i$ the stochastic path remains in the region where the cutoff and
original coefficients coincide and exits $\calD$ by time $T_i$.
Since $h_i\in\mathcal H_{T_i}$, the Cameron--Martin theorem states that the
laws of $B$ and $B-h_i$ on $C_0([0,T_i];\R^d)$ are equivalent; see
\cite{CameronMartin1944}.  Since Wiener measure has full support in this space
under the uniform topology,
\[
    \Pbb\left(\|B-h_i\|_\infty<\varepsilon_i\right)>0.
\]
Equivalently, for the localized degenerate diffusion this is a special case of
the Stroock--Varadhan support theorem
\cite[Theorem~5.2]{StroockVaradhan1972}.  Hence, with
$a_i:=\mathbb P(E_i)>0$,
\[
    \inf_{z\in U_i}
    \Pbb_z(\tau_\calC\le T_i)\ge a_i .
\]

The same construction is made on the compact set
$\overline\calD\setminus B$.  Points of $\partial\calD$ are already killed,
while the controlled paths starting on $\partial B$ are included in the
resulting finite family, which also covers $\calD\setminus B$.  Take a finite
subcover $U_1,\ldots,U_N$.
Set
\[
    T:=\max_{1\le i\le N}T_i,
    \qquad
    a:=\min_{1\le i\le N}a_i>0.
\]
Then
\[
    \inf_{z\in\calD\setminus B}
    \Pbb_z(\tau_\calC\le T)\ge a.
\]
Here hitting $B$ before exiting $\calD$ is also counted as killing, so the lower
bound remains valid for initial points close to $B$.

The strong Markov property now yields
\[
    \sup_{z\in\calD\setminus B}
    \Pbb_z(\tau_\calC>nT)
    \le (1-a)^n,
    \qquad n\ge0.
\]
Thus $\tau_\calC<\infty$ almost surely and
\[
    \sup_{z\in\calD\setminus B}\E_z[\tau_\calC]
    \le \sum_{n\ge0} T 
    \sup_{z\in\calD\setminus B}\Pbb_z(\tau_\calC>nT)
    \le \frac{T}{a}<\infty .
\]
For $z\in\partial\calD$ one has $\tau_\calC=0$, so the same bound gives the
stated supremum over $\overline\calD\setminus B$.
For the adjoint dynamics, choose instead
$P_i^*=-\dot\Theta_i$, with $\dot\Theta_i(0)=-p_i$, while retaining
$R_i^*=\gamma^{-1}(\ddot\Theta_i+L^{-1}\nabla U(\Theta_i))$.  Then the
adjoint controlled equations are satisfied by taking
$a_r\dot h_i^*=\dot R_i^*-\gamma P_i^*+\gamma R_i^*$; hence the same tube,
finite-cover, and strong Markov arguments apply.
\end{proof}

\subsection{Small-\texorpdfstring{$r$}{r} Boundary Stability for Phase-Space Balls}
\label{sec:small}


This section proves Theorem~\ref{thm:ball-stability}.  The proof adapts the
boundary-stability strategy of Lee--Ramil--Seo~\cite{LeeRamilSeo2026}: one first
proves immediate crossing at non-characteristic boundary hits and then shows
that boundary entries through the near-characteristic region have uniformly
vanishing probability.  In the present third-order setting, the highest
auxiliary variable $r$ plays the role of the momentum variable in the
underdamped Langevin setting.  For phase-space balls the martingale coefficient
of the signed boundary function is proportional to $r$, so the
non-characteristic condition is $r\ne0$.

\begin{definition}[Characteristic boundary sets]\label{def:characteristic-boundary}
Let
\[
    \Sigma_A:=\partial A,\qquad
    \Sigma_B:=\partial B,\qquad
    \Sigma_D:=\partial\calD .
\]
For a boundary component $\Sigma$, define its characteristic part by
\[
    \Gamma_\Sigma:=\{z=(\theta,p,r)\in\Sigma:r=0\}.
\]
For phase-space balls, $\Gamma_\Sigma$ is exactly the set on which the direct
diffusion direction is tangent to $\Sigma$.
\end{definition}

\begin{definition}[Signed defining functions]\label{def:signed-boundary-functions}
For the inner balls centered at $(c,0,0)$, $c=m,s$, set
\[
    \rho_c(\theta,p,r):=|\theta-c|^2+|p|^2+|r|^2-\rho^2 .
\]
For the outer boundary, set
\[
    \rho_N(z):=|z|^2-N^2 .
\]
We write $\rho_\Sigma$ for the corresponding quadratic defining function on a
boundary component $\Sigma\in\{\partial A,\partial B,\partial\calD\}$.
If $\rho$ denotes any of these functions, then the martingale part of
$\rho(Z_t)$ for the limiting process is
\[
    2\sigma\int_0^t r_s\cdot dB_s,
    \qquad
    \sigma=\sqrt{2\gamma\eps/L}.
\]
Thus the non-characteristic condition is $r\ne0$.
\end{definition}

\begin{remark}[Shortcut when $d\ge2$]
The proof uses the uniform small-$r$ estimate in every dimension.  In
dimensions $d\ge2$, polarity properties of the localized $r$-diffusion may
provide an alternative shortcut in some formulations.
\end{remark}

The first boundary ingredient shows that once the incoming $r$-component is
nonzero, the signed boundary function crosses both sides immediately.

\begin{lemma}[Immediate crossing when $r\ne0$]
\label{lem:crossing-nonzero-r}
Let $\rho$ be one of the signed defining functions in
Definition~\ref{def:signed-boundary-functions}. Suppose that the limiting
forward or adjoint process starts from a boundary point
$z_0=(\theta_0,p_0,r_0)$ with $r_0\ne0$. Then, for every $\eta>0$,
\[
\begin{aligned}
&\Pbb_{z_0}\left(
    \sup_{0<t\le\eta}\rho(Z_t)>0,\ 
    \inf_{0<t\le\eta}\rho(Z_t)<0
\right)=1,\\
&\Pbb_{z_0}^{*}\left(
    \sup_{0<t\le\eta}\rho(Z_t^*)>0,\ 
    \inf_{0<t\le\eta}\rho(Z_t^*)<0
\right)=1.
\end{aligned}
\]
For every fixed $\delta>0$, the same almost-sure immediate-crossing
statement holds for the corresponding regularized forward and adjoint
processes: for every $\eta>0$,
\[
\begin{aligned}
&\Pbb_{z_0}^{\delta}\left(
    \sup_{0<t\le\eta}\rho\left(Z_t^\delta\right)>0,\ 
    \inf_{0<t\le\eta}\rho\left(Z_t^\delta\right)<0
\right)=1,\\
&\Pbb_{z_0}^{\delta,*}\left(
    \sup_{0<t\le\eta}\rho\left(Z_t^{\delta,*}\right)>0,\ 
    \inf_{0<t\le\eta}\rho\left(Z_t^{\delta,*}\right)<0
\right)=1.
\end{aligned}
\]
\end{lemma}

\begin{proof}
We prove all four cases simultaneously.  For
$\diamond\in\{\mathrm f,\mathrm a\}$, use the conventions
\[
    Z^{0,\mathrm f}=Z,\qquad Z^{0,\mathrm a}=Z^*,
    \qquad
    Z^{\delta,\mathrm f}=Z^\delta,\qquad
    Z^{\delta,\mathrm a}=Z^{\delta,*}.
\]
For the three defining functions in
Definition~\ref{def:signed-boundary-functions}, set
\[
    c_\rho:=
    \begin{cases}
        m,&\rho=\rho_m,\\
        s,&\rho=\rho_s,\\
        0,&\rho=\rho_N.
    \end{cases}
\]
The last case corresponds to the outer boundary
$\partial\calD$, whose defining function $\rho_N(z)=|z|^2-N^2$ is centered
at the origin.
Since the forward and adjoint processes have
the same diffusion coefficients, It\^o's formula, up to the exit time from a
fixed compact coordinate patch containing $z_0$, gives
\begin{equation}\label{eq:crossing-semimartingale-decomposition}
    \rho\left(Z_t^{\delta,\diamond}\right)-\rho(z_0)
    =A_t^{\delta,\diamond}+M_t^{\delta,\diamond},
    \qquad
    \left|A_t^{\delta,\diamond}\right|\le Ct,
\end{equation}
where
\begin{align*}
    M_t^{\delta,\diamond}
    =2\sigma\int_0^t
        r_s^{\delta,\diamond}\cdot dB_s
    +2\sqrt{2\eps\delta}\int_0^t
        \left(\theta_s^{\delta,\diamond}-c_\rho\right)\cdot dW_s^\theta+2\sqrt{2\eps\delta}\int_0^t
        p_s^{\delta,\diamond}\cdot dW_s^p.
\end{align*}
For $\delta=0$, the last two integrals vanish. 
Since the three driving Brownian
motions are independent, we have
\begin{align}
    q_t^{\delta,\diamond}
    :=\left\langle M^{\delta,\diamond}\right\rangle_t 
    =4\sigma^2\int_0^t\left|r_s^{\delta,\diamond}\right|^2\,ds
      +8\eps\delta\int_0^t
        \left(
            \left|\theta_s^{\delta,\diamond}-c_\rho\right|^2
            +\left|p_s^{\delta,\diamond}\right|^2
        \right)ds .
    \label{eq:crossing-quadratic-variation}
\end{align}
Path continuity and $r_0\ne0$ imply
\begin{equation}\label{eq:crossing-qv-speed}
\lim_{t\downarrow0}\frac{q_t^{\delta,\diamond}}{t}
    =4\sigma^2|r_0|^2
      +8\eps\delta
        \left(|\theta_0-c_\rho|^2+|p_0|^2\right)
    =:v_\delta>0.
\end{equation}
In particular, almost surely there are random constants $t_*>0$ and
$0<c_1<c_2<\infty$ such that
\begin{equation}\label{eq:crossing-qv-comparison}
    c_1t\le q_t^{\delta,\diamond}\le c_2t,
    \qquad 0<t\le t_*.
\end{equation}

By the Dambis--Dubins--Schwarz theorem and the law of the iterated logarithm
for Brownian motion at the origin (see, e.g., \cite{RevuzYor1999}),
\begin{align}
    \limsup_{t\downarrow0}
    \frac{M_t^{\delta,\diamond}}
    {\sqrt{2q_t^{\delta,\diamond}
        \log\log(1/q_t^{\delta,\diamond})}}
    =1,
    \qquad
    \liminf_{t\downarrow0}
    \frac{M_t^{\delta,\diamond}}
    {\sqrt{2q_t^{\delta,\diamond}
        \log\log(1/q_t^{\delta,\diamond})}}=-1
    \qquad\text{a.s.}\label{eq:crossing-limsup-liminf}
\end{align}
Moreover, for each fixed $0\le\delta\le1$ and
$\diamond\in\{\mathrm f,\mathrm a\}$,
\eqref{eq:crossing-semimartingale-decomposition} and
\eqref{eq:crossing-qv-comparison} yield, as $t\downarrow0$,
\[
    \frac{|A_t^{\delta,\diamond}|}
    {\sqrt{2q_t^{\delta,\diamond}
        \log\log(1/q_t^{\delta,\diamond})}}
    \le
    \frac{Ct}{\sqrt{2c_1t\log\log(1/(c_2t))}}
    \longrightarrow0
    \qquad\text{a.s.}
\]
Combining this limit with \eqref{eq:crossing-limsup-liminf}
shows that $\rho(Z_t^{\delta,\diamond})$ takes both signs arbitrarily close to
$t=0$, almost surely.  Since the exit time from the coordinate patch is
strictly positive almost surely, the asserted conclusion follows for every
$\eta>0$.
\end{proof}

The next lemma removes exceptional hitting configurations that would make
first-hit decisions unstable under path perturbations.

\begin{lemma}[Separation of boundary hits and fixed-time atomlessness]
\label{lem:no-fixed-time-atoms-d1}
For the limiting and regularized forward and adjoint processes, simultaneous
first hits of distinct boundary components among $\partial A$, $\partial B$,
and $\partial\calD$ have probability zero.  Moreover, for every fixed $t>0$,
\[
    \Pbb_z\left(Z_t\in\partial A\cup\partial B\cup\partial\calD\right)=0,
\]
and
\[
    \Pbb_z^\delta\left(
        Z_t^\delta\in\partial A\cup\partial B\cup\partial\calD
    \right)=0.
\]
Consequently, the hitting times of these boundary components and their minima
are atomless at every deterministic time.  The same statements hold for the
adjoint processes.
\end{lemma}

\begin{proof}
The boundary components have pairwise disjoint distance collars by the
construction following the \hyperref[setup:domains]{bounded-domain setup}.
Path continuity therefore rules out simultaneous
first hits of two distinct components.

For fixed $t>0$, the limiting process has a smooth density by H\"ormander's
hypoellipticity theorem~\cite{Hormander1967}, since the bracket condition holds globally by
Proposition~\ref{prop:basic}.  This density-existence statement is local in
space and follows after the usual localization of the smooth polynomial-growth
coefficients.  For $\delta>0$, the regularized process is elliptic in the added
directions and hypoelliptic in any case; it also admits a density.  The same
statements hold for the adjoint family.  The set
$\partial A\cup\partial B\cup\partial\calD$ is a finite union of smooth
hypersurfaces and hence has Lebesgue measure zero.  Therefore the probability of
being on this boundary union at time $t$ is zero.
\end{proof}

To control near-characteristic hits, we estimate the occupation time of a thin
boundary layer in which both the boundary distance and $|r|$ are small.

\begin{remark}[Quantitative third-order propagation scale]
In the localized reordered variables
\[
    X^1=r,\qquad X^2=p,\qquad X^3=\theta,
\]
the stochastic displacement generated by the $r$-Brownian motion has the
short-time scale
\[
    \operatorname{Std}\left(X_t^1-X_0^1\right)\asymp t^{1/2},\qquad
    \operatorname{Std}\left(X_t^2-X_0^2\right)\asymp t^{3/2},\qquad
    \operatorname{Std}\left(X_t^3-X_0^3\right)\asymp t^{5/2}.
\]
Equivalently, the covariance generated by the direct $r$-noise is normalized
by
\[
    T_t=\operatorname{diag}\left(t^{1/2}I_d,t^{3/2}I_d,t^{5/2}I_d\right).
\]
In particular, for every fixed $c>0$ and $t=cb^2$,
\[
    T_{cb^2}
    =\operatorname{diag}\left(
        c^{1/2}bI_d,\ c^{3/2}b^3I_d,\ c^{5/2}b^5I_d
    \right),
\]
and hence
\[
    \operatorname{Std}\left(X_{cb^2}^1-X_0^1\right)\asymp b,\qquad
    \operatorname{Std}\left(X_{cb^2}^2-X_0^2\right)\asymp b^3,\qquad
    \operatorname{Std}\left(X_{cb^2}^3-X_0^3\right)\asymp b^5.
\]
To see these powers, freeze the coefficients at a point in a compact
coordinate patch and keep only
the noise path propagated by the chain $r\to p\to\theta$.  Up to nonsingular
constant matrices, the leading stochastic convolutions are
\[
    R_t=\sigma B_t,\qquad
    P_t=\gamma\sigma\int_0^t (t-s)\,dB_s,\qquad
    \Theta_t=\gamma\sigma\int_0^t \frac{(t-s)^2}{2}\,dB_s .
\]
Hence
\[
    \E|R_t|^2\asymp t,\qquad
    \E|P_t|^2\asymp t^3,\qquad
    \E|\Theta_t|^2\asymp t^5 .
\]
The cutoff coefficients have uniformly bounded derivatives and the matrices
$J_{x^1}B_2$ and $J_{x^2}B_3$ are uniformly non-degenerate on the boundary
patches.  The rigorous uniform covariance version of this scaling is
Lemma~\ref{lem:uniform-pigato-covariance}.
\end{remark}

This is also the scale used in the boundary-layer estimate below.
To relate this scale to the normal boundary layer, suppose that $|r_0|\le b$
and set
\[
    \tau_b:=\inf\{s\ge0:|r_s|>2b\}.
\]
For the martingale part $M_t^\rho=2\sigma\int_0^t r_s\cdot dB_s$ of a signed
boundary function and every fixed $c>0$,
\[
    \langle M^\rho\rangle_{cb^2\wedge\tau_b}
    =4\sigma^2\int_0^{cb^2\wedge\tau_b}|r_s|^2\,ds
    \le16c\sigma^2b^4.
\]
Thus the stopped normal martingale has scale $b^2$.  Until exit from a fixed
compact boundary patch, its finite-variation counterpart satisfies
$|A_{cb^2\wedge\tau_b}|\le Cb^2$.  This identifies the layer
\[
    E_b^C(\Sigma)
    =\left\{z:\dist(z,\Sigma)\le Cb^2,\ |r(z)|\le2b\right\},
\]
used in Lemma~\ref{lem:occupation-after-small-r-hit}.  The next lemma quantifies
the Lebesgue volume of this anisotropic layer.

\begin{lemma}[Thin boundary-layer volume]\label{lem:thin-layer-volume}
Let $\Sigma\in\{\partial A,\partial B,\partial\calD\}$ and let
$K_1\subset\R^{3d}$ be compact.  There exist constants $a_0,b_0,C>0$ such that,
for $0<a\le a_0$ and $0<b\le b_0$,
\[
    \operatorname{Leb}\left\{
        y\in K_1:\dist(y,\Sigma)\le a,\ |r(y)|\le b
    \right\}
    \le C a b^d .
\]
\end{lemma}

\begin{proof}
Each boundary component is a sphere of some radius $R_\Sigma>0$, centered at
$(c_\Sigma,0,0)$, where $c_\Sigma=m,s$, or $0$.  Put
$q=(\theta-c_\Sigma,p)\in\R^{2d}$.  Since distance to a sphere is radial, the
set to be estimated is contained in
\[
    \left\{(q,r):
    R_\Sigma-a\le\sqrt{|q|^2+|r|^2}\le R_\Sigma+a,
    \ |r|\le b\right\}.
\]
Choose $a_0,b_0<R_\Sigma/4$.  Write
$\omega_k$ for the volume of the unit ball in $\R^k$:
\[
    \omega_k:=\operatorname{Leb}_{\R^k}(B_1(0))
    =\frac{\pi^{k/2}}{\Gamma(k/2+1)}.
\]
For fixed $r$ with $|r|\le b_0$, the $2d$-dimensional volume of the
corresponding $q$-annulus is
\begin{align*}
    V_{a,r}
    =\omega_{2d}\left(
       \left((R_\Sigma+a)^2-|r|^2\right)^d
       -\left((R_\Sigma-a)^2-|r|^2\right)^d
       \right)
       \le C_\Sigma a.
\end{align*}
Indeed, the last inequality follows from the mean-value theorem, uniformly
for $0<a\le a_0$ and $|r|\le b_0$.  Fubini's theorem now gives
\[
    \operatorname{Leb}\{\dist(y,\Sigma)\le a,\ |r(y)|\le b\}
    \le \int_{|r|\le b}V_{a,r}\,dr
    \le C_\Sigma a\,\omega_d b^d.
\]
Intersecting with $K_1$ can only decrease the volume.  Taking the maximum of
the constants over the three boundary components proves the claim.
\end{proof}

Combining this geometric volume estimate with the localized heat-kernel bound
converts volume smallness into occupation-probability smallness.  This gives
the boundary-layer estimate below.

\begin{lemma}[Localized boundary-layer estimate near $\Gamma_\Sigma$]
\label{lem:boundary-layer-small-r}
Let $M<\infty$, $0<t_0<M$, and let
$\Sigma\in\{\partial A,\partial B,\partial\calD\}$.  Let
$\mathsf O\subset\R^{3d}$ be a bounded smooth open set containing a fixed
neighborhood of $\overline{\calD}$ and the distance collars of
$\partial A$, $\partial B$, and $\partial\calD$ used in the boundary occupation
estimates, and set
\[
    \tau_{\mathsf O^c}^\delta
    =
    \inf\left\{t\ge0:Z_t^\delta\notin\mathsf O\right\}.
\]
For every compact set $K\subset\mathsf O$, there exists a finite constant
$C=C(t_0,M,K,\mathsf O)$ such that, for all $0\le\delta\le1$, all $z\in K$,
and all sufficiently small $a,b>0$,
\[
    \int_{t_0}^{M}
    \Pbb_z^\delta\left(
        t<\tau_{\mathsf O^c}^\delta,\ 
        \dist\left(Z_t^\delta,\Sigma\right)\le a,\ \left|r_t^\delta\right|\le b
    \right) dt
    \le C a b^d .
\]
The same estimate holds for the adjoint processes.  Here $\delta=0$ denotes the
limiting process.
\end{lemma}

\begin{proof}
Choose a compact neighborhood $K_1\subset\mathsf O$ of the relevant boundary
component.  Let $Z^{\delta,\mathsf O}$ be the chain-compatible cutoff process
from Proposition~\ref{prop:localized-cutoff-heat-kernel}.  On the event
$\{t<\tau_{\mathsf O^c}^\delta\}$, the original and cutoff processes agree up
to time $t$.  Hence
\[
\Pbb_z^\delta\left(
    t<\tau_{\mathsf O^c}^\delta,\ Z_t^\delta\in E_{a,b}(\Sigma)
\right)
\le
\Pbb_z^{\delta,\mathsf O}\left(
    Z_t^{\delta,\mathsf O}\in E_{a,b}(\Sigma)
\right),
\]
where
\[
    E_{a,b}(\Sigma)=
    \{y\in K_1:\dist(y,\Sigma)\le a,\ |r(y)|\le b\}
\]
for all sufficiently small $a$.  By
Proposition~\ref{prop:localized-cutoff-heat-kernel}, the cutoff density is
bounded uniformly for $t\in[t_0,M]$, $z\in K$, and $y\in K_1$.  Moreover,
Lemma~\ref{lem:thin-layer-volume} gives $|E_{a,b}(\Sigma)|\le C a b^d$.
Therefore
\[
    \int_{t_0}^M
    \Pbb_z^\delta\left(
        t<\tau_{\mathsf O^c}^\delta,\ 
        Z_t^\delta\in E_{a,b}(\Sigma)
    \right)dt
    \le C(M-t_0)|E_{a,b}(\Sigma)|
    \le C a b^d .
\]
The adjoint case is identical.
\end{proof}

The following elementary estimate rules out reaching a compactly separated
boundary component in an arbitrarily short time, uniformly over the
regularization.

\begin{lemma}[Uniform small-time separation from a compactly separated boundary]
\label{lem:uniform-small-time-boundary}
Let $\Sigma\in\{\partial A,\partial B,\partial\calD\}$ and let
$K\subset\overline{\calD}\setminus\Sigma$ be compact.  Then
\[
    \lim_{t_0\downarrow0}
    \sup_{0\le\delta\le1}\sup_{z\in K}
    \Pbb_z^\delta\left(\tau_\Sigma^\delta\le t_0\right)=0 .
\]
The same estimate holds for the adjoint family.
\end{lemma}

\begin{proof}
Let $a=\dist(K,\Sigma)>0$.  Choose a compact neighborhood $\calK$ of $K$ whose
$a/4$-neighborhood is still disjoint from $\Sigma$.  Up to the exit time from
$\calK$, all drifts and diffusion coefficients of the forward regularized
family are bounded uniformly for $0\le\delta\le1$.  Hence, for $z\in K$,
\[
    \left\{\tau_\Sigma^\delta\le t_0,\ \tau_{\calK^c}^\delta>t_0\right\}
    \subset
    \left\{\sup_{s\le t_0}\left|Z_s^\delta-z\right|\ge a/2\right\}.
\]
Writing $Z_s^\delta-z$ as drift plus martingale, the drift contribution is at
most $Ct_0$ and the martingale $M^\delta$ satisfies, by Burkholder-Davis-Gundy inequality,
\[
    \sup_{0\le\delta\le1}\sup_{z\in K}
    \E_z^\delta\left[\sup_{s\le t_0}\left|M_s^\delta\right|^2\right]\le Ct_0 .
\]
For $t_0$ small enough that $Ct_0\le a/4$, Chebyshev's inequality gives
\[
    \sup_{0\le\delta\le1}\sup_{z\in K}
    \Pbb_z^\delta\left(\tau_\Sigma^\delta\le t_0,\ \tau_{\calK^c}^\delta>t_0\right)
    \le Ct_0/a^2 .
\]
The same estimate, with the distance from $K$ to $\calK^c$, controls the
probability of leaving $\calK$ before time $t_0$.  Letting $t_0\downarrow0$
proves the claim.  The adjoint drifts have the same local boundedness, so the
same proof applies.
\end{proof}

The occupation estimate is converted into a hitting estimate by showing that a
small-$r$ boundary hit forces a short interval of time inside the same thin
layer with positive probability.

\begin{lemma}[Conditional occupation after a small-$r$ boundary hit]
\label{lem:occupation-after-small-r-hit}
Let $\Sigma\in\{\partial A,\partial B,\partial\calD\}$ and let
$\tau=\tau_\Sigma$.  There exist constants $c,q,C>0$ such that, for all
sufficiently small $b\in(0,1)$, on the event
$\{\tau<\infty,\ |r_\tau|\le b\}$,
\[
    \Pbb_z\left(
    Z_{\tau+s}\in
    \{y:\dist(y,\Sigma)\le Cb^2,\ |r(y)|\le2b\}
    \text{ for all }0\le s\le cb^2
    \middle|\mathcal F_\tau
    \right)\ge q .
\]
On the same event one may also require
\[
    \sup_{0\le s\le cb^2}
    \left(|p_{\tau+s}-p_\tau|+|\theta_{\tau+s}-\theta_\tau|\right)
    \le Cb^2 .
\]
For the regularized process, the same statement holds uniformly whenever
$0<\delta\le b^2$.  The constants are common to the three boundary components,
the forward and adjoint processes, the hitting point, and
$0\le\delta\le b^2$.
\end{lemma}

\begin{proof}
Since the three boundary components are compact and have disjoint tubular
neighborhoods, we fix common upper bounds for the forward and adjoint drifts,
their first two derivatives, the derivatives of the quadratic defining
functions, and all It\^o correction terms on these neighborhoods, uniformly for
$0\le\delta\le1$.  All constants below are chosen from these common bounds and
are therefore independent of the boundary component, the hitting point, the
direction of time, and $\delta$.

Choose a distance collar
$U_a=\{z:\operatorname{dist}(z,\Sigma)<a\}$ and let
$\rho=\rho_\Sigma$ be the corresponding quadratic defining function from
Definition~\ref{def:signed-boundary-functions}.  After decreasing $a$ if
necessary, there are constants $c_1,c_2>0$, common to the three boundary
components, such that
\[
    c_1\operatorname{dist}(z,\Sigma)
    \le |\rho_\Sigma(z)|
    \le c_2\operatorname{dist}(z,\Sigma),
    \qquad z\in \overline U_a .
\]
On $\{\tau<\infty,\ |r_\tau|\le b\}$ set
\[
    \kappa=\inf\{s\ge0: Z_{\tau+s}\notin U_a\}.
\]
All estimates below are first made for $Z_{\tau+s\wedge\kappa}$.  Let
$\sigma=\sqrt{2\gamma\eps/L}$ and choose $c>0$ small.  Define
\[
    E_b^r=
    \left\{
    \sup_{0\le s\le cb^2}
    \left|\sigma(B_{\tau+s}-B_\tau)\right|\le \frac b4
    \right\}.
\]
By Brownian scaling and the strong Markov property,
\[
    \Pbb(E_b^r\mid\mathcal F_\tau)
    =
    \Pbb\left(
        \sup_{0\le u\le c}|\sigma B_u|\le\frac14
    \right)
    =:q_r>0,
\]
with $q_r$ independent of $b$, $\tau$, and the hitting point.  On $E_b^r$,
while $s\le cb^2\wedge\kappa$, the bounded drift in the $r$-equation has size
$\mathcal O(b^2)$ and the Brownian increment is at most $b/4$; after decreasing $c$ and
then $b_0$ if necessary,
\begin{equation}\label{eq:small-r-stopped-r}
    \sup_{s\le cb^2\wedge\kappa}|r_{\tau+s}|\le2b .
\end{equation}
The bounded drifts and the link estimates then give
\begin{equation}\label{eq:small-r-stopped-ptheta}
    \sup_{s\le cb^2\wedge\kappa}
    \left(|p_{\tau+s}-p_\tau|+|\theta_{\tau+s}-\theta_\tau|\right)\le Cb^2 .
\end{equation}
It\^o's formula applied to $\rho(Z_{\tau+s\wedge\kappa})$ gives
\[
    \rho(Z_{\tau+s\wedge\kappa})
    =
    A_{s\wedge\kappa}
    +2\sigma\int_0^{s\wedge\kappa}r_{\tau+u}\cdot dB_{\tau+u}.
\]
Here $|A_{s\wedge\kappa}|\le Cb^2$ for $0\le s\le cb^2$, using
\eqref{eq:small-r-stopped-r}--\eqref{eq:small-r-stopped-ptheta} and bounded
coefficients in $U_a$.  Integration by parts and the Brownian bound on $E_b^r$
also give
\[
    \sup_{s\le cb^2\wedge\kappa}
    \left|\int_0^s r_{\tau+u}\cdot dB_{\tau+u}\right|\le Cb^2 .
\]
Consequently, on $E_b^r$,
\begin{equation}\label{eq:small-r-stopped-rho}
    \sup_{s\le cb^2\wedge\kappa}|\rho(Z_{\tau+s})|\le Cb^2 .
\end{equation}
Choose $b_0>0$ so small that $Cb_0^2<c_1a/2$.  If $b\le b_0$ and
$\kappa\le cb^2$, then
\[
    \operatorname{dist}(Z_{\tau+\kappa},\Sigma)=a,
\]
and hence $|\rho_\Sigma(Z_{\tau+\kappa})|\ge c_1a$, contradicting
\eqref{eq:small-r-stopped-rho}.  Thus, on $E_b^r$, $\kappa>cb^2$.
Removing the stopping gives
\[
    |\rho_\Sigma(Z_{\tau+s})|\le Cb^2,\qquad
    \dist(Z_{\tau+s},\Sigma)\le Cb^2,\qquad
    |r_{\tau+s}|\le2b,\qquad 0\le s\le cb^2,
\]
and also the $p,\theta$ increment bound
\eqref{eq:small-r-stopped-ptheta}.  This proves the limiting
conditional estimate with $q=q_r$.

For the regularized process write $\tau^\delta=\tau_\Sigma^\delta$ and
\[
    \kappa^\delta=\inf\left\{s\ge0:Z^\delta_{\tau^\delta+s}\notin U_a\right\}.
\]
On $\{\tau^\delta<\infty,\ |r^\delta_{\tau^\delta}|\le b\}$ use the same
$r$-Brownian event $E_b^r$, with $\tau$ replaced by $\tau^\delta$, and add
\[
    E_b^{p,\theta}=
    \left\{
    \sqrt{2\eps\delta}\sup_{s\le cb^2}
    \left(\left|W^p_{\tau^\delta+s}-W^p_{\tau^\delta}\right|
    +\left|W^\theta_{\tau^\delta+s}-W^\theta_{\tau^\delta}\right|\right)
    \le c_1b^2
    \right\}.
\]
If $\delta\le b^2$, Brownian scaling gives the uniform lower bound
\[
\begin{aligned}
\inf_{0<\delta\le b^2}
\Pbb\left(E_b^{p,\theta}\mid\mathcal F_{\tau^\delta}\right)
&=
\inf_{0<\delta\le b^2}
\Pbb\left(
    \sqrt{2\eps}\frac{\sqrt\delta}{b}
    \sup_{0\le u\le c}\left(|W_u^p|+\left|W_u^\theta\right|\right)
    \le c_1
\right) \\
&\quad\ge
\Pbb\left(
    \sqrt{2\eps}
    \sup_{0\le u\le c}\left(|W_u^p|+\left|W_u^\theta\right|\right)
    \le c_1
\right)
=:q_{p,\theta}>0 .
\end{aligned}
\]
Thus
\[
    \Pbb_z^\delta\left(E_b^r\cap E_b^{p,\theta}\mid\mathcal F_{\tau^\delta}\right)
    \ge q_rq_{p,\theta}=:q_1>0,
\]
after choosing $c,c_1$ fixed, with $q_1$ independent of $b$, $\delta$,
$\tau^\delta$, and the boundary point.  Repeating the stopped estimates on
$[0,cb^2\wedge\kappa^\delta]$
yields that
\[
    \sup_{s\le cb^2\wedge\kappa^\delta}\left|r^\delta_{\tau^\delta+s}\right|\le2b,
\qquad
    \sup_{s\le cb^2\wedge\kappa^\delta}
    \left(\left|p^\delta_{\tau^\delta+s}-p^\delta_{\tau^\delta}\right|
    +\left|\theta^\delta_{\tau^\delta+s}-\theta^\delta_{\tau^\delta}\right|\right)\le Cb^2,
\]
where the added $p$- and $\theta$-martingales are controlled by
$E_b^{p,\theta}$.
It\^o's formula for $\rho(Z^\delta_{\tau^\delta+s\wedge\kappa^\delta})$ has the
same $r$-martingale as above and an extra martingale whose supremum is bounded
by $Cb^2$ on $E_b^{p,\theta}$; the regularizing It\^o correction is bounded by
$C\delta cb^2\le Cb^4$. 
Hence,
\begin{equation}\label{eq:small-r-regularized-rho}
    \sup_{s\le cb^2\wedge\kappa^\delta}
    \left|\rho\left(Z^\delta_{\tau^\delta+s}\right)\right|\le Cb^2 .
\end{equation}
Decrease $b_0$ if necessary so that $Cb_0^2<c_1a/2$.  Then
\eqref{eq:small-r-regularized-rho} and the same exit argument imply that
$\kappa^\delta>cb^2$ on
$E_b^r\cap E_b^{p,\theta}$: otherwise
\[    \operatorname{dist}\left(Z^\delta_{\tau^\delta+\kappa^\delta},\Sigma\right)=a,
    \qquad    \left|\rho_\Sigma\left(Z^\delta_{\tau^\delta+\kappa^\delta}\right)\right|\ge c_1a,
\]
contradicting the stopped $\rho$-bound.  The unstopped regularized occupation
estimate follows with a constant $q'>0$ uniform in
$0<\delta\le b^2$.  Taking the minimum of the limiting and regularized
constants gives the common constant $q$ in the statement.  The adjoint processes
have the same boundedness and noise structure, so the proof is identical.
\end{proof}

Combining the preceding occupation and post-hit estimates gives the key
boundary-stability condition: first entries with small $r$ are uniformly
negligible.

\begin{proposition}[Uniform small-$r$ boundary entry estimate]\label{prop:small-r-entry}
Let $\Sigma\in\{\partial A,\partial B,\partial\calD\}$, $M<\infty$, and
$K\subset\overline\calD\setminus\Sigma$ be compact.  Then
\begin{equation}\label{eq:small-r-entry-limit}
    \lim_{b\downarrow0}
    \sup_{z\in K}
    \Pbb_z\left(|r_{\tau_\Sigma}|\le b,\ \tau_\Sigma\le M\right)=0 .
\end{equation}
Moreover, for every $\lambda\in(0,1/2)$ and every fixed $K_0>0$, with
$b_\delta=\delta^{\lambda/4}$,
\[
    \sup_{z\in K}
    \Pbb_z^\delta\left(
        |r^\delta_{\tau_\Sigma^\delta}|\le K_0b_\delta,\ 
        \tau_\Sigma^\delta\le M
    \right)\longrightarrow0
    \qquad\text{as }\delta\downarrow0 .
\]
The same statements hold for the adjoint processes.
\end{proposition}

\begin{proof}
We prove the forward estimate; the adjoint case is identical.  By
Lemma~\ref{lem:uniform-small-time-boundary},
\[
    \lim_{t_0\downarrow0}
    \sup_{0\le\delta\le1}\sup_{z\in K}
    \Pbb_z^\delta\left(\tau_\Sigma^\delta\le t_0\right)=0 .
\]
Fix a bounded smooth open set $\mathsf O$ containing a neighborhood of
$\overline{\calD}$ and the boundary distance collars used in
Lemma~\ref{lem:occupation-after-small-r-hit}.  Let
\[
    \tau_{\mathsf O^c}^\delta
    =
    \inf\left\{t\ge0:Z_t^\delta\notin\mathsf O\right\},
    \qquad
    \tau_{\mathsf O^c}=\tau_{\mathsf O^c}^0 .
\]
The argument applies to arbitrary hits of $\Sigma$; the exit term
$\tau_{\mathsf O^c}$ accounts for possible excursions away from the bounded
region.

For the limiting process, set
\[
    E_b=\left\{|r_{\tau_\Sigma}|\le b,\ t_0<\tau_\Sigma\le M\right\}.
\]
Introduce the pre-hit localization
\[
    E_b^{\mathsf O}
    :=E_b\cap\{\tau_\Sigma<\tau_{\mathsf O^c}\}.
\]
The event $E_b^{\mathsf O}$ belongs to $\mathcal F_{\tau_\Sigma}$.  Moreover,
the post-hit tube in Lemma~\ref{lem:occupation-after-small-r-hit} is contained
in $\mathsf O$, because $\mathsf O$ contains the fixed boundary collar.  We may
therefore apply the conditional lower bound at $\tau_\Sigma$ and then use
Tonelli's theorem to obtain
\begin{align}
    qcb^2 \Pbb_z\left(E_b^{\mathsf O}\right)
    &\le
    \E_z\left[
    {\bfone}_{E_b^{\mathsf O}}
    \int_{\tau_\Sigma}^{\tau_\Sigma+cb^2}
    {\bfone}_{\{t<\tau_{\mathsf O^c},\
        \dist(Z_t,\Sigma)\le Cb^2,\ |r_t|\le2b\}}dt
    \right]\nonumber \\
    &\le
    \int_{t_0}^{M+cb^2}
    \Pbb_z\left(
        t<\tau_{\mathsf O^c},\ 
        \dist(Z_t,\Sigma)\le Cb^2,\ |r_t|\le2b
    \right) dt.\label{eqn:RHS}
\end{align}
Applying Lemma~\ref{lem:boundary-layer-small-r} with $M+1$ in place of $M$ and
$a=Cb^2$ gives, uniformly in the boundary component and in the starting point,
\[
\int_{t_0}^{M+cb^2}
    \Pbb_z\left(
        t<\tau_{\mathsf O^c},\
        \dist(Z_t,\Sigma)\le Cb^2,\ |r_t|\le2b
    \right)dt
    \le Cb^{d+2}.
\]
Thus the full constant chain is
\[
    qcb^2
    \Pbb_z\left(E_b^{\mathsf O}\right)
    \le Cb^{d+2}.
\]
After dividing by $b^2$, we get
\[
    \sup_{z\in K}
    \Pbb_z\left(E_b^{\mathsf O}\right)\le Cb^d .
\]
Since $\Sigma$ is contained in the interior of $\mathsf O$,
$E_b\setminus E_b^{\mathsf O}$ implies
$\tau_{\mathsf O^c}\le\tau_\Sigma\le M$.  Thus, for all sufficiently small
$b$,
\[
    \sup_{z\in K}\Pbb_z(E_b)
    \le
    Cb^d+
    \sup_{z\in K}\Pbb_z\left(\tau_{\mathsf O^c}\le M\right).
\]
Letting $b\downarrow0$ gives an error controlled by the exit probability from
$\mathsf O$.

For the regularized process, set
$b_\delta=\delta^{\lambda/4}$ and
$\widetilde b_\delta=K_0b_\delta$.  Since $\lambda/2<1$ and $K_0>0$ is fixed,
we have $\delta\le\widetilde b_\delta^2$ for all sufficiently small
$\delta$.  Thus Lemma~\ref{lem:occupation-after-small-r-hit} applies to the
regularized process with $b=\widetilde b_\delta$.  Define the
$\mathcal F_{\tau_\Sigma^\delta}$-measurable event
\[
    E_\delta^{\mathsf O}
    :=
    \left\{
        |r^\delta_{\tau_\Sigma^\delta}|\le\widetilde b_\delta,\
        \ t_0<\tau_\Sigma^\delta\le M,\
        \ \tau_\Sigma^\delta<\tau_{\mathsf O^c}^\delta
    \right\}.
\]
Repeating the preceding conditional occupation estimate gives
\[
    \sup_{z\in K}
    \Pbb_z^\delta\left(E_\delta^{\mathsf O}\right)
    \le C\widetilde b_\delta^d
    =CK_0^d\delta^{\lambda d/4}
    \rightarrow0,\qquad\text{as $\delta\downarrow0$}.
\]
The complement of $E_\delta^{\mathsf O}$ inside the event with
$t_0<\tau_\Sigma^\delta\le M$ is controlled by
$\{\tau_{\mathsf O^c}^\delta\le M\}$.  Consequently,
\[
\begin{aligned}
&\limsup_{\delta\downarrow0}\sup_{z\in K}
    \Pbb_z^\delta\left(
        |r^\delta_{\tau_\Sigma^\delta}|\le K_0b_\delta,\ 
        \tau_\Sigma^\delta\le M
    \right)
\\
&\quad\le
    \sup_{0\le\delta\le1}\sup_{z\in K}
    \Pbb_z^\delta\left(\tau_\Sigma^\delta\le t_0\right)
    +
    \sup_{0\le\delta\le1}\sup_{z\in K}
    \Pbb_z^\delta\left(\tau_{\mathsf O^c}^\delta\le M\right).
\end{aligned}
\]
The same bound, with the first term restricted to $\delta=0$, holds for the
limiting estimate after taking $b\downarrow0$.

By the local moment bound in
Assumption~\ref{ass:controlled-coefficients}, choosing $\mathsf O$ large makes
the exit term
\[
    \sup_{0\le\delta\le1}\sup_{z\in K}
    \Pbb_z^\delta\left(\tau_{\mathsf O^c}^\delta\le M\right)
\]
arbitrarily small.  Then Lemma~\ref{lem:uniform-small-time-boundary} sends the
small-time term to zero as $t_0\downarrow0$.  This proves both forward
estimates.  Since there are only three boundary components, the same estimate
may be summed over $\partial A$, $\partial B$, and $\partial\calD$ whenever the
first hit of the boundary union is used.  The adjoint proof uses the adjoint
local moment bound and is identical.
\end{proof}

Proposition~\ref{prop:small-r-entry} shows that entries near the
characteristic set $\{r=0\}$ are negligible.  Once the boundary is approached
away from that set, the following
collar estimate shows that a nearby process reaches the same killed side
quickly.  Together with the chain-compatible density estimate used in
Lemma~\ref{lem:boundary-layer-small-r}, this is the propagation step replacing
any normal-trace argument: the $r\to p\to\theta$ H\"ormander chain controls the
probability of near-characteristic entries, while non-characteristic entries
leave a macroscopic boundary layer on the fast collar scale.

\begin{lemma}[Hypoelliptic propagation control from a non-characteristic boundary layer]
\label{lem:fast-entry-collar}
Let $\lambda\in(0,1/2)$.  For $\alpha>0$, define
\[
    C_{\alpha,\lambda}
    :=
    \left\{z\in\calD\setminus\calC:
        \dist(z,\calC)\le \alpha^\lambda,\quad
        |r|\ge \alpha^{\lambda/4}
      \right\}.
\]
Then
\[
    \sup_{z\in C_{\alpha,\lambda}}
    \Pbb_z\left(\tau_\calC>\alpha^\lambda\right)\to0
    \qquad\text{as }\alpha\downarrow0,
\]
and
\[
    \sup_{z\in C_{\alpha,\lambda}}
    \Pbb_z^\alpha\left(\tau_\calC^\alpha>\alpha^\lambda\right)
    \longrightarrow0
    \qquad\text{as }\alpha\downarrow0 .
\]
The same estimates hold for the adjoint processes.
\end{lemma}

\begin{proof}
Choose the signed defining function $\rho$ of the closest component of
$\calC$.  For $\partial B$ we use
\[
    \rho_s(z)=|z-(s,0,0)|^2-\rho^2,
\]
while for $\partial\calD$ we use $-\rho_N(z)=N^2-|z|^2$.  With this convention,
the killed side is always $\{\rho\le0\}$.  The tubular neighborhoods of
$\partial B$ and $\partial\calD$ are disjoint, so this choice is unambiguous for
small $\alpha$.

Fix a smaller tubular neighborhood $\mathsf T_\Sigma$ of the selected component
$\Sigma$, with closure disjoint from the other killed component.  For the
limiting process set
\[
    \zeta_\alpha
    :=\inf\{t\ge0:Z_t\notin\mathsf T_\Sigma\},
    \qquad
    \kappa_\alpha:=\tau_\calC\wedge\zeta_\alpha .
\]
For the regularized process the same symbols, with a superscript $\alpha$ on
the process and stopping times, will be used.  All coefficient bounds below
are taken on the compact closure of $\mathsf T_\Sigma$ and are uniform in the
two components, the forward and adjoint drifts, and $0\le\alpha\le1$.

Set
\[
    T_\alpha=\alpha^\lambda,\qquad
    v_\alpha=\alpha^{\lambda/4},\qquad
    d_\alpha=\alpha^\lambda .
\]
For $z\in C_{\alpha,\lambda}$ we have
$\dist(z,\calC)\le d_\alpha$ and $|r_0|\ge v_\alpha$.  Since
$z\in\calD\setminus\calC$ and the killed side is $\{\rho\le0\}$, the interior
side satisfies $\rho(z)\ge0$.  In the fixed tubular neighborhood, the absolute
value of $\rho$ is comparable to the distance from the boundary.  Hence, for all
sufficiently small $\alpha$,
\[
    0\le\rho(z)\le C\dist(z,\calC)\le Cd_\alpha .
\]
Up to $\kappa_\alpha$, It\^o's formula for the limiting process or the
$\alpha$-regularized process gives
\[
    \rho(Z_{t\wedge\kappa_\alpha})-\rho(Z_0)
    =
    A_t+M_t+R_t^\alpha,
\]
where
$|A_t|\le Ct$,
and $M$ is the one-dimensional martingale obtained from the $r$-noise term; up
to a harmless sign,
\[
    M_t=2\sigma\int_0^{t\wedge\kappa_\alpha} r_s\cdot dB_s .
\]
The lower bound on $|r|$ supplies stochastic non-degeneracy through the
quadratic variation
\[
    d\langle M\rangle_t
    =4\sigma^2|r_t|^2\mathbf 1_{\{t<\kappa_\alpha\}}\,dt .
\]
Here $R^\alpha=0$ for the limiting process, while for the
$\alpha$-regularized process
\[
    R_t^\alpha
    =
    \sqrt{2\eps\alpha}
    \int_0^{t\wedge\kappa_\alpha^\alpha}
        \nabla_{\theta,p}\rho(Z_s^\alpha)\cdot dW_s .
\]

We first justify the localization.  On
$\{\zeta_\alpha\le T_\alpha<\tau_\calC\}$, the path leaves
$\mathsf T_\Sigma$ through its interior, non-killed side.  Consequently,
$\rho(Z_{\zeta_\alpha})\ge c_*>0$, where $c_*$ is independent of $\alpha$.
For sufficiently small $\alpha$, at least half of this fixed increase must be
produced by the stopped martingale.  The quadratic variation of
the full stopped martingale $M+R^\alpha$ is bounded by
$C(T_\alpha+\alpha T_\alpha)$.  Doob's inequality therefore gives, for both the
limiting and regularized processes,
\begin{equation}
    \sup_{z\in C_{\alpha,\lambda}}
    \Pbb_z\left(\zeta_\alpha\le T_\alpha<\tau_\calC\right)
    \le CT_\alpha\longrightarrow0.
    \label{eq:fast-entry-outward-exit}
\end{equation}
with the evident superscript-$\alpha$ version for the regularized law.

Next define
\[
    E_\alpha^r
    :=
    \left\{
      \sup_{s\le T_\alpha\wedge\kappa_\alpha}|r_s-r_0|
      \le\frac12v_\alpha
    \right\}.
\]
The stopped $r$-drift is bounded, and the stopped $r$-martingale has quadratic
variation bounded by $CT_\alpha$.  The Burkholder--Davis--Gundy inequality and
Chebyshev's inequality yield, uniformly over the collar,
\begin{equation}
    \Pbb_z\left((E_\alpha^r)^c\right)
    \le
    C\frac{T_\alpha+T_\alpha^2}{v_\alpha^2}
    \le C\alpha^{\lambda/2}\longrightarrow0.
    \label{eq:fast-entry-r-good}
\end{equation}
The same estimate holds for the regularized and adjoint processes.  On
$E_\alpha^r\cap\{\kappa_\alpha>T_\alpha\}$,
\[
    \langle M\rangle_{T_\alpha}
    \ge c v_\alpha^2T_\alpha
    =:Q_\alpha
    =c\alpha^{3\lambda/2}.
\]
By Dambis--Dubins--Schwarz and the reflection principle,
\begin{equation}
\begin{aligned}
&\Pbb_z\left(
    \inf_{0\le t\le T_\alpha}M_t>-C_1d_\alpha,\
    E_\alpha^r,\ \kappa_\alpha>T_\alpha
    \right) \\
&\qquad\le
    \Pbb\left(
       \inf_{0\le u\le Q_\alpha}\widetilde B_u>-C_1d_\alpha
    \right)
    \le C\frac{d_\alpha}{\sqrt{Q_\alpha}}
    =C\alpha^{\lambda/4}\longrightarrow0.
\end{aligned}
\label{eq:fast-entry-reflection}
\end{equation}
Here $\widetilde B$ is the Brownian motion supplied by the time-change theorem.
This form of the estimate explicitly retains the event on which the lower
quadratic-variation bound is valid.

For the regularized process, boundedness of
$\nabla_{\theta,p}\rho$ on the stopped tube gives
\[
    \langle R^\alpha\rangle_{T_\alpha}
    \le C\alpha T_\alpha=C\alpha^{1+\lambda},
\]
and hence
\begin{equation}
    \sup_{z\in C_{\alpha,\lambda}}
    \Pbb_z^\alpha\left(
       \sup_{t\le T_\alpha}|R_t^\alpha|>d_\alpha
    \right)
    \le C\frac{\alpha T_\alpha}{d_\alpha^2}
    =C\alpha^{1-\lambda}\longrightarrow0.
    \label{eq:fast-entry-extra-noise}
\end{equation}

On the event that the process survives until $T_\alpha$, stays in the tube,
belongs to $E_\alpha^r$, and (in the regularized case) satisfies
$\sup_{t\le T_\alpha}|R_t^\alpha|\le d_\alpha$, one has
$\rho(Z_t)\ge0$ for $t\le T_\alpha$.  Since
$\rho(Z_0)\le Cd_\alpha$ and
$\sup_{t\le T_\alpha}|A_t|\le CT_\alpha=Cd_\alpha$, the stopped decomposition
implies
\[
    \inf_{0\le t\le T_\alpha}M_t>-C_1d_\alpha.
\]
Combining \eqref{eq:fast-entry-outward-exit}--
\eqref{eq:fast-entry-extra-noise} therefore yields
\[
    \sup_{z\in C_{\alpha,\lambda}}
    \left[
      \Pbb_z\left(\tau_\calC>T_\alpha\right)
      +\Pbb_z^\alpha\left(\tau_\calC^\alpha>T_\alpha\right)
    \right]
    \le
    C\left(
       \alpha^\lambda+\alpha^{\lambda/2}
       +\alpha^{\lambda/4}+\alpha^{1-\lambda}
    \right)
    \longrightarrow0.
\]
The adjoint proof is identical, and the same uniform estimates may be applied
after stopping times by the strong Markov property.
\end{proof}

The next proposition transfers pathwise convergence and the boundary estimates
into uniform convergence of truncated killing times for starts on $\partial A$.

\begin{proposition}[Uniform convergence of truncated killing times]
\label{prop:uniform-truncated-killing}
For every $M<\infty$, as $\delta\downarrow0$,
\begin{equation}\label{eq:truncated-killing-L1}
    \sup_{z\in\partial A}
    \E_z\left[
        \left|(\tau_\calC^\delta\wedge M)-(\tau_\calC\wedge M)\right|
    \right]
    \longrightarrow0 .
\end{equation}
The same statement holds for the adjoint processes.
\end{proposition}

\begin{proof}
Couple $Z^\delta$ and $Z$ by the same $r$-Brownian motion.  Fix
$\lambda\in(0,1/2)$, and set
\[
    a_\delta=\delta^\lambda,
    \qquad
    b_\delta=\delta^{\lambda/4}.
\]
Let
\[
    \tau_*^\delta=\tau_\calC\wedge\tau_\calC^\delta,
    \qquad
    G_\delta=
    \left\{
    \sup_{0\le s\le (M+1)\wedge\tau_*^\delta}
    |Z_s^\delta-Z_s|\le a_\delta
    \right\}.
\]
All stopping times in this proof are taken with respect to the joint filtration
generated by the Brownian motions used in the coupling; after such a stopping
time, the post-stopping Brownian increments have the usual strong Markov
property for the coupled construction.
Before time $\tau_*^\delta$, both paths remain in the bounded set
$\overline\calD$.  Since $\partial A$ is compact, choose one fixed compact
neighborhood $\calK$ of $\overline\calD$ and apply the $L^2$ estimate in
Lemma~\ref{lem:pathwise} with $K=\partial A$ and $T=M+1$.  The stopped supremum
defining $G_\delta$ is bounded by the supremum up to the exit time from this
compact neighborhood, and hence
\[
    \sup_{z\in\partial A}
    \E_z\left[
        \sup_{0\le s\le (M+1)\wedge\tau_*^\delta}
        |Z_s^\delta-Z_s|^2
    \right]
    \le C_M\delta .
\]
Therefore, Chebyshev's inequality implies that
\[
    \sup_{z\in\partial A}
    \Pbb_z(G_\delta^c)
    \le
    a_\delta^{-2}C_M\delta
    =
    C_M\delta^{1-2\lambda}
    \longrightarrow0 ,\qquad\text{as $\delta\downarrow 0$},
\]
because $\lambda<1/2$.

For the first direction, on $G_\delta$, if $\tau_\calC\le M$,
$|r_{\tau_\calC}|>2b_\delta$, and the regularized path remains alive at
time $\tau_\calC$, then
\[
    \dist\left(Z_{\tau_\calC}^\delta,\calC\right)\le a_\delta,
    \qquad
    |r_{\tau_\calC}^\delta|\ge b_\delta .
\]
Thus $Z_{\tau_\calC}^\delta\in C_{\delta,\lambda}$, and the strong Markov
property gives
\[
\begin{aligned}
&
    \Pbb_z\left(
        \tau_\calC^\delta>\tau_\calC+2a_\delta,\ \tau_\calC\le M
    \right)
\\
&\le
    \Pbb_z(G_\delta^c)
    +\Pbb_z\left(|r_{\tau_\calC}|\le2b_\delta,\ \tau_\calC\le M\right)
    +\sup_{w\in C_{\delta,\lambda}}
      \Pbb_w^\delta\left(\tau_\calC^\delta>a_\delta\right).
\end{aligned}
\]
The small-$r$ term is controlled by
Proposition~\ref{prop:small-r-entry}, applied to the first hit of
$\partial B$ or $\partial\calD$, and the collar term is controlled by
Lemma~\ref{lem:fast-entry-collar}.

For the reverse direction, on $G_\delta$, if $\tau_\calC^\delta\le M$,
$|r_{\tau_\calC^\delta}^\delta|>2b_\delta$, and the limiting path remains alive
at $\tau_\calC^\delta$, then
$Z_{\tau_\calC^\delta}\in C_{\delta,\lambda}$. Hence,
\[
\begin{aligned}
&
    \Pbb_z\left(
        \tau_\calC>\tau_\calC^\delta+2a_\delta,\ 
        \tau_\calC^\delta\le M
    \right)
\\
&\le
    \Pbb_z(G_\delta^c)
    +\Pbb_z^\delta\left(
        |r_{\tau_\calC^\delta}^\delta|\le2b_\delta,\ 
        \tau_\calC^\delta\le M\right)
    +\sup_{w\in C_{\delta,\lambda}}
      \Pbb_w(\tau_\calC>a_\delta).
\end{aligned}
\]
The second term is the diagonal regularized estimate in
Proposition~\ref{prop:small-r-entry}, and the last term is the limiting part of
Lemma~\ref{lem:fast-entry-collar}.

Consequently,
\[
    \sup_{z\in\partial A}
    \Pbb_z\left(
        \left|(\tau_\calC^\delta\wedge M)-(\tau_\calC\wedge M)\right|>2a_\delta
    \right)\longrightarrow0
    \qquad\text{as }\delta\downarrow0 .
\]
Since the difference is bounded by $M$,
\[
    \E_z\left|
        (\tau_\calC^\delta\wedge M)-(\tau_\calC\wedge M)
    \right|
    \le
    2a_\delta+
    M\Pbb_z\left(
        |(\tau_\calC^\delta\wedge M)-(\tau_\calC\wedge M)|>2a_\delta
    \right).
\]
Taking the supremum over $z\in\partial A$ proves
\eqref{eq:truncated-killing-L1}.
The proof for the adjoint process is identical.
\end{proof}

The same stability mechanism also leads to continuity of the limiting truncated
killing functional along the entrance boundary.

\begin{lemma}[Continuity of the truncated killing functional on $\partial A$]
\label{lem:gM-continuity-boundary}
For every $M<\infty$, the function
\[
    g_M(z)=\E_z[\tau_\calC\wedge M]
\]
is continuous on $\partial A$.
\end{lemma}

\begin{proof}
Let $z_n,z\in\partial A$ with $z_n\to z$, and couple the limiting processes
started from these points by the same Brownian motion.  Let $\mathsf K$ be a
compact neighborhood of $\overline\calD$, and stop the coupled paths at the
first time either path leaves $\mathsf K$.
The limiting third-order Langevin equation has a constant diffusion matrix,
or equivalently additive noise.  Hence, under this synchronous coupling, the
Brownian terms cancel in $Z^{z_n}-Z^z$.  Since the drift is Lipschitz on
$\mathsf K$, Gr\"onwall's lemma implies
\[
    \E\left[
        \sup_{0\le s\le (M+1)\wedge\sigma_n}
        |Z_s^{z_n}-Z_s^z|^2
    \right]
    \le C_M|z_n-z|^2,
\]
where $\sigma_n$ is this joint exit time.  Fix $\lambda\in(0,1/2)$.  After
discarding indices for which $z_n=z$, set
\[
    e_n:=|z_n-z|,\qquad
    a_n:=e_n^{1/2},\qquad
    \alpha_n:=a_n^{1/\lambda},\qquad
    b_n:=\alpha_n^{\lambda/4}=a_n^{1/4}.
\]
Thus $a_n=\alpha_n^\lambda$ and $e_n^2/a_n^2=e_n\to0$.  Write
\[
    \tau_n:=\tau_\calC(Z^{z_n}),\qquad
    \tau:=\tau_\calC(Z^z),\qquad
    \tau_{*,n}:=\tau_n\wedge\tau,
\]
and define
\[
    G_n
    :=\left\{
       \sup_{0\le s\le(M+1)\wedge\tau_{*,n}}
       |Z_s^{z_n}-Z_s^z|\le a_n
    \right\}.
\]
Before $\tau_{*,n}$ both paths remain in the bounded set
$\overline\calD\setminus B$.  The preceding synchronous-coupling estimate and
Chebyshev's inequality therefore give
\[
    \Pbb(G_n^c)\le C_M\frac{e_n^2}{a_n^2}\longrightarrow0.
\]

On $G_n$, whenever one path hits $\calC$ by time $M$ with
$|r|>2b_n$ while the other path remains alive, the latter lies in
$C_{\alpha_n,\lambda}$.  Applying the strong Markov property at the earlier
hitting time in both directions gives
\begin{align*}
&\Pbb\left(
   |(\tau_n\wedge M)-(\tau\wedge M)|>2a_n
 \right) \\
&\quad\le
 \Pbb(G_n^c)
 +\Pbb_z\left(|r_\tau|\le2b_n,\ \tau\le M\right)
 +\Pbb_{z_n}\left(|r_{\tau_n}|\le2b_n,\ \tau_n\le M\right)\\
&\qquad
 +2\sup_{w\in C_{\alpha_n,\lambda}}
      \Pbb_w\left(\tau_\calC>a_n\right).
\end{align*}
The two small-$r$ terms tend to zero uniformly for starting points on
$\partial A$ by Proposition~\ref{prop:small-r-entry}, after summing over
$\partial B$ and $\partial\calD$.  The last term tends to zero by the limiting
part of Lemma~\ref{lem:fast-entry-collar}.  Hence
$\tau_n\wedge M\to\tau\wedge M$ in probability.  These variables are bounded
by $M$, so their expectations converge, proving continuity of $g_M$.
\end{proof}

The next deterministic lemma isolates the pathwise topological fact behind the
stochastic first-hit convergence.

\begin{lemma}[Stability of the first hitting time and label]
\label{lem:hitting-decision-stability}
Let $x^n,x\in C([0,M];\R^{3d})$ and suppose $x^n\to x$ uniformly.  
Suppose the
first hit of $\partial A\cup\partial B\cup\partial\calD$ by $x$ occurs at a
time $\tau\in(0,M)$, at a unique boundary component, at a point with $r\ne0$,
and the path crosses that boundary immediately after the hit.  Let $\tau_n$ be
the first hitting time of $\partial A\cup\partial B\cup\partial\calD$ by
$x^n$.  Then $\tau_n\to\tau$, and, for all large $n$, the first-hit label of
$x^n$ agrees with that of $x$.
\end{lemma}

\begin{proof}
Let $\Gamma=\partial A\cup\partial B\cup\partial\calD$, and let $\Sigma$ be the
unique component hit by $x$ at time $\tau$.  For any
$0<\eta<\tau$, the compact set $x([0,\tau-\eta])$ is disjoint from $\Gamma$;
therefore
\[
    d_\eta:=\dist\left(x([0,\tau-\eta]),\Gamma\right)>0 .
\]
Uniform convergence implies $x^n([0,\tau-\eta])\cap\Gamma=\varnothing$ for all
large $n$, and hence $\tau_n\ge\tau-\eta$.  Letting $\eta\downarrow0$ yields
\[
    \liminf_{n\to\infty}\tau_n\ge\tau .
\]

For the reverse inequality, choose $\eta>0$ so small that
$\tau+\eta<M$ and $x([\tau-\eta,\tau+\eta])$ lies in a tubular neighborhood of
$\Sigma$ which is disjoint from the tubular neighborhoods of the other boundary
components.  Let $\rho_\Sigma$ be a signed defining function for $\Sigma$,
oriented so that the crossed side has positive sign.  Since $x$ crosses
immediately after $\tau$, there exists $s_\eta\in(0,\eta)$ such that
$\rho_\Sigma(x_{\tau+s_\eta})$ has the crossed-side sign and is non-zero.  By
boundary avoidance before $\tau$, after decreasing $\eta$ if necessary,
$\rho_\Sigma(x_{\tau-\eta})$ has the pre-hit sign and is non-zero.  Uniform
convergence implies that, for all large $n$, $x^n$ remains away from the other
boundary components on $[\tau-\eta,\tau+s_\eta]$ and changes the sign of
$\rho_\Sigma$ between $\tau-\eta$ and $\tau+s_\eta$.  By continuity, $x^n$ hits
$\Sigma$ in this interval.  Therefore,
\[
    \limsup_{n\to\infty}\tau_n\le\tau+s_\eta\le\tau+\eta .
\]
Letting $\eta\downarrow0$ implies that $\tau_n\to\tau$.  The preceding argument also
shows that, for all large $n$, the first boundary hit is on the same component
$\Sigma$, so the first-hit label agrees with that of $x$.
\end{proof}

The following tail estimate lets the pointwise first-hit convergence be proved
on finite horizons and then extended to the full committor.

\begin{lemma}[Finite-time reduction for first-hit decisions]
\label{lem:first-hit-tail-reduction}
Suppose Assumptions~\ref{ass:U} and~\ref{ass:controlled-coefficients} hold,
and adopt the \hyperref[setup:domains]{bounded-domain setup}.
Let $K\subset\Omega$ be compact.  Then
\[
    \lim_{M\to\infty}
    \limsup_{\delta\downarrow0}
    \sup_{z\in K}
    \Pbb_z^\delta\left(
        \tau_A^\delta\wedge\tau_\calC^\delta>M
    \right)=0 .
\]
Moreover,
\[
    \lim_{M\to\infty}
    \sup_{z\in K}
    \Pbb_z\left(
        \tau_A\wedge\tau_\calC>M
    \right)=0 .
\]
The same statements hold for the adjoint family.
\end{lemma}

\begin{proof}
Since $\tau_A^\delta\wedge\tau_\calC^\delta\le\tau_\calC^\delta$, it is enough
to prove a uniform tail for $\tau_\calC^\delta$.  Let
$\{O_i\}_{i=1}^J$ be the finite cover of
$\overline\calD\setminus B$ used in Lemma~\ref{lem:finite-killing}, and let
$\phi^{i,z}$, $T_i$, and $\eta_i$ denote the corresponding controlled path,
time horizon, and tube radius for $z\in O_i$.  Up to the first exit from the
compact union of these tubes, the Lipschitz estimate for the controlled SDE
and Gr\"onwall's lemma give
\begin{align}
 \sup_{0\le t\le T_i}\left|Z_t^{\delta,z}-\phi_t^{i,z}\right|
 \le C_i\left(
       \|B-h_i\|_{\infty,[0,T_i]}+\delta T_i 
       +\sqrt\delta\left(
          \|W^\theta\|_{\infty,[0,T_i]}
          +\|W^p\|_{\infty,[0,T_i]}
        \right)\right),
 \label{eq:regularized-controlled-tube}
\end{align}
where $h_i$ is the Cameron--Martin path realizing the $r$-control and fixed
diffusion coefficients have been absorbed into $C_i$.  Let $E_i^r$ be a
Brownian tube on which the first term on the right of
\eqref{eq:regularized-controlled-tube} is at most $\eta_i/4$.  Choose
$\delta_0>0$ so that $C_i\delta_0T_i\le\eta_i/4$ for every $i$.  On the
additional event
\[
    E_i^{\rm extra}
    =
    \left\{
        \sup_{t\le T_i}\left(\left|W_t^\theta\right|+\left|W_t^p\right|\right)
        \le
        \frac{\eta_i}{4C_i\sqrt{\delta_0}}
    \right\},
\]
the artificial martingales satisfy, for every $0<\delta\le\delta_0$,
\[
    C_i\sqrt{\delta}
    \sup_{t\le T_i}\left(\left|W_t^\theta\right|+\left|W_t^p\right|\right)
    \le \frac{\eta_i}{4} .
\]
Thus $E_i^r\cap E_i^{\rm extra}$ forces the regularized path, uniformly over
the corresponding initial neighborhood and all $0<\delta\le\delta_0$, to remain
inside the same controlled tube and to hit $\calC$ by time $T_i$.  Independence
of the Brownian coordinates gives
\[
    a_i:=\Pbb(E_i^r\cap E_i^{\rm extra})
    =\Pbb(E_i^r)\Pbb(E_i^{\rm extra})>0.
\]
Since the cover is finite, with $T:=\max_iT_i$ and
$a:=\min_i a_i$, there are $T>0$, $a>0$, and $\delta_0>0$ such
that
\[
    \inf_{0\le\delta\le\delta_0}
    \inf_{z\in\overline\calD\setminus B}
    \Pbb_z^\delta\left(\tau_\calC^\delta\le T\right)\ge a .
\]
The strong Markov property yields
\[
    \sup_{0\le\delta\le\delta_0}\sup_{z\in K}
    \Pbb_z^\delta\left(\tau_\calC^\delta>nT\right)
    \le (1-a)^n .
\]
This proves the regularized tail reduction.  The limiting case is
Lemma~\ref{lem:finite-killing}.  The adjoint case is identical, since the
controlled $r$-equation remains directly forced.
\end{proof}

With finite-time reduction and boundary stability established, the regularized
first killed time and first-hit label converge outside a $\pi^\eps$-null
exceptional set.

\begin{proposition}[Finite-horizon stability of the first killed hit]
\label{prop:hitting-time-vector-stability}
Let
\[
    \sigma^\delta=\tau_A^\delta\wedge\tau_\calC^\delta,
    \qquad
    \sigma=\tau_A\wedge\tau_\calC .
\]
For $\pi^\eps$-almost every $z\in\Omega$ and every $M<\infty$, as
$\delta\downarrow0$,
\[
    \sigma^\delta\wedge M
    \Longrightarrow
    \sigma\wedge M .
\]
Moreover, as $\delta\downarrow0$,
\[
    \bfone_{\{\tau_A^\delta<\tau_\calC^\delta\}}
    \Longrightarrow
    \bfone_{\{\tau_A<\tau_\calC\}}
    \qquad\text{under }\Pbb_z .
\]
The same statement holds for the adjoint processes.
\end{proposition}

\begin{proof}
Fix a starting point $z\in\Omega$ and use the full finite-horizon coupling
from Theorem~\ref{thm:delta-coupling}.  Lemma~\ref{lem:no-fixed-time-atoms-d1}
and Proposition~\ref{prop:small-r-entry}, applied to the finite family of
boundary components, show directly under $\Pbb_z$ that the exceptional
hitting configurations used below have probability zero.  Thus, for every
fixed $M<\infty$,
\[
    \sup_{0\le t\le M}|Z_t^\delta-Z_t|\longrightarrow0
    \qquad\text{in probability}.
\]
For $b>0$, define the good-path event
\[
    \widehat\sigma
    :=\inf\{t\ge0:Z_t\in\partial A\cup\partial B\cup\partial\calD\},
\]
and
\[
    G_{M,b}:=\{\sigma>M,\ \widehat\sigma>M\}
    \cup\{\sigma<M,\ \mathsf U,\ |r_\sigma|>b,\ \mathsf X\},
\]
where $\mathsf U$ is the event that the first boundary component is unique and
$\mathsf X$ is the event that this boundary is crossed immediately.
If $\widehat\sigma\le M<\sigma$, then the first boundary contact must have
$r_{\widehat\sigma}=0$, because Lemma~\ref{lem:crossing-nonzero-r} would
identify that contact with entry into the corresponding open ball (or with the
outer-boundary killing time).  Proposition~\ref{prop:small-r-entry}, applied to
the first boundary contact, shows that the remaining characteristic-contact
event has probability zero.  Fixed-time atomlessness and first-hit separation,
followed by the same non-grazing and immediate-crossing estimates, therefore
imply
\begin{equation}\label{eq:good-hit-event-limit}
    \lim_{b\downarrow0}\Pbb_z(G_{M,b}^c)=0.
\end{equation}
On $\{\sigma<M,\mathsf U,|r_\sigma|>b,\mathsf X\}$,
Lemma~\ref{lem:hitting-decision-stability} says that the maps
\[
    x\mapsto \left(\tau_A(x)\wedge\tau_\calC(x)\right)\wedge M,
    \qquad
    x\mapsto
    \bfone_{\{\tau_A(x)<\tau_\calC(x),\,
                   \tau_A(x)\wedge\tau_\calC(x)\le M\}}
\]
are continuous at the limiting path.  On
$\{\sigma>M,\widehat\sigma>M\}$, the compact path
$Z([0,M])$ has strictly positive distance from the boundary union, so the same
two truncated maps are also continuous there.  Fixed-time atomlessness
removes the remaining event $\{\sigma=M\}$.  Consequently, for every $\eta>0$,
the full path coupling and the continuous mapping theorem give
\begin{align}
    \limsup_{\delta\downarrow0}
    \Pbb_z\left(
       |\sigma^\delta\wedge M-\sigma\wedge M|>\eta
    \right)
    &\le \Pbb_z(G_{M,b}^c),\label{eq:truncated-hit-time-error}\\
    \limsup_{\delta\downarrow0}
    \Pbb_z(I_M^\delta\ne I_M)
    &\le \Pbb_z(G_{M,b}^c),\label{eq:truncated-hit-label-error}
\end{align}
where
\[
    I_M^\delta
    :=\bfone_{\{\tau_A^\delta<\tau_\calC^\delta,\,\sigma^\delta\le M\}},
    \qquad
    I_M
    :=\bfone_{\{\tau_A<\tau_\calC,\,\sigma\le M\}}.
\]
Letting $b\downarrow0$ in
\eqref{eq:truncated-hit-time-error}--\eqref{eq:truncated-hit-label-error}
proves convergence in probability of both truncated functionals, and hence
the stated convergence in law of the time.

For the untruncated labels, put
$I^\delta=\bfone_{\{\tau_A^\delta<\tau_\calC^\delta\}}$ and
$I=\bfone_{\{\tau_A<\tau_\calC\}}$.  Then
\[
\begin{aligned}
    \Pbb_z(I^\delta\ne I)
    &\le \Pbb_z(I_M^\delta\ne I_M)
       +\Pbb_z(\sigma^\delta>M)+\Pbb_z(\sigma>M).
\end{aligned}
\]
First let $\delta\downarrow0$, then use
Lemma~\ref{lem:first-hit-tail-reduction} and let $M\to\infty$.  This proves
$I^\delta\to I$ in probability and therefore in law.  The same estimates
hold for the adjoint coupling, proving the adjoint statement.
\end{proof}

The convergence of the first-hit times and labels can now be passed through
the corresponding bounded path functionals.  This yields pointwise
convergence of the committors and the finite-horizon killing functionals.

\begin{proposition}[Pointwise convergence of hitting functionals]
\label{prop:pointwise-functional-convergence-lrs}
For $\pi^\eps$-almost every $z\in\calD$, as $\delta\downarrow0$,
\[
    h_\delta(z)\to h(z),
    \qquad
    h_\delta^*(z)\to h^*(z),
\]
and, for every $M<\infty$,
\[
    F_{\delta,M}(z)\to F_M(z).
\]
\end{proposition}

\begin{proof}
Fix $z\in\Omega$ outside the $\pi^\eps$-null exceptional set in
Proposition~\ref{prop:hitting-time-vector-stability}, and work under its
coupling.  Introduce
\[
\begin{gathered}
    I^\delta:=\bfone_{\{\tau_A^\delta<\tau_\calC^\delta\}},
    \qquad
    I:=\bfone_{\{\tau_A<\tau_\calC\}},\\
    I_M^\delta
      :=I^\delta\bfone_{\{\tau_A^\delta\wedge\tau_\calC^\delta\le M\}},
    \qquad
    I_M
      :=I\bfone_{\{\tau_A\wedge\tau_\calC\le M\}}.
\end{gathered}
\]
By using \eqref{eq:truncated-hit-label-error} and
\eqref{eq:good-hit-event-limit}, we get
\begin{equation}\label{eq:pointwise-truncated-label}
    \Pbb_z\left(I_M^\delta\ne I_M\right)\longrightarrow0
    \qquad\text{as $\delta\downarrow0$}.
\end{equation}
Moreover,
\begin{align*}
    \Pbb_z\left(I^\delta\ne I\right)
    &\le \Pbb_z\left(I_M^\delta\ne I_M\right)
       +\Pbb_z\left(\tau_A^\delta\wedge\tau_\calC^\delta>M\right)
       +\Pbb_z\left(\tau_A\wedge\tau_\calC>M\right).
\end{align*}
Letting first $\delta\downarrow0$ and then $M\to\infty$, using
Lemma~\ref{lem:first-hit-tail-reduction}, proves
$I^\delta\to I$ in probability.  Since the indicators are bounded,
\[
    \left|h_\delta(z)-h(z)\right|
    =\left|\E_z I^\delta-\E_z I\right|
    \le \E_z\left|I^\delta-I\right|
    =\Pbb_z\left(I^\delta\ne I\right)\longrightarrow0.
\]

For the finite-horizon killing functional, set
\[
    J_M^\delta:=\bfone_{\{\tau_\calC^\delta\le M\}},
    \qquad
    J_M:=\bfone_{\{\tau_\calC\le M\}}.
\]
Applying the same good-path argument to the boundary family
$\partial B\cup\partial\calD$ gives
\(\Pbb_z(J_M^\delta\ne J_M)\to0\).  Therefore,
\[
    |F_{\delta,M}(z)-F_M(z)|
    \le \E_z\left|J_M^\delta-J_M\right|\longrightarrow0.
\]
On $A$ and $B$, the extended committor values make the first convergence
immediate, and the boundary union is $\pi^\eps$-null.  Repeating the indicator
estimates under the adjoint coupling proves $h_\delta^*(z)\to h^*(z)$.
\end{proof}

Finally, we are ready to prove Theorem~\ref{thm:ball-stability}.

\begin{proof}[Proof of Theorem~\ref{thm:ball-stability}]
Item (i) follows from
Proposition~\ref{prop:pointwise-functional-convergence-lrs} for the adjoint
process and bounded convergence.

For item (ii), Proposition~\ref{prop:pointwise-functional-convergence-lrs}
gives pointwise convergence of $h_\delta^*$ and $F_{\delta,M}$.  Since the
products are bounded by one, bounded convergence gives convergence in
$L^1(\calD,\pi^\eps)$.

For item (iii), Proposition~\ref{prop:uniform-truncated-killing} implies, as
$\delta\downarrow0$, that
\[
    \sup_{z\in\partial A}|g_{\delta,M}(z)-g_M(z)|
    \le
    \sup_{z\in\partial A}
    \E_z\left|
        \left(\tau_\calC^\delta\wedge M\right)-(\tau_\calC\wedge M)
    \right|
    \to0 .
\]
Continuity of $g_M$ on $\partial A$ is
Lemma~\ref{lem:gM-continuity-boundary}.

Item (iv) is Proposition~\ref{prop:hitting-time-vector-stability}.

Item (v) is Lemma~\ref{lem:finite-killing}.  The adjoint statements are proved
in the same way.
\end{proof}

The proof of Theorem~\ref{thm:ball-stability} is now complete.  The remaining
results form a second main block of the paper: we construct the weak
equilibrium measure and weak capacity from the adjoint committor formula, and
then prove the hitting identity and two useful extensions.

\subsection{Weak Equilibrium Measures and Capacity--Hitting Identities}\label{sec:weak}

This part constructs the weak equilibrium measure and capacity directly from
the path-space stability proved in Theorem~\ref{thm:ball-stability}.

\begin{definition}[Admissible extensions]\label{def:admissible}
Let $\varphi\in C^2(\partial A)$.  An admissible extension of $\varphi$ is a
function $\Phi\in C^2(\overline{\calD})$ such that
\[
    \Phi|_{\partial A}=\varphi,
    \qquad
    \Phi=0\quad\text{in a neighborhood of }\overline B\cup\partial\calD.
\]
\end{definition}

For each fixed $0<\delta\le1$, the regularized operator supplies classical
elliptic equilibrium and killed-semigroup identities.  The limit is taken
directly at the level of these integral identities through path-space
stability.

\begin{lemma}[Fixed-$\delta$ elliptic auxiliary identities]\label{lem:elliptic}
Fix $\delta>0$.  Let $h_\delta$ and $h_\delta^*$ be the hitting committors in
\eqref{eq:committors}.  There exists a finite positive Borel measure
$\eta_{\eps,\delta}$ supported on $\partial A$ such that, for every
$\varphi\in C^2(\partial A)$ and every admissible extension $\Phi$,
\begin{equation}\label{eq:eta-delta}
    \int_{\partial A}\varphi d\eta_{\eps,\delta}
    =
    \int_{\calD}h_\delta^*(z)
    \left(-\calL_{\eps,\delta}\Phi\right)(z) \pi^\eps(dz).
\end{equation}
If $\psi\in C^2(\overline\calD)$ satisfies
\[
    \psi=1\quad\text{near }\overline A,
    \qquad
    \psi=0\quad\text{near }\overline B\cup\partial\calD,
\]
then
\begin{equation}\label{eq:cap-delta}
    \Cap_{\eps,\delta}(A,B;\calD):=\eta_{\eps,\delta}(\partial A)
    =
    \int_{\calD}h_\delta^*
    \left(-\calL_{\eps,\delta}\psi\right) d\pi^\eps.
\end{equation}
For every $M<\infty$,
\begin{equation}\label{eq:truncated-identity-delta}
    \int_{\partial A}g_{\delta,M} d\eta_{\eps,\delta}
    =
    \int_{\calD}h_\delta^*F_{\delta,M} d\pi^\eps.
\end{equation}
\end{lemma}

\begin{proof}
For fixed $\delta>0$, the operator $\calL_{\eps,\delta}$ is uniformly elliptic
on the bounded smooth domain
$\Omega=\calD\setminus(\overline A\cup\overline B)$, and the
committors are the classical solutions of the corresponding Dirichlet
problems:
\[
    \calL_{\eps,\delta}h_\delta=0,\qquad
    \calL_{\eps,\delta}^*h_\delta^*=0
    \quad\text{in }\Omega,
\]
with boundary values $1$ on $\partial A$ and $0$ on
$\partial B\cup\partial\calD$.  Classical
elliptic potential theory for non-self-adjoint uniformly elliptic operators
therefore defines the equilibrium measure on $\partial A$ as the conormal-flux
measure associated with the adjoint equilibrium potential $h_\delta^*$; see,
for instance, the fixed-operator Green identities and capacity construction in
\cite{LandimMarianiSeo2019}.  More explicitly, if $n_\Omega$ denotes the
outward unit normal of $\Omega$ and hence points into $A$ on $\partial A$, then
on $\partial A$
\[
    d\eta_{\eps,\delta}
    =
    \eps Z_\eps^{-1}e^{-H/\eps}
    \left[(\D_\delta-\Q)\nabla h_\delta^*\right]\cdot n_\Omega\,dS .
\]
On $\partial A$, the Dirichlet value of $h_\delta^*$ is constant.  Hence its
tangential gradient vanishes and
\[
    \nabla h_\delta^*
    =
    (\partial_{n_\Omega}h_\delta^*)n_\Omega .
\]
Since $\Q^\top=-\Q$, we have $n_\Omega^\top\Q n_\Omega=0$, and therefore
\[
    \left[(\D_\delta-\Q)\nabla h_\delta^*\right]\cdot n_\Omega
    =
    \left(n_\Omega^\top\D_\delta n_\Omega\right)
    \partial_{n_\Omega}h_\delta^* .
\]
For fixed $\delta>0$, $\D_\delta$ is uniformly positive definite, so
$n_\Omega^\top\D_\delta n_\Omega>0$.  The strong maximum principle gives
$0<h_\delta^*<1$ in $\Omega$.  Since $n_\Omega$ points out of $\Omega$ and
into $A$ on $\partial A$, the Hopf boundary point lemma
\cite[Lemma~3.4]{GilbargTrudinger2001} gives
\[
    \partial_{n_\Omega}h_\delta^*>0
    \qquad\text{on }\partial A.
\]
Hence the conormal density is nonnegative.  The full distribution obtained by
extending $h_\delta^*$ by
its Dirichlet values may contain boundary components on
$\partial A$, $\partial B$, and $\partial\calD$.  We define
$\eta_{\eps,\delta}$ to be only its $\partial A$-component.  Since an admissible
extension $\Phi$ vanishes in a neighborhood of
$\partial B\cup\partial\calD$, Green's identity implies that
\[
    \int_{\partial A}\varphi\,d\eta_{\eps,\delta}
    =
    \int_{\calD}h_\delta^*(-\calL_{\eps,\delta}\Phi)\,d\pi^\eps .
\]
This proves \eqref{eq:eta-delta} and also shows that the right-hand side
depends only on the trace of $\Phi$ on $\partial A$.  Taking
$\varphi\equiv1$ and a separating cutoff $\psi$ yields the fixed-$\delta$
capacity formula \eqref{eq:cap-delta}.  Its total mass is precisely
$\Cap_{\eps,\delta}(A,B;\calD)$; normalization is introduced after positivity
of the capacity is established.

Finally, we derive the fixed-$\delta$ truncated last-exit identity from the
killed semigroup.  Let
\[
    G:=\calD\setminus\overline B .
\]
The process used to define $g_{\delta,M}$ has killing set
$\calC=B\cup\partial\calD$.  Thus its natural state space is $G$.  Uniform
ellipticity and regularity of
$\partial B$ identify entry into the open ball $B$ with the Dirichlet exit from
$G$ at $\partial B$.  Let
\[
    P_t^{\delta,\calC}f(z)
    :=
    \E_z^\delta\left[f(Z_t^\delta)\bfone_{\{t<\tau_\calC^\delta\}}\right]
\]
be this killed semigroup on $G$.  Then
\[
    g_{\delta,M}(z)
    =
    \int_0^M P_t^{\delta,\calC}\mathbf 1(z)\,dt .
\]
View $(P_t^{\delta,\calC})_{t\ge0}$ as the strongly continuous killed
semigroup on $L^2(G,\pi^\eps)$, and let
$\calL_{\eps,\delta}^{\calC}$ denote its generator.  Since
$\mathbf1\in L^2(G,\pi^\eps)$, the integrated-semigroup generator identity
\cite[Lemma~II.1.9]{EngelNagel2000} gives
\[
    g_{\delta,M}
    =\int_0^M P_t^{\delta,\calC}\mathbf1\,dt
    \in\operatorname{Dom}\left(\calL_{\eps,\delta}^{\calC};G\right),
    \qquad
    \calL_{\eps,\delta}^{\calC}g_{\delta,M}
    =P_M^{\delta,\calC}\mathbf1-\mathbf1.
\]
Equivalently,
\[
    -\calL_{\eps,\delta}g_{\delta,M}
    =
    1-P_M^{\delta,\calC}\mathbf1
    =
    \Pbb_z^\delta\left(\tau_\calC^\delta\le M\right)
    =F_{\delta,M}(z).
\]
Interior and boundary elliptic regularity for the fixed uniformly elliptic
operator supply the local regularity and the zero Dirichlet trace needed in
the cutoff and Green-identity argument below.  Choose
$\chi\in C_c^\infty(G)$ such that $\chi=1$ in a neighborhood of
$\overline A$, and set $\Phi_M=\chi g_{\delta,M}$.  Interior elliptic
regularity on the support of $\chi$ shows that $\Phi_M$ is an admissible
extension of $g_{\delta,M}|_{\partial A}$.  Put
$u_M=\Phi_M-g_{\delta,M}$ on $G$.  Then $u_M$ vanishes in a neighborhood of
$\overline A$, while $u_M$ has zero Dirichlet trace on
$\partial B\cup\partial\calD$.

Extend $h_\delta^*$ by $1$ on $\overline A$ and by $0$ on $\overline B$, as
in \eqref{eq:committors}.  Splitting $G$ into $A$ and $\Omega$, we have
$u_M=0$ on $A$.  Green's identity on $\Omega$ implies that
\[
    \int_{\Omega}h_\delta^*
    \calL_{\eps,\delta}u_M\,d\pi^\eps=0.
\]
Indeed, on $\partial A$, both $u_M$ and its conormal derivative vanish because
$u_M$ vanishes in a full neighborhood of $\overline A$; on
$\partial B\cup\partial\calD$, both $u_M$ and $h_\delta^*$ have zero trace.
All identities involving $g_{\delta,M}$ are justified by the
killed-generator domain approximations just described.  Consequently,
\eqref{eq:eta-delta} implies that
\[
    \int_{\partial A}g_{\delta,M}\,d\eta_{\eps,\delta}
    =
    \int_{\calD}h_\delta^*
    \left(-\calL_{\eps,\delta}\Phi_M\right)d\pi^\eps
    =
    \int_G h_\delta^*
    \left(-\calL_{\eps,\delta}g_{\delta,M}\right)d\pi^\eps
    =
    \int_{\calD}h_\delta^*F_{\delta,M}\,d\pi^\eps,
\]
which is \eqref{eq:truncated-identity-delta}, a fixed-$\delta$ integral
identity.
\end{proof}

The next theorem is the first main weak-potential result.  It defines the
weak equilibrium measure through the adjoint probabilistic committor.  The
regularized measures serve as fixed-$\delta$ auxiliary elliptic objects
whose identities are stable under the coupling theorem.

\begin{theorem}[Weak equilibrium measure and weak capacity]
\label{thm:weak-measure}
Suppose Assumptions~\ref{ass:U} and~\ref{ass:controlled-coefficients} hold,
and adopt the \hyperref[setup:domains]{bounded-domain setup}.
Then there exists a
unique finite positive Borel measure $\eta_\eps=\eta_{A,B}^\eps$ on
$\partial A$ such that, for every $\varphi\in C^2(\partial A)$ and every
admissible extension $\Phi$,
\begin{equation}\label{eq:weak-measure}
    \int_{\partial A}\varphi d\eta_\eps
    =
    \int_{\calD}h^*(z)\left(-\calL_\eps\Phi\right)(z) \pi^\eps(dz).
\end{equation}
The right-hand side is independent of the chosen admissible extension, and
$\eta_{\eps,\delta}\Rightarrow\eta_\eps$ weakly on $\partial A$ as
$\delta\downarrow0$.
Equivalently, for any smooth separating cutoff $\psi$ satisfying
\[
    \psi=1\quad\text{near }\overline A,
    \qquad
    \psi=0\quad\text{near }\overline B\cup\partial\calD,
\]
we have
\begin{equation}\label{eq:cap-cutoff}
    \Cap_\eps(A,B;\calD)
    =
    \int_{\calD}h^*(z)\left(-\calL_\eps\psi\right)(z) \pi^\eps(dz),
\end{equation}
where $\Cap_\eps(A,B;\calD)$ is the weak capacity associated with $(A,B;\calD)$ defined as:
\begin{equation}\label{eq:weak-cap}
\Cap_\eps(A,B;\calD)=\eta_\eps(\partial A).
\end{equation}
\end{theorem}

\begin{proof}
Let $\varphi\in C^2(\partial A)$ and let $\Phi$ be admissible.  
By definition, set
\[
    \Lambda(\varphi)
    :=
    \int_{\calD}h^*(z)(-\calL_\eps\Phi)(z)\,\pi^\eps(dz).
\]
If $\Phi_1$ and $\Phi_2$ are two admissible extensions of the same boundary
function, apply \eqref{eq:eta-delta} to the auxiliary path-space measures
$\eta_{\eps,\delta}$ and pass to the limit using
Theorem~\ref{thm:ball-stability} and the uniform convergence
$\calL_{\eps,\delta}(\Phi_1-\Phi_2)\to\calL_\eps(\Phi_1-\Phi_2)$ on
$\overline\calD$.  The fixed-$\delta$ identity used here is the classical
Green identity of Lemma~\ref{lem:elliptic}.  Passing directly to the limit in
this integral identity shows that $\Lambda$ is well-defined on boundary traces.

For $\varphi\ge0$, choose any admissible extension $\Phi$.  Lemma~\ref{lem:elliptic}
implies that
\[
    \int_{\partial A}\varphi\,d\eta_{\eps,\delta}
    =
    \int_{\calD}h_\delta^*(-\calL_{\eps,\delta}\Phi)\,d\pi^\eps\ge0 .
\]
The right-hand side converges to $\Lambda(\varphi)$ by
Theorem~\ref{thm:ball-stability}; hence $\Lambda(\varphi)\ge0$.  If $\psi$ is
a separating cutoff, the same argument with $\varphi\equiv1$ implies that
\[
    \Lambda(1)=\int_{\calD}h^*(-\calL_\eps\psi)d\pi^\eps<\infty .
\]
Thus $\Lambda$ is a positive bounded functional on the dense algebra
$C^2(\partial A)\subset C(\partial A)$ and extends uniquely to a positive
bounded functional on $C(\partial A)$.  By the Riesz--Markov representation
theorem~\cite[Theorem~2.14]{Rudin1987}, there is a
unique finite positive Borel measure $\eta_\eps$ satisfying
\eqref{eq:weak-measure}.

Moreover, for every $\varphi\in C^2(\partial A)$, as $\delta\downarrow 0$,
\[
    \int_{\partial A}\varphi\,d\eta_{\eps,\delta}
    =
    \int_{\calD}h_\delta^*(-\calL_{\eps,\delta}\Phi)d\pi^\eps
    \longrightarrow
    \Lambda(\varphi)
    =
    \int_{\partial A}\varphi\,d\eta_\eps .
\]
The masses converge by taking $\varphi\equiv1$, and density of
$C^2(\partial A)$ in $C(\partial A)$ yields
$\eta_{\eps,\delta}\Rightarrow\eta_\eps$ as $\delta\downarrow 0$.  Taking $\varphi\equiv1$ in
\eqref{eq:weak-measure} yields \eqref{eq:cap-cutoff}.
\end{proof}

Having constructed the weak measure from the path-space committor, we use the
coupling stability of the stopped processes to obtain the capacity--hitting
formula.

\begin{theorem}[Capacity--hitting identity]\label{thm:hitting}
Suppose Assumptions~\ref{ass:U} and~\ref{ass:controlled-coefficients} hold,
and adopt the \hyperref[setup:domains]{bounded-domain setup}.
Let
\[
    g(z)=\E_z[\tau_\calC]
    =\E_z[\tau_B\wedge\tau_{\partial\calD}].
\]
Then
\begin{equation}\label{eq:hitting-unnormalized}
    \int_{\partial A}g(z) \eta_\eps(dz)
    =
    \int_{\calD}h^*(z) \pi^\eps(dz).
\end{equation}
\end{theorem}

\begin{proof}
For $M<\infty$, Lemma~\ref{lem:elliptic}, specifically the fixed-$\delta$
identity \eqref{eq:truncated-identity-delta}, implies that
\[
    \int_{\partial A}g_{\delta,M} d\eta_{\eps,\delta}
    =
    \int_{\calD}h_\delta^*F_{\delta,M} d\pi^\eps.
\]
Let $\delta\downarrow0$.  By Theorem~\ref{thm:ball-stability},
$g_{\delta,M}\to g_M$ uniformly on $\partial A$, and the masses
$\eta_{\eps,\delta}(\partial A)$ are uniformly bounded. Hence,
\[
    \int_{\partial A}g_{\delta,M} d\eta_{\eps,\delta}
    -
    \int_{\partial A}g_M d\eta_{\eps,\delta}
    \longrightarrow0.
\]
The function $g_M$ is continuous on $\partial A$ by
Lemma~\ref{lem:gM-continuity-boundary}.  Since
$\eta_{\eps,\delta}\Rightarrow\eta_\eps$ as $\delta\downarrow0$ by
Theorem~\ref{thm:weak-measure},
\[
    \int_{\partial A}g_M d\eta_{\eps,\delta}
    \longrightarrow
    \int_{\partial A}g_M d\eta_\eps.
\]
For the right-hand side, Theorem~\ref{thm:ball-stability} implies, as
$\delta\downarrow0$, that
\[
    h_\delta^*F_{\delta,M}\to h^*F_M
    \qquad\text{in }L^1(\calD,\pi^\eps).
\]
Therefore, for every $M<\infty$,
\begin{equation}\label{eq:truncated-identity-limit}
    \int_{\partial A}g_M d\eta_\eps
    =
    \int_{\calD}h^*F_M d\pi^\eps.
\end{equation}

Finally, we let $M\to\infty$.  The functions
$g_M(z)=\E_z[\tau_\calC\wedge M]$ increase pointwise to
$g(z)=\E_z[\tau_\calC]$.  By finite killing in
Theorem~\ref{thm:ball-stability}, 
$F_M(z)=\Pbb_z(\tau_\calC\le M)\uparrow1$ for
$\pi^\eps$-almost every $z\in\calD$, and
$0\le h^*F_M\le h^*$. Applying the monotone convergence theorem on the left-hand side of \eqref{eq:truncated-identity-limit} with respect to
$\eta_\eps$ and on the right-hand side of \eqref{eq:truncated-identity-limit} with respect to $\pi^\eps$ yields
\eqref{eq:hitting-unnormalized}.
\end{proof}

The preceding identity immediately yields positivity of the weak capacity and
the normalized hitting identity used in applications.

\begin{corollary}[Strict positivity and normalized hitting identity]
\label{cor:capacity-positive}
Under the hypotheses of Theorem~\ref{thm:hitting},
\[
    0<\Cap_\eps(A,B;\calD)<\infty.
\]
Consequently,
\[
    \nu_\eps=\frac{\eta_\eps}{\Cap_\eps(A,B;\calD)}
\]
is a probability measure on $\partial A$, and
\begin{equation}\label{eq:hitting-normalized}
    \int_{\partial A}
    \E_z[\tau_B\wedge\tau_{\partial\calD}] \nu_\eps(dz)
    =
    \frac{\displaystyle\int_{\calD}h^*(z) \pi^\eps(dz)}
    {\Cap_\eps(A,B;\calD)}.
\end{equation}
\end{corollary}

\begin{proof}
Finiteness follows from Theorem~\ref{thm:weak-measure}.  By the extension
convention in \eqref{eq:limiting-committors}, $h^*=1$ on $A$.  The set $A$ has
non-empty interior and $\pi^\eps$ has a smooth strictly positive density, hence
\[
    \int_{\calD}h^* d\pi^\eps\ge \pi^\eps(A)>0.
\]
By Lemma~\ref{lem:finite-killing}, $g$ is bounded on $\partial A$.  If
$\eta_\eps(\partial A)=0$, then the left-hand side of
\eqref{eq:hitting-unnormalized} would be zero, contradicting the strict
positivity of the right-hand side.  Thus
$\Cap_\eps(A,B;\calD)=\eta_\eps(\partial A)>0$.  Dividing
\eqref{eq:hitting-unnormalized} by the capacity gives
\eqref{eq:hitting-normalized}.
\end{proof}

\subsection{Passage to the Whole-Space Identity}\label{sec:passage}

The purpose of this subsection is to derive the whole-space identity from the
bounded-domain construction.  Following the truncation strategy of
Lee--Ramil--Seo, we first record the recurrence condition, proved here through a
Lyapunov drift estimate, and then let the outer radius tend to infinity.

\begin{lemma}[Lyapunov drift for the forward and adjoint chains]
\label{lem:whole-space-lyapunov-drift}
Assume Assumptions~\ref{ass:U} and~\ref{ass:whole-space-growth}.  Write
$\calL_\eps^{(+)}=\calL_\eps$ and
$\calL_\eps^{(-)}=\calL_\eps^*$.  Equivalently,
\[
    \calL_\eps^{(\sigma)} f
    =
    \sigma p\cdot\nabla_\theta f
    +\left(-\sigma L^{-1}\nabla U(\theta)+\sigma\gamma r\right)
        \cdot\nabla_p f
    +(-\sigma\gamma p-\gamma r)\cdot\nabla_r f
    +\frac{\gamma\eps}{L}\Delta_r f,
    \qquad \sigma\in\{+1,-1\}.
\]
There exist constants $\kappa,\alpha,\eta,C_V,c,C>0$, with
$\eta=\kappa\alpha$, such that the functions
\[
    V_\sigma(\theta,p,r)
    =
    C_V+H(\theta,p,r)+\eta H(\theta,p,r)^2
    +\alpha K_\sigma(\theta,p,r),
    \qquad \sigma\in\{+1,-1\},
\]
where
\[
    K_\sigma
    =
    \sigma\theta\cdot p+\theta\cdot r
    +\frac{2\sigma}{\gamma}p\cdot r
    +\frac{\gamma}{2}|\theta|^2,
\]
are nonnegative, have compact sublevel sets, and satisfy
\begin{equation}\label{eq:lyap-drift-forward-adjoint}
    \calL_\eps^{(\sigma)}V_\sigma
    \le
    -c\left(U(\theta)+|p|^2+|r|^2+H(\theta,p,r)|r|^2\right)+C,
    \qquad \sigma\in\{+1,-1\}.
\end{equation}
In particular, for every
\[
    R\ge \frac{C+1}{c},
\]
the compact set $\mathsf K_R$ defined in
\eqref{eq:whole-space-lyapunov-set} satisfies
$\calL_\eps^{(\sigma)}V_\sigma\le -1$ on
$\mathsf K_R^c$ for both signs $\sigma$.
\end{lemma}

\begin{proof}
Assumption~\ref{ass:U} implies $|\theta|^2\le C(1+U(\theta))$, after changing
$C$ to cover a compact set.  Hence $|K_\sigma|\le C(1+H)$, uniformly in
$\sigma$, and choosing $C_V$ large makes $V_\sigma$ nonnegative and proper.

The Hamiltonian identities are independent of $\sigma$:
\[
    \calL_\eps^{(\sigma)}H=-\gamma L|r|^2+\gamma d\eps,
\]
and, since $\nabla_rH=Lr$,
\[
    \calL_\eps^{(\sigma)}H^2
    =
    -2\gamma L H|r|^2+2\gamma d\eps H+2\gamma\eps L|r|^2 .
\]
The corrector was chosen so that the chain terms cancel.  Its four pieces give
\[
\begin{aligned}
\calL_\eps^{(\sigma)}(\sigma\theta\cdot p)
    &=
    |p|^2-\frac1L\theta\cdot\nabla U(\theta)
    +\gamma\theta\cdot r,\\
\calL_\eps^{(\sigma)}(\theta\cdot r)
    &=
    \sigma p\cdot r
    -\sigma\gamma\theta\cdot p
    -\gamma\theta\cdot r,\\
\calL_\eps^{(\sigma)}\left(\frac{2\sigma}{\gamma}p\cdot r\right)
    &=
    -\frac{2}{\gamma L}r\cdot\nabla U(\theta)
    -2|p|^2+2|r|^2-2\sigma p\cdot r,\\
\calL_\eps^{(\sigma)}\left(\frac{\gamma}{2}|\theta|^2\right)
    &=
    \sigma\gamma\theta\cdot p .
\end{aligned}
\]
Thus the $\theta\cdot r$ and $\theta\cdot p$ terms cancel, and the
$p\cdot r$ coefficient is $\sigma-2\sigma=-\sigma$.  Consequently,
\begin{equation}\label{eq:lyap-K-computation}
    \calL_\eps^{(\sigma)}K_\sigma
    =
    -\frac1L\theta\cdot\nabla U(\theta)
    -\frac{2}{\gamma L}r\cdot\nabla U(\theta)
    -|p|^2+2|r|^2-\sigma p\cdot r .
\end{equation}
Combining these formulas gives the exact expansion of
$\calL_\eps^{(\sigma)}V_\sigma$:
\begin{align}
\calL_\eps^{(\sigma)}V_\sigma
={}&
-\gamma L|r|^2+\gamma d\eps
-2\eta\gamma L H|r|^2
+2\eta\gamma d\eps H
+2\eta\gamma\eps L|r|^2 \notag\\
&-\frac{\alpha}{L}\theta\cdot\nabla U(\theta)
-\frac{2\alpha}{\gamma L}r\cdot\nabla U(\theta)
-\alpha|p|^2+2\alpha|r|^2-\alpha\sigma p\cdot r .
\label{eq:lyap-exact-expansion}
\end{align}

We estimate the indefinite-sign terms.  First choose $\kappa>0$
small enough.  By \eqref{eq:radial-growth}, for $|\theta|\ge R_0$,
\[
    U(\theta)\le c_0^{-1}\theta\cdot\nabla U(\theta).
\]
Hence, by choosing
\[
    \kappa
    \le
    \min\left\{
        \frac{c_0}{256\gamma d\eps L},
        \frac{1}{64\gamma d\eps L}
    \right\},
\]
we get, after increasing the final constant to cover the compact set
$\{|\theta|\le R_0\}$,
\[
    2\kappa\gamma d\eps U(\theta)
    \le \frac{1}{64L}\theta\cdot\nabla U(\theta)+C,
    \qquad
    \kappa\gamma d\eps L\le \frac1{32}.
\]
Let $\eta=\kappa\alpha$.  For a numerical constant $C_1$, Young's
inequality and $1+H\ge1+U(\theta)$ give
\[
    \frac{2\alpha}{\gamma L}|r\cdot\nabla U|
    \le
    \frac{\eta\gamma L}{2}(1+H)|r|^2
    +
    \frac{C_1\alpha}{\gamma^3L^3\kappa}
    \frac{|\nabla U|^2}{1+H}
    \le
    \frac{\eta\gamma L}{2}(1+H)|r|^2
    +
    \frac{C_1\alpha}{\gamma^3L^3\kappa}
    \frac{|\nabla U(\theta)|^2}{1+U(\theta)}.
\]
After $\kappa$ is fixed, Assumption~\ref{ass:whole-space-growth}, in the
equivalent form \eqref{eq:whole-space-growth-equivalent}, implies there exists some
$R_\kappa<\infty$ such that, for $|\theta|\ge R_\kappa$,
\[
    \frac{C_1}{\gamma^3L^3\kappa}
    \frac{|\nabla U(\theta)|^2}{1+U(\theta)}
    \le
    \frac{1}{64L}\theta\cdot\nabla U(\theta).
\]
On the compact set $\{|\theta|\le R_\kappa\}$, the positive part of the
difference between the left-hand side and the right-hand side is bounded;
hence
\[
    \frac{C_1\alpha}{\gamma^3L^3\kappa}
    \frac{|\nabla U(\theta)|^2}{1+U(\theta)}
    \le
    \frac{\alpha}{64L}\theta\cdot\nabla U(\theta)+\alpha C_\kappa .
\]
Therefore the gradient-ratio term and the positive potential contribution
$2\eta\gamma d\eps U$ from $\eta\calL_\eps^{(\sigma)}H^2$ are absorbed by
the negative term $-(\alpha/L)\theta\cdot\nabla U$, up to a bounded
remainder.  The positive $p$-part is controlled by the choice of $\kappa$:
\[
    2\eta\gamma d\eps\frac L2|p|^2
    =
    \kappa\alpha\gamma d\eps L|p|^2
    \le \frac{\alpha}{32}|p|^2.
\]
Finally,
\[
    \alpha|p\cdot r|
    \le \frac{\alpha}{16}|p|^2+C_{\rm pr}\alpha |r|^2.
\]
We now collect the preceding estimates.  First, the trace contribution in
\eqref{eq:lyap-exact-expansion} splits as
\[
    2\eta\gamma d\eps H
    =2\eta\gamma d\eps U(\theta)
    +\eta\gamma d\eps L|p|^2
    +\eta\gamma d\eps L|r|^2.
\]
The first term on the right and the gradient-ratio term in
\eqref{eq:lyap-exact-expansion} each consume at most $1/64$ of
$(\alpha/L)\theta\cdot\nabla U$.  The second term consumes at most
$1/32$ of $\alpha|p|^2$, and the $p\cdot r$ term consumes at most
$1/16$ of $\alpha|p|^2$.  The Young bound for
$r\cdot\nabla U$ consumes one half of the coefficient of
$\eta\gamma L H|r|^2$.  Consequently, after increasing a constant $C_3$ to
absorb all bounded remainders, we obtain
\begin{align*}
\calL_\eps^{(\sigma)}V_\sigma
\le{}&
-\left(1-\frac1{64}-\frac1{64}\right)
    \frac{\alpha}{L}\theta\cdot\nabla U(\theta)\\
&\qquad-\left(1-\frac1{32}-\frac1{16}\right)\alpha|p|^2
-\left(2-\frac12\right)\eta\gamma L H|r|^2
-\gamma L|r|^2\\
&\qquad\qquad+\left(
    2\alpha+\frac{\eta\gamma L}{2}
    +\eta\gamma d\eps L+2\eta\gamma\eps L
    +C_{\rm pr}\alpha
  \right)|r|^2+C_3.
\end{align*}
Since $\eta=\kappa\alpha$, define
\[
    C_2
    :=2+C_{\rm pr}+\frac{\kappa\gamma L}{2}
      +\kappa\gamma d\eps L+2\kappa\gamma\eps L.
\]
Thus $C_2<\infty$ is independent of $\alpha$ and of
$\sigma\in\{+1,-1\}$, and we have
\begin{align}
\calL_\eps^{(\sigma)}V_\sigma
&\le
-\frac{31\alpha}{32L}\,\theta\cdot\nabla U(\theta)
-\frac{29\alpha}{32}|p|^2
-\frac{3\eta\gamma L}{2}H|r|^2
-\gamma L|r|^2
+C_2\alpha|r|^2+C_3.                         \label{eq:lyap-collected}
\end{align}
Choose $\alpha>0$ sufficiently small that
$C_2\alpha\le\gamma L/2$.  Weakening the remaining numerical constants in
\eqref{eq:lyap-collected} then yields
\begin{equation}\label{eq:lyap-coercive-intermediate}
\calL_\eps^{(\sigma)}V_\sigma
\le
-\frac{\alpha}{2L}\,\theta\cdot\nabla U(\theta)
-\frac{\alpha}{2}|p|^2
-\frac{\gamma L}{2}|r|^2
-\eta\gamma L H|r|^2+C_3.
\end{equation}
Finally, Assumption~\ref{ass:U} and compactness of
$\{\theta:|\theta|\le R_0\}$ imply that, for some $C_0<\infty$,
\[
\theta\cdot\nabla U(\theta)
\ge c_0\left(U(\theta)+|\theta|^2\right)-C_0,
\qquad \theta\in\R^d.
\]
Substituting this estimate into
\eqref{eq:lyap-coercive-intermediate}, and decreasing $c>0$ if necessary,
gives
\[
\calL_\eps^{(\sigma)}V_\sigma
\le
-c\left(U(\theta)+|p|^2+|r|^2+H|r|^2\right)+C,
\]
uniformly for $\sigma\in\{+1,-1\}$, which is
\eqref{eq:lyap-drift-forward-adjoint}.
\end{proof}

The Lyapunov drift estimate gives finite-mean entrance into a compact set.  This
return property is then combined with a controlled-path argument and the strong
Markov property to prove the almost-sure hitting statements needed for
whole-space exhaustion.

\begin{proposition}[Whole-space recurrence for exhaustion]
\label{prop:whole-space-recurrence}
Suppose Assumptions~\ref{ass:U}, \ref{ass:controlled-coefficients},
\ref{ass:whole-space-growth}, and \ref{ass:two-balls} hold.  Then the forward
and adjoint limiting processes are non-explosive, admit the invariant
probability measure $\pi^\eps$, and return to a compact Lyapunov set in finite
mean time.  Moreover,
\[
    \Pbb_z(\tau_B<\infty)=1,
    \qquad
    \Pbb_z^*(\tau_A\wedge\tau_B<\infty)=1,
    \qquad z\in\R^{3d}.
\]
\end{proposition}

\begin{proof}
Proposition~\ref{prop:basic} implies non-explosion, finiteness of $\pi^\eps$, and
the infinitesimal invariance identity for the forward process.  The adjoint
non-explosion is included in Assumption~\ref{ass:controlled-coefficients}, and
infinitesimal invariance follows from the adjoint relation:
\[
    \int_{\R^{3d}}\calL_\eps^*fd\pi^\eps
    =
    \int_{\R^{3d}}f\calL_\eps1d\pi^\eps=0,
    \qquad f\in C_c^\infty(\R^{3d}).
\]
Since the forward and adjoint martingale problems are well-posed and the
processes are conservative, the Echeverr\'ia invariant-measure criterion
\cite{Echeverria1982} applies to the test class $C_c^\infty(\R^{3d})$, using
conservativeness, well-posedness of the martingale problem, and the identities
\[
    \int_{\R^{3d}}\calL_\eps f\,d\pi^\eps=0,
    \qquad
    \int_{\R^{3d}}\calL_\eps^*f\,d\pi^\eps=0,
    \qquad f\in C_c^\infty(\R^{3d}).
\]
Hence $\pi^\eps$ is invariant for both semigroups.
Equivalently, $\pi^\eps$ is a distributional stationary solution of the
Fokker--Planck equation, and well-posedness of the martingale problem gives
uniqueness of that measure-valued evolution.  Thus,
\[
    \int_{\R^{3d}}P_tfd\pi^\eps
    =
    \int_{\R^{3d}}fd\pi^\eps,
    \qquad
    \int_{\R^{3d}}P_t^*fd\pi^\eps
    =
    \int_{\R^{3d}}fd\pi^\eps,
\]
first for bounded continuous $f$, and then for bounded measurable $f$ by a
monotone-class argument.

Choose
\[
    R\ge \frac{C+1}{c},
    \qquad \mathsf K:=\mathsf K_R,
\]
where $c,C$ are from Lemma~\ref{lem:whole-space-lyapunov-drift}.
The function $V_\sigma$ is bounded on $\mathsf K$, and the generator estimate
gives
\[
    \calL_\eps^{(\sigma)}V_\sigma\le -1+b\mathbf 1_{\mathsf K},
\]
for a finite constant $b$ and both $\sigma\in\{+1,-1\}$.  Applying Dynkin's
formula up to
$t\wedge\tau_{\mathsf K}$ gives, for $z\notin\mathsf K$,
\[
    \mathbb E_z[t\wedge\tau_{\mathsf K}]
    \le V_\sigma(z),
\]
and monotone convergence as $t\to\infty$ implies
$\mathbb E_z\tau_{\mathsf K}<\infty$.  Thus, from every starting point outside
$\mathsf K$, each process enters $\mathsf K$ almost surely in finite time and
with finite mean.  This stopped Dynkin estimate is the only recurrence input
needed below; in particular, no uniformly elliptic recurrence criterion is
invoked for the degenerate chain.

It remains to show that each return to $\mathsf K$ gives a uniformly positive
chance of reaching the required target.  Put
\[
    \Pbb_z^{(+)}:=\Pbb_z,
    \qquad
    \Pbb_z^{(-)}:=\Pbb_z^*,
    \qquad
    G_+:=B,
    \qquad
    G_-:=A\cup B.
\]
The controlled skeleton associated with $\calL_\eps^{(\sigma)}$ is
\begin{equation}\label{eq:whole-space-controlled-skeleton}
\dot\theta=\sigma p,
\qquad
\dot p=-\frac{\sigma}{L}\nabla U(\theta)+\sigma\gamma r,
\qquad
\dot r=-\sigma\gamma p-\gamma r+a_r u,
\qquad
a_r:=\sqrt{2\gamma\eps/L}.
\end{equation}
Fix $q_+=s$ and $q_-=m$, so that $(q_\sigma,0,0)\in G_\sigma$ and
$\nabla U(q_\sigma)=0$.  For every
$z=(\theta_0,p_0,r_0)\in\mathsf K$, choose a smooth path
$\Theta_z^\sigma:[0,1]\to\R^d$ satisfying
\[
\begin{gathered}
\Theta_z^\sigma(0)=\theta_0,
\qquad
\dot\Theta_z^\sigma(0)=\sigma p_0,
\qquad
\ddot\Theta_z^\sigma(0)=-L^{-1}\nabla U(\theta_0)+\gamma r_0,\\
\Theta_z^\sigma(1)=q_\sigma,
\qquad
\dot\Theta_z^\sigma(1)=0,
\qquad
\ddot\Theta_z^\sigma(1)=0.
\end{gathered}
\]
For example, these six endpoint conditions determine a componentwise quintic
Hermite polynomial.  Define
\[
P_z^\sigma=\sigma\dot\Theta_z^\sigma,
\qquad
R_z^\sigma=\frac1\gamma\left(
\ddot\Theta_z^\sigma+\frac1L\nabla U(\Theta_z^\sigma)
\right),
\]
and
\[
u_z^\sigma
=a_r^{-1}\left(
\dot R_z^\sigma+\sigma\gamma P_z^\sigma+\gamma R_z^\sigma
\right).
\]
Then $(\Theta_z^\sigma,P_z^\sigma,R_z^\sigma)$ solves
\eqref{eq:whole-space-controlled-skeleton}, starts from $z$, and ends at
$(q_\sigma,0,0)\in G_\sigma$.  Since $\mathsf K$ is compact, the family of
controlled paths is contained in one compact set and the controls have
uniformly bounded Cameron--Martin norms.

Localize the coefficients on a compact neighborhood of these paths.  The
Stroock--Varadhan support theorem \cite[Theorem~5.2]{StroockVaradhan1972},
continuous dependence on the initial point, and a finite covering of
$\mathsf K$ then give constants $q_\sigma>0$ such that
\begin{equation}\label{eq:whole-space-uniform-accessibility}
    \inf_{z\in\mathsf K}
    \Pbb_z^{(\sigma)}\left(\tau_{G_\sigma}\le1\right)
    \ge q_\sigma>0,
    \qquad \sigma\in\{+1,-1\}.
\end{equation}

Finally, start the first attempt when the process enters $\mathsf K$ (at time
zero if it already lies there) and allow one unit of time for the attempt.  At
the end of a missed attempt, start the next one immediately if the process is
in $\mathsf K$; otherwise, wait until its next entrance into $\mathsf K$.  The
finite-mean entrance estimate above, applied at the end of each attempt through
the strong Markov property, makes every such waiting time almost surely finite.
Iterating the strong Markov property and
\eqref{eq:whole-space-uniform-accessibility} therefore yields
\[
    \Pbb_z^{(\sigma)}\left(
        \text{the first $n$ attempts all miss }G_\sigma
    \right)
    \le(1-q_\sigma)^n.
\]
Letting $n\to\infty$ proves
\[
    \Pbb_z(\tau_B<\infty)=1,
    \qquad
    \Pbb_z^*(\tau_A\wedge\tau_B<\infty)=1
\]
for every $z\in\R^{3d}$.
\end{proof}

\begin{remark}[Relation with the Lyapunov step in Lee--Ramil--Seo]
In Lee--Ramil--Seo, the recurrence comes from the Lyapunov estimate in
\cite[Lemma~3.2]{LeeRamilSeo2026}: non-explosion and positive recurrence are
then stated in \cite[Proposition~3.3]{LeeRamilSeo2026}.  Their proof cites
Pinsky's diffusion Lyapunov criterion
\cite[Chapter~2, Assumption~A, and Theorem~6.1.3]{Pinsky1995}, whose stated
assumptions require a strictly elliptic diffusion matrix.  Since both the
underdamped chain and the present third-order chain are degenerate, we do not
invoke that criterion here.
Instead, the stopped Dynkin formula gives finite-mean entrance into the compact
Lyapunov set directly.  The controlled-path accessibility estimate and the
strong Markov property then give the repeated attempts and the almost-sure
target-hitting statements.
Lee--Ramil--Seo use recurrence in the exhaustion passage in
\cite[Section~4.5, Proof of Proposition~2.9]{LeeRamilSeo2026}, both for the
outer-domain limit and for the final time-truncation limit.  The present
third-order chain uses the additional $\eta H^2$ term in
Lemma~\ref{lem:whole-space-lyapunov-drift} and the gradient-ratio estimate to
absorb $r\cdot\nabla U$, whereas the LRS underdamped corrector cancels the force
terms directly.
\end{remark}

The preceding return-and-accessibility argument controls the infinite-time
tails, whereas the exhaustion argument also requires continuity in the initial
point at each fixed horizon.  The next lemma supplies this continuity on the
entrance boundary.

\begin{lemma}[Continuity of finite-time survival probabilities]
\label{lem:whole-space-survival-continuity}
Assume Assumptions~\ref{ass:U}, \ref{ass:controlled-coefficients}, and
\ref{ass:two-balls}.  Choose $R_*$ large enough that
$\overline A\cup\overline B\subset B(0,R_*)$ and, for every $R\ge R_*$, the
boundary components
$\partial A$, $\partial B$, and $\partial B(0,R)$ have pairwise disjoint tubular
neighborhoods.  For every $t>0$ and $R\ge R_*$, with $\calD_R=B(0,R)$ and
$\calC_R=B\cup\partial\calD_R$, the maps
\[
    z\mapsto\Pbb_z(\tau_{\calC_R}>t),
    \qquad
    z\mapsto\Pbb_z(\tau_B>t)
\]
are continuous on $\partial A$.
\end{lemma}

\begin{proof}
By the choice of $R_*$, for each $R\ge R_*$ the triple $(A,B,\calD_R)$ satisfies
the bounded-domain geometry assumption with $N=R$.
Consider first $z\mapsto\Pbb_z(\tau_{\calC_R}>t)$.  The finite-horizon
boundary-decision stability argument from
Lemma~\ref{lem:gM-continuity-boundary} applies to the boundary union
$\partial B\cup\partial\calD_R$.  The possible characteristic hits are controlled
by the local estimates in Proposition~\ref{prop:small-r-entry} and
Lemma~\ref{lem:fast-entry-collar}, applied to the ball $\calD_R$.  Fixed-time
boundary atoms are excluded by the density argument used in
Lemma~\ref{lem:no-fixed-time-atoms-d1}.  Hence, if $z_n,z\in\partial A$ and
$z_n\to z$, the indicators
${\bfone}_{\{\tau_{\calC_R}(Z^{z_n})>t\}}$ converge in probability to
${\bfone}_{\{\tau_{\calC_R}(Z^z)>t\}}$, and bounded convergence proves
continuity of $z\mapsto\Pbb_z(\tau_{\calC_R}>t)$ on $\partial A$.

For the second map, localize paths in a larger ball $\calD_S=B(0,S)$, $S>R_*$,
up to time $t$.  The finite-time moment bounds in
Assumption~\ref{ass:controlled-coefficients} give
\begin{equation}\label{eq:whole-space-localization-tail}
    \sup_{z\in\partial A}\Pbb_z(\tau_{\partial\calD_S}\le t)\longrightarrow0
    \qquad\text{as }S\to\infty.
\end{equation}
On the event $\{\tau_{\partial\calD_S}>t\}$, the survival event
$\{\tau_B>t\}$ is the same as the corresponding event in the localized ball.
The first part, applied with $R=S$, gives the continuity of
$z\mapsto\Pbb_z(\tau_{B\cup\partial\calD_S}>t)$ on $\partial A$.  Letting
$S\to\infty$ and using the uniform localization bound
\eqref{eq:whole-space-localization-tail} proves
the continuity of $z\mapsto\Pbb_z(\tau_B>t)$.
\end{proof}

The preceding continuity result allows Dini's theorem to be used when the
outer radius tends to infinity.  Together with non-explosion, it yields the
finite-time exhaustion limits collected below.

\begin{lemma}[Finite-time exhaustion stability]
\label{lem:finite-time-exhaustion}
Assume Assumptions~\ref{ass:U}, \ref{ass:controlled-coefficients}, and
\ref{ass:two-balls}.  Choose $R_*$ large enough that
$\overline A\cup\overline B\subset B(0,R_*)$, set $\calD_R=B(0,R)$ and
$\calC_R=B\cup\partial\calD_R$, and extend
\[
    h_R^*(z)=\Pbb_z^*(\tau_A<\tau_{\calC_R})
\]
by zero outside $\calD_R$.  Let
\[
    h_\infty^*(z)=\Pbb_z^*(\tau_A<\tau_B),
    \qquad
    g_{R,M}(z)=\E_z[\tau_{\calC_R}\wedge M],
    \qquad
    g_{\infty,M}(z)=\E_z[\tau_B\wedge M],
\]
and
\[
    F_{R,M}(z)=\Pbb_z(\tau_{\calC_R}\le M),
    \qquad
    F_{\infty,M}(z)=\Pbb_z(\tau_B\le M).
\]
We extend $F_{R,M}$ by zero outside $\calD_R$; since $h_R^*$ is also extended by
zero there, the product $h_R^*F_{R,M}$ is a Borel function on $\R^{3d}$.
Then, as $R\to\infty$,
\[
    h_R^*\to h_\infty^*
    \qquad\text{in }L^1(\pi^\eps),
\]
and, for every $M<\infty$,
\[
    \sup_{z\in\partial A}
    |g_{R,M}(z)-g_{\infty,M}(z)|
    \longrightarrow0,
    \qquad
    g_{\infty,M}\in C(\partial A).
\]
Moreover,
\[
    h_R^*F_{R,M}\to h_\infty^*F_{\infty,M}
    \qquad\text{in }L^1(\pi^\eps).
\]
\end{lemma}

\begin{proof}
Let $\tau_E^*$ denote the adjoint hitting time of $E$.
By non-explosion, $\tau_{\partial\calD_R}\uparrow\infty$ almost surely for both
the forward and adjoint processes.  Since $\overline A$ and $\overline B$ are
disjoint compact sets with positive distance, path continuity excludes
simultaneous first entry into $A$ and $B$.  Hence, for each fixed $z$,
\[
    {\bfone}_{\{\tau_A^*<\tau_B^*\wedge\tau_{\partial\calD_R}^*\}}
    \uparrow
    {\bfone}_{\{\tau_A^*<\tau_B^*\}},
\]
after $R$ is large enough to contain $z$.  Bounded convergence implies the
pointwise convergence of $h_R^*$, and dominated convergence with respect to the
probability measure $\pi^\eps$ gives $L^1(\pi^\eps)$ convergence.

Fix $t>0$.  Lemma~\ref{lem:whole-space-survival-continuity} gives continuity on
$\partial A$ of
\[
    u_R(z)=\Pbb_z(\tau_{\calC_R}>t),
    \qquad
    u_\infty(z)=\Pbb_z(\tau_B>t).
\]

For each $z\in\partial A$,
$\tau_{\calC_R}=\tau_B\wedge\tau_{\partial\calD_R}\uparrow\tau_B$, and therefore
\[
    u_R(z)\uparrow u_\infty(z).
\]
The functions $u_R$ and $u_\infty$ are continuous on the compact set
$\partial A$.  Dini's theorem, applied along any sequence $R_n\uparrow\infty$ and
then using monotonicity in $R$, yields that
\[
    \sup_{z\in\partial A}
    |u_R(z)-u_\infty(z)|\longrightarrow0
\]
for each $t>0$.  Since the integrand is bounded by $1$,
\[
    \sup_{z\in\partial A}|g_{R,M}(z)-g_{\infty,M}(z)|
    \le
    \int_0^M
    \sup_{z\in\partial A}
    \left|
        \Pbb_z(\tau_{\calC_R}>t)-\Pbb_z(\tau_B>t)
    \right|dt
    \longrightarrow0.
\]
The preceding continuity of the survival probabilities also gives
$g_{\infty,M}\in C(\partial A)$.

Finally, $\tau_{\calC_R}\uparrow\tau_B$, and therefore
\[
    {\bfone}_{\{\tau_{\calC_R}\le M\}}
    \longrightarrow
    {\bfone}_{\{\tau_B\le M\}}
\]
almost surely.  Bounded convergence therefore gives
$F_{R,M}(z)\to F_{\infty,M}(z)$ for every $z$.
Since the $F$'s are bounded by $1$,
\[
    |h_R^*F_{R,M}-h_\infty^*F_{\infty,M}|
    \le |h_R^*-h_\infty^*|
       +h_\infty^*|F_{R,M}-F_{\infty,M}|,
\]
and dominated convergence proves the asserted $L^1(\pi^\eps)$ convergence.
\end{proof}

The preceding lemma controls all fixed-horizon terms appearing in the
bounded-domain identity.  The almost-sure hitting statements in
Proposition~\ref{prop:whole-space-recurrence} then remove the time truncation
and allow passage to the limit at the level of equilibrium measures and
capacities.

\begin{proposition}[Exhaustion of the outer boundary]
\label{prop:exhaustion}
Suppose Assumptions~\ref{ass:U}, \ref{ass:controlled-coefficients},
\ref{ass:whole-space-growth}, and \ref{ass:two-balls} hold.  Choose $R_*$ large enough that
$\overline A\cup\overline B\subset B(0,R_*)$, and let $\calD_R=B(0,R)$,
$R\ge R_*$.  For each $R$,
let $\eta_R$ and $\Cap_R=\eta_R(\partial A)$ be the bounded-domain weak
equilibrium measure and weak capacity in $\calD_R$ for the pair $(A,B)$.
Set
\[
    h_\infty^*(z)=\Pbb_z^*(\tau_A<\tau_B),
    \qquad z\in\R^{3d}.
\]
As before, $h_\infty^*$ is extended by the boundary values $1$ on
$\overline A$ and $0$ on $\overline B$.
Then there exists a unique finite positive Borel measure $\eta_\infty$ on
$\partial A$ such that, for every $\varphi\in C^2(\partial A)$ and every
$\Phi\in C_c^2(\R^{3d})$ satisfying $\Phi|_{\partial A}=\varphi$ and
$\Phi=0$ in a neighborhood of $\overline B$,
\begin{equation}\label{eq:exhaustion-weak-measure}
    \int_{\partial A}\varphi d\eta_\infty
    =
    \int_{\R^{3d}}h_\infty^*(-\calL_\eps\Phi) d\pi^\eps .
\end{equation}
Moreover, as $R\to\infty$, $\eta_R\Rightarrow\eta_\infty$ weakly on
$\partial A$ and
$\Cap_R\to\Cap_\infty:=\eta_\infty(\partial A)$.
In addition,
\begin{equation}\label{eq:exhaustion-hitting}
    \int_{\partial A}\E_z[\tau_B] \eta_\infty(dz)
    =
    \int_{\R^{3d}}h_\infty^* d\pi^\eps .
\end{equation}
Consequently $\Cap_\infty>0$, and the corresponding normalized whole-space
hitting identity follows by dividing \eqref{eq:exhaustion-hitting} by
$\Cap_\infty$.
\end{proposition}

\begin{proof}
For $R$ large, the triple $(A,B,\calD_R)$ satisfies the
\hyperref[setup:domains]{bounded-domain setup}.  Hence
Theorem~\ref{thm:weak-measure} gives the bounded-domain weak equilibrium
measure $\eta_R$ and the cutoff formula for $\Cap_R$,
Theorem~\ref{thm:hitting} gives the bounded-domain hitting identity, and
Corollary~\ref{cor:capacity-positive} gives $\Cap_R>0$.
Fix a smooth compactly supported cutoff $\psi$ such that
$\psi=1$ in a neighborhood of $\overline A$ and $\psi=0$ in a neighborhood of
$\overline B$.  For all
large $R$, $\psi$ is admissible in $\calD_R$, and
\[
    \Cap_R
    =
    \int_{\calD_R}h_R^*(-\calL_\eps\psi)d\pi^\eps .
\]
Since $0\le h_R^*\le1$ and $\calL_\eps\psi$ is bounded with compact support, the
masses $\Cap_R=\eta_R(\partial A)$ are uniformly bounded.  The compactness of
$\partial A$ implies tightness.

Let $R_n\to\infty$ be any sequence along which
$\eta_{R_n}\Rightarrow\eta$ weakly.  If $\Phi$ is as in the statement, then
$\supp\Phi\subset\calD_R$ for all large $R$, so the bounded-domain
weak-measure identity implies that
\[
    \int_{\partial A}\varphi d\eta_R
    =
    \int_{\calD_R}h_R^*(-\calL_\eps\Phi)d\pi^\eps .
\]
By Lemma~\ref{lem:finite-time-exhaustion}, the right-hand side in the above
equation converges to
$\int_{\R^{3d}} h_\infty^*(-\calL_\eps\Phi)d\pi^\eps$,
while the left-hand side in the above equation converges
to $\int_{\partial A}\varphi d\eta$.  Thus every subsequential limit satisfies
\eqref{eq:exhaustion-weak-measure}.  In particular, the right-hand side of
\eqref{eq:exhaustion-weak-measure} is independent of the chosen compactly
supported extension $\Phi$, because it is the weak limit of
$\int_{\partial A}\varphi d\eta_R$.  This formula determines the measure
uniquely because $C^2(\partial A)$ is dense in $C(\partial A)$.  Hence the full
family converges weakly to a single measure, denoted by $\eta_\infty$.  Taking
$\varphi\equiv1$ shows that $\Cap_R\to\Cap_\infty$.

It remains to pass the hitting identity.  For fixed $M<\infty$, the truncated
identity \eqref{eq:truncated-identity-limit}, applied in the bounded domain
$\calD_R$, implies that
\[
    \int_{\partial A}g_{R,M}d\eta_R
    =
    \int_{\calD_R}h_R^*F_{R,M}d\pi^\eps .
\]
The left-hand side converges to
\[
    \int_{\partial A}g_{\infty,M}d\eta_\infty .
\]
Indeed,
$g_{R,M}\to g_{\infty,M}$ uniformly on $\partial A$ by
Lemma~\ref{lem:finite-time-exhaustion}, the masses $\eta_R(\partial A)$ are
uniformly bounded, $g_{\infty,M}$ is continuous on $\partial A$, and
$\eta_R\Rightarrow\eta_\infty$.  The right-hand side converges to
\[
    \int_{\R^{3d}}h_\infty^*(z)F_{\infty,M}(z)\,\pi^\eps(dz),
\]
by the last assertion of Lemma~\ref{lem:finite-time-exhaustion}.
Therefore
\[
    \int_{\partial A}\E_z[\tau_B\wedge M]d\eta_\infty
    =
    \int_{\R^{3d}}h_\infty^*
    \Pbb_z(\tau_B\le M)d\pi^\eps .
\]
Letting $M\to\infty$, monotone convergence on the left and
Proposition~\ref{prop:whole-space-recurrence} on the right yield
\eqref{eq:exhaustion-hitting}.  Finally, the right-hand side is strictly
positive because $h_\infty^*=1$ on $A$ by extension and
$\pi^\eps(A)>0$.  Hence $\eta_\infty(\partial A)=\Cap_\infty>0$.
\end{proof}

\begin{remark}[Comparison with Lee--Ramil--Seo]
The preceding proof follows the finite-time truncation strategy used in
\cite[Proof of Proposition~2.9]{LeeRamilSeo2026}.  The passage $R\to\infty$ is
performed after fixing the cutoff $M$, where Dini's theorem gives uniform
convergence on the compact entrance boundary.  The final $M\to\infty$ step is
then only monotone convergence, backed by recurrence.
\end{remark}

\section{Conclusion}\label{sec:conclusion}

We have developed a fixed-temperature weak equilibrium-measure and capacity
framework for the hypoelliptic third-order Langevin diffusion in phase-space
balls for every $d\ge1$, under the potential and controlled-coefficient
assumptions of Section~\ref{sec:main}.  The bounded-domain argument uses
chain-compatible localization together with uniform Pigato-type density and
covariance estimates.  The equilibrium measure is characterized directly by
the adjoint hitting committor (Theorem~\ref{thm:weak-measure}):
\[
    \int_{\partial A}\varphi d\eta_\eps
    =
    \int_{\calD}h^*(-\calL_\eps\Phi) d\pi^\eps.
\]
The auxiliary elliptic measures serve as fixed-$\delta$ approximating
objects: their classical Green and last-exit identities are passed to the
degenerate limit through hitting-law stability.  The main output is the
capacity--hitting identity (Theorem~\ref{thm:hitting}) and its normalized form
(Corollary~\ref{cor:capacity-positive}), where
$\nu_\eps=\eta_\eps/\Cap_\eps(A,B;\calD)$:
\[
    \int_{\partial A}
    \E_z[\tau_B\wedge\tau_{\partial\calD}] \nu_\eps(dz)
    =
    \frac{1}
    {\Cap_\eps(A,B;\calD)}\int_{\calD}h^* d\pi^\eps.
\]

Under the additional whole-space growth condition used for the Lyapunov
argument, the artificial outer boundary can be removed.
In that case, the bounded-domain weak equilibrium measures
$\eta_R$ converge to a
whole-space weak equilibrium measure $\eta_\infty$, and the bounded-domain
identity passes to (Proposition~\ref{prop:exhaustion}):
\[
    \int_{\partial A}\E_z[\tau_B]\eta_\infty(dz)
    =
    \int_{\R^{3d}}h_\infty^*d\pi^\eps .
\]

Regularization stability is obtained by combining pathwise coupling, immediate
crossing at non-characteristic hits, and a uniform small-$r$ boundary-entry
estimate.  The latter adapts the boundary strategy of
Lee--Ramil--Seo~\cite{LeeRamilSeo2026} to the third-order chain and uses
chain-compatible localization together with Pigato-type anisotropic density
estimates.


\appendix

\section{Uniform Perturbation Estimates for the Cutoff Covariance}
\label{app:uniform-perturbation}

This appendix records the uniform perturbation step used in
Lemma~\ref{lem:uniform-pigato-covariance}.  Fix a bounded smooth
$\mathsf O\subset\R^{3d}$, a compact set $K_0\subset\mathsf O$, and use the
block order
\[
    X^1=r,\qquad X^2=p,\qquad X^3=\theta .
\]
For the cutoff regularized process write
\[
    dX_t
    =
    B^{\delta,\mathsf O}(X_t)\,dt
    +\Sigma_r\,dW_t^r
    +\Sigma_p^\delta\,dW_t^p
    +\Sigma_\theta^\delta\,dW_t^\theta ,
\]
where
\[
    \Sigma_r=
    \begin{pmatrix}
        S_r\\ 0\\ 0
    \end{pmatrix},
    \qquad
    S_r=\sqrt{2\gamma\eps/L}\,I_d,
\]
and
\[
    \Sigma_p^\delta=
    \begin{pmatrix}
        0\\ \sqrt{2\eps\delta}\,I_d\\ 0
    \end{pmatrix},
    \qquad
    \Sigma_\theta^\delta=
    \begin{pmatrix}
        0\\ 0\\ \sqrt{2\eps\delta}\,I_d
    \end{pmatrix}.
\]
All these diffusion matrices are constant in space and
$|\Sigma_p^\delta|+|\Sigma_\theta^\delta|\le C$ for $0\le\delta\le1$.
Let
\[
    Y_t^\delta=\partial_zX_t^{\delta,\mathsf O},
    \qquad
    Z_t^\delta=(Y_t^\delta)^{-1}.
\]
Since the diffusion matrices are constant,
\begin{equation}\label{eq:appendix-inverse-flow}
    dZ_t^\delta
    =
    -Z_t^\delta DB^{\delta,\mathsf O}(X_t^{\delta,\mathsf O})\,dt .
\end{equation}

The short-time covariance expansion begins with the three frozen noise
directions generated by the chain.  The next lemma shows that their signed
chain matrix is uniformly invertible in both $\delta$ and the initial point.

\begin{lemma}[Uniform local chain non-degeneracy]
\label{lem:appendix-local-chain}
Let $K_0\subset\mathsf O_0$, where $K_0$ is compact and $\mathsf O_0$ is an
open set with compact closure satisfying
$\overline{\mathsf O_0}\subset\mathsf O$, and choose the cutoff maps in
Lemma~\ref{lem:chain-compatible-cutoffs} equal to the identity near
$\overline{\mathsf O_0}$.  For $z\in K_0$, set
\[
    \widetilde A^\delta(z)=A_3^\delta(z)
    =
    \operatorname{diag}\left(
        S_r,\,
        -J_{x^1}B_2^{\delta,\mathsf O}(z)S_r,\,
        J_{x^2}B_3^{\delta,\mathsf O}(z)
        J_{x^1}B_2^{\delta,\mathsf O}(z)S_r
    \right).
\]
For a square matrix $M$, write
$s_{\min}(M):=\inf_{|v|=1}|Mv|=\sqrt{\lambda_{\min}(MM^\top)}$ for its
smallest singular value.  Then
\begin{equation}\label{eq:appendix-chain-smin}
    \inf_{\substack{0\le\delta\le1\\z\in K_0}}
    s_{\min}\left(A_3^\delta(z)\right)>0,
\end{equation}
and equivalently,
\begin{equation}\label{eq:appendix-chain-gram}
    \inf_{\substack{0\le\delta\le1\\z\in K_0}}
    \lambda_{\min}\left(A_3^\delta(z)A_3^\delta(z)^\top\right)>0.
\end{equation}
That is,
\[
    \sup_{\substack{0\le\delta\le1\\z\in K_0}}
    \left(
        \left\|\widetilde A^\delta(z)\right\|
        +\left\|\left(\widetilde A^\delta(z)\right)^{-1}\right\|
    \right)<\infty .
\]
The same conclusion holds for the adjoint cutoff family.
\end{lemma}

\begin{proof}
On $\mathsf O_0$ the cutoff maps are identity, so
$J_{x^1}B_2^{\delta,\mathsf O}=\gamma I_d$ and
$J_{x^2}B_3^{\delta,\mathsf O}=I_d$ there.  Hence, for $z\in K_0$,
$A_3^\delta(z)=\operatorname{diag}(S_r,-\gamma S_r,\gamma S_r)$, independent of
$\delta$.  Since $S_r=\sqrt{2\gamma\eps/L}\,I_d$ is invertible, the lower
bounds \eqref{eq:appendix-chain-smin} and
\eqref{eq:appendix-chain-gram} follow.
The adjoint drift changes only the corresponding signs,
which leaves the singular values unchanged.
\end{proof}

Finally, put
\[
    G(u)=
    \begin{pmatrix}
        I_d\\ uI_d\\ \frac{u^2}{2}I_d
    \end{pmatrix},
    \qquad 0\le u\le1.
\]
The sign in the middle block follows from the convention in the definition of
$A_3^\delta$.

The matrix $G$ is the deterministic leading profile of the normalized
inverse-flow expansion.  We next estimate the error produced by evaluating
the coefficients along the stochastic trajectory rather than at the initial
point.

\begin{lemma}[Uniform chain remainders]\label{lem:appendix-chain-remainder}
There exists $t_0>0$ such that, for every $k\ge0$ and $m\ge1$, there is a
constant $C_{k,m}$ satisfying, for all $0\le\delta\le1$, $z\in K_0$, and
$0<t\le t_0$,
\[
\begin{aligned}
    Z_s^\delta\Sigma_r
    &=
    \begin{pmatrix}
        S_r\\
        -sJ_{x^1}B_2^{\delta,\mathsf O}(z)S_r\\
        \frac{s^2}{2}
        J_{x^2}B_3^{\delta,\mathsf O}(z)
        J_{x^1}B_2^{\delta,\mathsf O}(z)S_r
    \end{pmatrix}
    +
    \begin{pmatrix}
        \overline R_s^{\,1}\\
        \overline R_s^{\,2}\\
        \overline R_s^{\,3}
    \end{pmatrix},
    \qquad 0\le s\le t,
\end{aligned}
\]
where, for $\ell=1,2,3$,
\[
    \sup_{\substack{0\le\delta\le1\\z\in K_0\\0<s\le t_0}}
    s^{-(\ell-1/2)}
    \left\|\overline R_s^{\,\ell}\right\|_{k,m}
    \le C_{k,m}.
\]
Equivalently,
\begin{equation}\label{eq:appendix-normalized-chain-expansion}
    T_t^{-1}\left(A_3^\delta(z)\right)^{-1}
    Z_{tu}^{\delta,z}\Sigma_r
    =
    t^{-1/2}\left(G(u)+\mathcal R_{t,u}^\delta(z)\right),
    \qquad 0\le u\le1,
\end{equation}
where the $\ell$-th block satisfies
\begin{equation}\label{eq:appendix-normalized-remainder-bound}
    \left\|\mathcal R_{t,u}^{\delta,\ell}(z)\right\|_{k,m}
    \le C_{k,m}t^{1/2}u^{\ell-1/2},
    \qquad 0\le u\le1,\quad \ell=1,2,3.
\end{equation}
Moreover, for every $m\ge1$,
\[
    \sup_{\substack{0\le\delta\le1\\z\in K_0\\0<t\le t_0}}
    \mathbb E_z\left[\left\|T_t^{-1}Z_t^\delta T_t\right\|^m\right]<\infty.
\]
\end{lemma}

\begin{proof}
The cutoff construction gives bounded derivatives of all orders, uniformly for
$0\le\delta\le1$.  Let $\mathcal A^{\delta,\mathsf O}$ denote the generator of
the cutoff regularized process.  For every smooth cutoff coefficient $F$,
It\^o's formula along the full regularized trajectory implies that
\begin{align}
    dF(X_t^{\delta,\mathsf O})
    ={}&
    \mathcal A^{\delta,\mathsf O}F(X_t^{\delta,\mathsf O})\,dt
    +DF(X_t^{\delta,\mathsf O})\Sigma_r\,dW_t^r \notag\\
    &+DF(X_t^{\delta,\mathsf O})\Sigma_p^\delta\,dW_t^p
    +DF(X_t^{\delta,\mathsf O})\Sigma_\theta^\delta\,dW_t^\theta .
    \label{eq:appendix-ito-coefficient}
\end{align}
The two additional martingales have uniformly bounded integrands.  
Hence, Burkholder-Davis-Gundy inequality
implies that, for every $m\ge1$,
\[
    \left\|
        \sup_{s\le t}\left|
        \int_0^sDF(X_u^{\delta,\mathsf O})\Sigma_p^\delta\,dW_u^p
        \right|
    \right\|_{L^m}
    \le C_m t^{1/2},
\]
and the same estimate holds with $\Sigma_\theta^\delta dW^\theta$ in place of
$\Sigma_p^\delta dW^p$.  To make the Malliavin bounds explicit, let
$\alpha\in\{r,p,\theta\}$ denote a noise block.  Since the noise is additive,
for $0\le u\le t$,
\begin{equation}\label{eq:appendix-first-malliavin-derivative}
    D_u^\alpha X_t^{\delta,\mathsf O}
    =Y_t^\delta Z_u^\delta\Sigma_\alpha^\delta,
    \qquad
    D_u^\alpha X_t^{\delta,\mathsf O}=0\quad(u>t),
\end{equation}
where $\Sigma_r^\delta:=\Sigma_r$.  The boundedness of
$DB^{\delta,\mathsf O}$ and the variational equations for $Y^\delta,Z^\delta$
therefore imply, for all $m\ge1$ and $T<\infty$,
\begin{equation}\label{eq:appendix-first-malliavin-bound}
    \sup_{\substack{0\le\delta\le1,\ z\in K_0\\
                    0\le u\le t\le T}}
    \left\|D_u^\alpha X_t^{\delta,\mathsf O}\right\|_{L^m}
    \le C_{m,T}.
\end{equation}
For $j\ge2$, differentiating once more gives a linear variation-of-constants
formula of the form
\begin{equation}\label{eq:appendix-higher-malliavin-recursion}
\begin{aligned}
    D_{\boldsymbol u}^{\boldsymbol\alpha,j}X_t^{\delta,\mathsf O}
    =\int_{u_*}^tY_t^\delta Z_s^\delta
      \sum_{\substack{\mathcal P\text{ a partition of }\{1,\ldots,j\}\\
                      |\mathcal P|\ge2}}
      D^{|\mathcal P|}B^{\delta,\mathsf O}(X_s^{\delta,\mathsf O})
      \left[
        D_{\boldsymbol u_I}^{\boldsymbol\alpha_I,|I|}
        X_s^{\delta,\mathsf O}:I\in\mathcal P
      \right]ds,
\end{aligned}
\end{equation}
with $u_*:=\max_i u_i$; terms with $u_*>t$ vanish.  All derivatives of the
cutoff drift are bounded uniformly in $\delta$.  Induction in $j$, using
H\"older's inequality and Gr\"onwall's lemma in
\eqref{eq:appendix-higher-malliavin-recursion}, yields
\[
    \sup_{0\le\delta\le1,\ z\in K_0}
    \sup_{0\le u_1,\ldots,u_j\le t\le T}
    \left\|D_{\boldsymbol u}^{\boldsymbol\alpha,j}
    X_t^{\delta,\mathsf O}\right\|_{L^m}
    \le C_{j,m,T}.
\]
Differentiating the variational equations
\[
    dY_t^\delta
    =DB^{\delta,\mathsf O}(X_t^{\delta,\mathsf O})Y_t^\delta\,dt,
    \qquad
    dZ_t^\delta
    =-Z_t^\delta DB^{\delta,\mathsf O}(X_t^{\delta,\mathsf O})\,dt
\]
with respect to the noise variables gives linear equations whose forcing terms
contain only bounded derivatives of $B^{\delta,\mathsf O}$ and lower-order
Malliavin derivatives of $X^\delta,Y^\delta,Z^\delta$.  The same induction
therefore gives, for every $k\ge0$, $m\ge1$, and $T<\infty$,
\[
    \sup_{\substack{0\le\delta\le1,\ z\in K_0\\0\le t\le T}}
    \left(
        \|Y_t^\delta\|_{k,m}+\|Z_t^\delta\|_{k,m}
    \right)<\infty .
\]
The Malliavin chain rule (equivalently, the Fa\`a di Bruno formula) transfers
these bounds to $F(X_t^{\delta,\mathsf O})$; Burkholder--Davis--Gundy inequality then
gives the same uniform bounds for the stochastic integrals in
\eqref{eq:appendix-ito-coefficient}.  Thus the coefficients frozen at
$z\in K_0$ give the leading chain directions, while the coefficients sampled
along the full trajectory $X_s^{\delta,\mathsf O}$ are estimated as remainders.
The globally bounded cutoff derivatives provide uniform local-to-global bounds.

We next write the three-layer expansion directly from the matrix blocks of the
inverse flow.  Put
\[
    A_{ij}^\delta(x)=J_{x^j}B_i^{\delta,\mathsf O}(x),
    \qquad 1\le i,j\le3,
\]
and write $Z_s^\delta=((Z_s^\delta)_{ij})_{1\le i,j\le3}$ in $d\times d$
blocks.  The chain-compatible cutoff preserves the triangular links
$A_{21}^\delta=J_{x^1}B_2^{\delta,\mathsf O}$,
$A_{32}^\delta=J_{x^2}B_3^{\delta,\mathsf O}$, and
$A_{31}^\delta=0$.  We use the notation
$E_s=\mathcal O_{k,m}(a_s)$ when
$\sup_{0\le u\le s}\|E_u\|_{k,m}\le C_{k,m}a_s$, uniformly in
$0\le\delta\le1$ and $z\in K_0$.

The diagonal blocks satisfy
\[
    (Z_s^\delta)_{ii}=I_d+\mathcal O_{k,m}(s),\qquad i=1,2,3,
\]
and the lower blocks satisfy the equations
\begin{subequations}\label{eq:appendix-lower-block-system}
\begin{align}
    \frac{d}{ds}(Z_s^\delta)_{21}
    &=
    -(Z_s^\delta)_{21}A_{11}^\delta(X_s^\delta)
    -(Z_s^\delta)_{22}A_{21}^\delta(X_s^\delta)
    -(Z_s^\delta)_{23}A_{31}^\delta(X_s^\delta),
    \label{eq:appendix-block-21}\\
    \frac{d}{ds}(Z_s^\delta)_{32}
    &=
    -(Z_s^\delta)_{31}A_{12}^\delta(X_s^\delta)
    -(Z_s^\delta)_{32}A_{22}^\delta(X_s^\delta)
    -(Z_s^\delta)_{33}A_{32}^\delta(X_s^\delta),
    \label{eq:appendix-block-32}\\
    \frac{d}{ds}(Z_s^\delta)_{31}
    &=
    -(Z_s^\delta)_{31}A_{11}^\delta(X_s^\delta)
    -(Z_s^\delta)_{32}A_{21}^\delta(X_s^\delta)
    -(Z_s^\delta)_{33}A_{31}^\delta(X_s^\delta).
    \label{eq:appendix-block-31}
\end{align}
\end{subequations}
Since $A_{31}^\delta=0$, bounded cutoff derivatives and Gr\"onwall's lemma first give
$(Z_s^\delta)_{21}=\mathcal O_{k,m}(s)$ and
$(Z_s^\delta)_{32}=\mathcal O_{k,m}(s)$, and then
$(Z_s^\delta)_{31}=\mathcal O_{k,m}(s^2)$.  Using also
\[
    \sup_{0\le u\le s}\|X_u^\delta-z\|_{k,m}
    \le C_{k,m}s^{1/2},
    \qquad
    \left\|\sup_{0\le u\le s}|X_u^\delta-z|\right\|_{L^m}
    \le C_m s^{1/2},
\]
the first two lower links improve to
\begin{align*}
    (Z_s^\delta)_{21}
    &=
    -sA_{21}^\delta(z)+\mathcal O_{k,m}(s^{3/2}),\\
    (Z_s^\delta)_{32}
    &=
    -sA_{32}^\delta(z)+\mathcal O_{k,m}(s^{3/2}).
\end{align*}
Indeed, after subtracting the frozen leading term in
\eqref{eq:appendix-block-21}, the difference from
$-A_{21}^\delta(z)$ is the sum of
$((Z_s^\delta)_{22}-I_d)A_{21}^\delta(X_s^\delta)$,
$(A_{21}^\delta(X_s^\delta)-A_{21}^\delta(z))$, and
$(Z_s^\delta)_{21}A_{11}^\delta(X_s^\delta)$, all of order
$\mathcal O_{k,m}(s^{1/2})$ or better before integration.  The argument for
\eqref{eq:appendix-block-32} is the same, with the additional term
$(Z_s^\delta)_{31}A_{12}^\delta(X_s^\delta)=\mathcal O_{k,m}(s^2)$.

Finally, integrating \eqref{eq:appendix-block-31} gives the second-order link
with the correct matrix order:
\[
\begin{aligned}
    (Z_s^\delta)_{31}S_r
    &=
    -\int_0^s
        (Z_u^\delta)_{32}A_{21}^\delta(X_u^\delta)S_r\,du
      -\int_0^s
        (Z_u^\delta)_{31}A_{11}^\delta(X_u^\delta)S_r\,du\\
    &=
    \int_0^s
        uA_{32}^\delta(z)A_{21}^\delta(z)S_r\,du
      +\mathcal O_{k,m}(s^{5/2})\\
    &=
    \frac{s^2}{2}A_{32}^\delta(z)A_{21}^\delta(z)S_r
      +\mathcal O_{k,m}(s^{5/2}).
\end{aligned}
\]
The second integral is $\mathcal O_{k,m}(s^3)$, while the error in replacing
$(Z_u^\delta)_{32}$ by $-uA_{32}^\delta(z)$ and
$A_{21}^\delta(X_u^\delta)$ by $A_{21}^\delta(z)$ integrates to
$\mathcal O_{k,m}(s^{5/2})$.

Since 
\[
Z_s^\delta\Sigma_r=\left((Z_s^\delta)_{11}S_r,
(Z_s^\delta)_{21}S_r,(Z_s^\delta)_{31}S_r\right)^\top, 
\]
we have, uniformly over
$0\le s\le t\le t_0$,
\[
\begin{aligned}
    R_s
    &=S_r+\overline R_s^{\,1},\\
    P_s
    &=-sJ_{x^1}B_2^{\delta,\mathsf O}(z)S_r
        +\overline R_s^{\,2},\\
    \Theta_s
    &=\frac{s^2}{2}
        J_{x^2}B_3^{\delta,\mathsf O}(z)
        J_{x^1}B_2^{\delta,\mathsf O}(z)S_r
        +\overline R_s^{\,3},
\end{aligned}
\]
where
\[
    \sup_{0\le s\le t}\left\|\overline R_s^{\,j}\right\|_{k,m}
    \le C_{k,m}t^{j-1/2},
    \qquad j=1,2,3.
\]
For $j=1$ this bound deliberately records the weaker half-integer scale
$t^{1/2}$, although the block equation gives the stronger
$\mathcal O(t)$ estimate.
The weaker statement is convenient after normalization and matches the scale of
the coefficient fluctuations caused by the full regularized trajectory.  For
the second and third blocks, the estimates $t^{3/2}$ and $t^{5/2}$ enter through
the integrated covariance rather than through a pointwise higher-order-error
comparison with the raw principal terms.  After the anisotropic normalization,
the principal part is separated from these coefficient-fluctuation terms.
Lemma~\ref{lem:appendix-pigato-stopping}
performs the additional integration in the trajectory-time variable $s$ over
$0\le s\le t\xi$ and controls the normalized remainder-covariance matrix
\[
    D_\xi^{-1}
    \left(
        \int_0^{t\xi}
        \widetilde R_{t,s}^\delta
        \left(\widetilde R_{t,s}^\delta\right)^\top\,ds
    \right)
    D_\xi^{-1},
    \qquad 0<\xi\le1.
\]
This makes the comparison quantitative.  The principal terms above are
generated only by
\[
    \Sigma_r,\qquad
    J_{x^1}B_2^{\delta,\mathsf O}\Sigma_r,\qquad
    J_{x^2}B_3^{\delta,\mathsf O}
    J_{x^1}B_2^{\delta,\mathsf O}\Sigma_r .
\]
The constant $p$- and $\theta$-diffusion fields are spatially constant, so their
diffusion-derivative commutators vanish in the Jacobian equation.
Their contributions enter the uniformly controlled martingale and drift
remainders in the It\^o expansion \eqref{eq:appendix-ito-coefficient}.
Multiplication by $(A_3^\delta(z))^{-1}$ and by the diagonal weights in $T_t$
turns these three estimates into
\eqref{eq:appendix-normalized-remainder-bound}.

The final estimate follows from \eqref{eq:appendix-inverse-flow}.  Uniform
Gr\"onwall bounds give all block moments of $Z_t^\delta$.  The triangular
dependence implies the refined lower-block estimates
\[
    \|(Z_t^\delta)_{21}\|_{L^m}\le C_mt,\qquad
    \|(Z_t^\delta)_{32}\|_{L^m}\le C_mt,\qquad
    \|(Z_t^\delta)_{31}\|_{L^m}\le C_mt^2,
\]
while diagonal and upper blocks are uniformly bounded.  The weights in
	$T_t^{-1}Z_t^\delta T_t$ exactly compensate these lower-block powers.
	\end{proof}

The estimates above are unconditional.  In particular, they include
trajectories that leave $\mathsf O_0$ before time $t$.  The quantitative chain
lower bound is used only in the frozen matrices
$A_{21}^\delta(z)$ and $A_{32}^\delta(z)$ at the initial point
$z\in K_0$.  After cutoff, all coefficients and their derivatives are globally
bounded uniformly in $\delta$, so rapid displacement from $z$, including rapid
exit from $\mathsf O_0$, is already controlled by the Malliavin--Sobolev
remainder estimates and by the remainder-covariance stopping time in
Lemma~\ref{lem:appendix-pigato-stopping}.

The preceding lemma gives pointwise-in-time remainder estimates.  Since the
Malliavin covariance involves their time integrals, we next derive the
corresponding blockwise covariance-moment bounds.

\begin{lemma}[Normalized remainder covariance moments]
\label{lem:appendix-remainder-covariance}
For fixed $0<t\le t_0$ and $0\le s\le t$, write
\[
    \widetilde R_{t,s}^{\delta,\ell}
    =
    t^{-\ell+1/2}
    \left(\widetilde A_{\ell\ell}^\delta(z)\right)^{-1}
    \overline R_s^{\,\ell},
    \qquad \ell=1,2,3,
\]
where $\widetilde A_{\ell\ell}^\delta$ denotes the $\ell$-th diagonal block of
$\widetilde A^\delta$.  Then
\[
    T_t^{-1}\left(\widetilde A^\delta(z)\right)^{-1}Z_s^\delta\Sigma_r
    =
    t^{-1/2}G(s/t)+\widetilde R_{t,s}^\delta .
\]
Moreover, for $1\le \ell,j\le3$, $p\ge2$, and $0<\tau\le t\le t_0$,
\[
    \mathbb E_z\left[
        \left\|
        \int_0^\tau
        \widetilde R_{t,s}^{\delta,\ell}
        \left(\widetilde R_{t,s}^{\delta,j}\right)^\top ds
        \right\|_{F}^p
    \right]
    \le
    C_p\frac{\tau^{p(\ell+j)}}{t^{p(\ell+j-1)}} .
\]
\end{lemma}

\begin{proof}
The identity is just the block expansion in
Lemma~\ref{lem:appendix-chain-remainder} after multiplication by the signed
chain matrix and by $T_t^{-1}$.  By the same lemma and
Lemma~\ref{lem:appendix-local-chain},
\[
    \left\|\widetilde R_{t,s}^{\delta,\ell}\right\|_{L^{2p}}
    \le C_p t^{-\ell+1/2}s^{\ell-1/2}.
\]
H\"older's inequality implies that
\[
    \mathbb E_z\left[
        \left\|
        \int_0^\tau
        \widetilde R_{t,s}^{\delta,\ell}
        \left(\widetilde R_{t,s}^{\delta,j}\right)^\top ds
        \right\|_{F}^p
    \right]
    \le
    \left(
        \int_0^\tau
        \left\|\widetilde R_{t,s}^{\delta,\ell}
        \widetilde R_{t,s}^{\delta,j}\right\|_{L^p}\,ds
    \right)^p  \le
    C_p\left(
        t^{-\ell-j+1}\int_0^\tau s^{\ell+j-1}\,ds
    \right)^p,
\]
which is the stated estimate.
\end{proof}

These moment bounds allow us to stop the comparison before the normalized
remainder covariance becomes comparable with the deterministic Gram matrix.
The next lemma shows that such an early stopping event has probability of
arbitrarily high polynomial order.

\begin{lemma}[Pigato stopping time for the remainder covariance]
\label{lem:appendix-pigato-stopping}
Let
\[
    Q=\int_0^1G(u)G(u)^\top\,du,
    \qquad q_0=\lambda_{\min}(Q)>0.
\]
For $0<\xi\le1$ define the block matrix process
\[
    \rho_s^{\delta,t,\xi}
    =
    \left(
        \xi^{-(\ell+j-1)}
        \int_0^s
        \widetilde R_{t,u}^{\delta,\ell}
        \left(\widetilde R_{t,u}^{\delta,j}\right)^\top du
    \right)_{1\le \ell,j\le3}
\]
and the stopping time
\[
    S_\xi
    =
    \inf\left\{0\le s\le t:\
        \left\|\rho_s^{\delta,t,\xi}\right\|_{\mathrm{op}}
        \ge q_0/4\right\}\wedge t ,
\]
with the convention $\inf\varnothing=\infty$.
For every $N\ge1$, after decreasing $t_0$ if necessary,
\[
    \sup_{0\le\delta\le1,\ z\in K_0,\ 0<t\le t_0}
    \mathbb P_z(S_\xi<t\xi)
    \le C_N\xi^N,
    \qquad 0<\xi\le1 .
\]
\end{lemma}

\begin{proof}
Let
\begin{align}\label{defn:D:xi}
    D_\xi
    :=
    \operatorname{diag}
    \left(\xi^{1/2}I_d,\xi^{3/2}I_d,
        \xi^{5/2}I_d\right).
\end{align}
By definition,
\[
    \rho_s^{\delta,t,\xi}
    =
    D_\xi^{-1}
    \left(
        \int_0^s
        \widetilde R_{t,u}^{\delta}
        \left(\widetilde R_{t,u}^{\delta}\right)^\top du
    \right)
    D_\xi^{-1}.
\]
Hence $0\preceq\rho_s^{\delta,t,\xi}
\preceq\rho_v^{\delta,t,\xi}$ for $s\le v$.  Therefore
\[
    \{S_\xi<t\xi\}
    \subseteq
    \left\{
        \left\|\rho_{t\xi}^{\delta,t,\xi}\right\|_{\mathrm{op}}
        \ge q_0/4
    \right\}.
\]
For every $p\ge2$, Markov's inequality implies that
\[
    \mathbb P_z(S_\xi<t\xi)
    \le
    \left(\frac4{q_0}\right)^p
    \mathbb E_z
    \left\|\rho_{t\xi}^{\delta,t,\xi}\right\|_{\mathrm{op}}^p .
\]
Using the finite block decomposition and
Lemma~\ref{lem:appendix-remainder-covariance} with $\tau=t\xi$,
\[
    \mathbb E_z\left\|\rho_{t\xi}^{\delta,t,\xi}\right\|_{\mathrm{op}}^p
    \le
    C_p\sum_{\ell,j=1}^3
    \xi^{-p(\ell+j-1)}
    \frac{(t\xi)^{p(\ell+j)}}{t^{p(\ell+j-1)}}
    \le C_p t^p\xi^p .
\]
This is the point at which the apparently critical remainder scale is used:
after the covariance integration on $[0,t\xi]$ and the anisotropic
$D_\xi$ normalization, every block gains the common factor
$t\xi$.  The right side is bounded by $C_p\xi^p$ for
$t\le t_0\le1$.  Markov's inequality, with $p$ chosen larger than $N$, proves
the claim.
\end{proof}

Outside the exceptional stopping event, the deterministic Gram matrix
dominates the remainder covariance.  This yields lower-tail estimates and,
consequently, negative moments for the reduced covariance.

\begin{lemma}[Reduced covariance negative moments]
\label{lem:appendix-reduced-negative-moments}
Set
\[
    C_{t,r}^{\delta,\mathsf O}
    =
    \int_0^tZ_s^\delta\Sigma_r\Sigma_r^\top\left(Z_s^\delta\right)^\top\,ds,
    \qquad
    \overline C_{t,r}^\delta
    =
    T_t^{-1}\left(\widetilde A^\delta(z)\right)^{-1}
    C_{t,r}^{\delta,\mathsf O}
    \left(\widetilde A^\delta(z)\right)^{-\top}T_t^{-1}.
\]
For every $q\ge1$, after decreasing $t_0$ if necessary,
\[
    \sup_{\substack{0\le\delta\le1,\ z\in K_0\\0<t\le t_0}}
    \mathbb E_z\left[
        \lambda_{\min}\left(\overline C_{t,r}^\delta\right)^{-q}
    \right]<\infty .
\]
\end{lemma}

\begin{proof}
By Lemma~\ref{lem:appendix-remainder-covariance},
\[
    \overline C_{t,r}^\delta
    =
    \int_0^t
    \left(t^{-1/2}G(s/t)+\widetilde R_{t,s}^\delta\right)
    \left(t^{-1/2}G(s/t)+\widetilde R_{t,s}^\delta\right)^\top ds .
\]
On $\{S_\xi\ge t\xi\}$ we integrate only over
$[0,t\xi]$ and use
$(a+b)(a+b)^\top\succeq \frac12aa^\top-bb^\top$.  With
$D_\xi$ defined in \eqref{defn:D:xi}, we have
\begin{equation}\label{eq:appendix-leading-gram}
    \frac1t\int_0^{t\xi}G(s/t)G(s/t)^\top ds
    =D_\xi QD_\xi .
\end{equation}
Moreover,
\[
    \rho_{t\xi}^{\delta,t,\xi}
    =
    D_\xi^{-1}
    \left(
        \int_0^{t\xi}
        \widetilde R_{t,s}^\delta
        \left(\widetilde R_{t,s}^\delta\right)^\top ds
    \right)
    D_\xi^{-1}.
\]
Thus, on $\{S_\xi\ge t\xi\}$,
\begin{equation}\label{eq:appendix-stopped-remainder-upper}
    \int_0^{t\xi}
    \widetilde R_{t,s}^\delta
    \left(\widetilde R_{t,s}^\delta\right)^\top ds
    \preceq
    \frac{q_0}{4}D_\xi^2 .
\end{equation}
Combining \eqref{eq:appendix-leading-gram} and
\eqref{eq:appendix-stopped-remainder-upper} with
$(a+b)(a+b)^\top\succeq \frac12aa^\top-bb^\top$ gives
\[
    \overline C_{t,r}^\delta\succeq \frac{q_0}{4}D_\xi^2,
    \qquad
    \lambda_{\min}\left(\overline C_{t,r}^\delta\right)\ge c\xi^5,
\]
on $\{S_\xi\ge t\xi\}$, with $c=q_0/4$.  Thus, the deterministic
Gram matrix is required to dominate the remainder covariance only on the
stopped event; the complement is handled by
Lemma~\ref{lem:appendix-pigato-stopping}.  More precisely,
\[
    \left\{
        \lambda_{\min}\left(\overline C_{t,r}^\delta\right)<c\xi^5
    \right\}
    \subseteq
    \{S_\xi<t\xi\},
\]
and hence, for every $N\ge1$,
\[
    \mathbb P_z\left(
        \lambda_{\min}\left(\overline C_{t,r}^\delta\right)<c\xi^5
    \right)
    \le C_N\xi^N .
\]
Therefore, for $0<\eta\le c$ and $\xi=(\eta/c)^{1/5}$,
\[
    \mathbb P_z\left(
        \lambda_{\min}\left(\overline C_{t,r}^\delta\right)<\eta
    \right)
    \le C_N\eta^{N/5},
\]
for every $N\ge1$.  For $c<\eta\le1$ the same bound is absorbed by increasing
the constant.  Since $N$ is arbitrary, replacing it by $5N$ gives, after
renaming the constant,
\[
    \mathbb P_z\left(
        \lambda_{\min}\left(\overline C_{t,r}^\delta\right)<\eta
    \right)
    \le C_N\eta^N,
    \qquad 0<\eta\le1,
\]
for every $N\ge1$.  Choosing $N>q$, tail integration yields that
\[
    \mathbb E_z\left[
        \lambda_{\min}\left(\overline C_{t,r}^\delta\right)^{-q}
    \right]
    \le
    1+q\int_0^1
        \eta^{-q-1}
        \mathbb P_z\left(
            \lambda_{\min}\left(\overline C_{t,r}^\delta\right)<\eta
        \right)d\eta
    <\infty,
\]
uniformly in $z$, $\delta$, and $0<t\le t_0$.
\end{proof}

The reduced covariance omits the endpoint Jacobian factor $Y_t^\delta$.  To
restore this factor, we first control the inverse Jacobian
under the anisotropic conjugation by $T_t$.

\begin{lemma}[Scaled inverse Jacobian moments]
\label{lem:appendix-scaled-inverse-jacobian}
For every $m\ge1$, after decreasing $t_0$ if necessary,
\[
    \sup_{\substack{0\le\delta\le1,\ z\in K_0\\0<t\le t_0}}
    \mathbb E_z\left[
        \left\|T_t^{-1}Z_t^\delta T_t\right\|^m
    \right]<\infty .
\]
\end{lemma}

\begin{proof}
This is the final estimate of Lemma~\ref{lem:appendix-chain-remainder}.  It
follows from the inverse-flow equation, uniform Gr\"onwall bounds for all
blocks of $Z_t^\delta$, and the triangular estimates
$(Z_t^\delta)_{21}=\mathcal O(t)$,
$(Z_t^\delta)_{32}=\mathcal O(t)$, and
$(Z_t^\delta)_{31}=\mathcal O(t^2)$ in every $L^m$ norm.  The diagonal weights in
$T_t^{-1}Z_t^\delta T_t$ exactly compensate these lower-block powers.
\end{proof}

Combining reduced-covariance non-degeneracy with the scaled inverse-Jacobian
moments transfers the negative-moment bound to the full Malliavin covariance.

\begin{lemma}[Short-time full covariance negative moments]
\label{lem:appendix-negative-moments}
For every $q\ge1$, after decreasing $t_0$ if necessary,
\[
    \sup_{\substack{
        0\le\delta\le1,\ z\in K_0\\
        0<t\le t_0}}
    \mathbb E_z\left[
        \lambda_{\min}\left(
            T_t^{-1}\mathcal M_{t,r}^{\delta,\mathsf O}T_t^{-1}
        \right)^{-q}
    \right]<\infty .
\]
\end{lemma}

\begin{proof}
Since $J_{t,s}^{\delta,\mathsf O}=Y_t^\delta Z_s^\delta$,
\[
    \mathcal M_{t,r}^{\delta,\mathsf O}
    =Y_t^\delta C_{t,r}^{\delta,\mathsf O}\left(Y_t^\delta\right)^\top .
\]
Put
\[
    K_t^\delta=T_t^{-1}Y_t^\delta\widetilde A^\delta(z)T_t .
\]
Then
\[
    T_t^{-1}\mathcal M_{t,r}^{\delta,\mathsf O}T_t^{-1}
    =K_t^\delta\overline C_{t,r}^\delta\left(K_t^\delta\right)^\top .
\]
Moreover,
\[
    \left(K_t^\delta\right)^{-1}
    =T_t^{-1}\left(\widetilde A^\delta(z)\right)^{-1}Z_t^\delta T_t .
\]
Both $T_t$ and $\widetilde A^\delta(z)$ are block diagonal, and each block of
$T_t$ is a scalar multiple of $I_d$; hence they commute.  Consequently,
\[
    \left(K_t^\delta\right)^{-1}
    =\left(\widetilde A^\delta(z)\right)^{-1}
      \left(T_t^{-1}Z_t^\delta T_t\right).
\]
Lemma~\ref{lem:appendix-local-chain} and
Lemma~\ref{lem:appendix-scaled-inverse-jacobian} give finite moments of
$\|(K_t^\delta)^{-1}\|$ of every order, uniformly in $z$, $\delta$, and
$0<t\le t_0$.  Since
\[
    \lambda_{\min}\left(
        K_t^\delta\overline C_{t,r}^\delta(K_t^\delta)^\top
    \right)^{-q}
    \le
    \left\|\left(K_t^\delta\right)^{-1}\right\|^{2q}
    \lambda_{\min}\left(\overline C_{t,r}^\delta\right)^{-q},
\]
H\"older's inequality and Lemma~\ref{lem:appendix-reduced-negative-moments}
prove the short-time bound.
\end{proof}

The preceding argument establishes non-degeneracy on a short initial
interval.  The flow decomposition now transports this lower bound to every
fixed finite time horizon.

\begin{lemma}[Finite-time extension]\label{lem:appendix-finite-time-extension}
For every $q\ge1$, every $M<\infty$, and every compact
$K_0\subset\mathsf O$,
\[
    \sup_{\substack{
        0\le\delta\le1,\ z\in K_0\\
        0<t\le M}}
    \mathbb E_z\left[
        \lambda_{\min}\left(
            T_t^{-1}\mathcal M_{t,r}^{\delta,\mathsf O}T_t^{-1}
        \right)^{-q}
    \right]<\infty .
\]
The same estimate holds for the adjoint cutoff family.
\end{lemma}

\begin{proof}
The range $0<t\le t_0$ is Lemma~\ref{lem:appendix-negative-moments}.  Let
$t_0\le t\le M$.  The covariance over the first interval $[0,t_0]$ gives the
positive-semidefinite lower bound
\[
    \mathcal M_{t,r}^{\delta,\mathsf O}
    \succeq
    J_{t,t_0}^{\delta,\mathsf O}
    \mathcal M_{t_0,r}^{\delta,\mathsf O}
    \left(J_{t,t_0}^{\delta,\mathsf O}\right)^\top .
\]
Thus, before applying the anisotropic normalization,
\[
    \lambda_{\min}\left(\mathcal M_{t,r}^{\delta,\mathsf O}\right)^{-q}
    \le
    \left\|\left(J_{t,t_0}^{\delta,\mathsf O}\right)^{-1}\right\|^{2q}
    \lambda_{\min}\left(\mathcal M_{t_0,r}^{\delta,\mathsf O}\right)^{-q}.
\]
With
\[
    K_{t,t_0}^\delta
    =
    T_t^{-1}J_{t,t_0}^{\delta,\mathsf O}T_{t_0},
\]
this implies
\[
    \lambda_{\min}\left(
        T_t^{-1}\mathcal M_{t,r}^{\delta,\mathsf O}T_t^{-1}
    \right)^{-q}
    \le
    \left\|\left(K_{t,t_0}^\delta\right)^{-1}\right\|^{2q}
    \lambda_{\min}\left(
        T_{t_0}^{-1}\mathcal M_{t_0,r}^{\delta,\mathsf O}T_{t_0}^{-1}
    \right)^{-q}.
\]
Equivalently, the lower bound generated on $[0,t_0]$ is transported by
$K_{t,t_0}^\delta$ and loses only the inverse singular value of this rescaled
Jacobian.  The weak H\"ormander lower bound is used only on the initial interval
$[0,t_0]$; for $t\ge t_0$ the argument uses only moments of the Jacobian flow
and its inverse.
The inverse matrix equals
$T_{t_0}^{-1}J_{t_0,t}^{\delta,\mathsf O}T_t$.  Because
$t\in[t_0,M]$, the deterministic matrices $T_t$, $T_t^{-1}$,
$T_{t_0}$, and $T_{t_0}^{-1}$ are bounded by constants depending only on
$t_0$ and $M$.  The cutoff drift has uniformly bounded derivatives, so the
Jacobian flow and its inverse have finite moments of every order, uniformly in
$0\le\delta\le1$, $z\in K_0$, and $t\in[t_0,M]$; explicitly,
\[
    \sup_{\substack{0\le\delta\le1,\ z\in K_0\\t_0\le t\le M}}
    \mathbb E_z
    \left\|
        \left(J_{t,t_0}^{\delta,\mathsf O}\right)^{-1}
    \right\|^p
    <\infty,\qquad p\ge1,
\]
Together with the bounded deterministic scaling matrices,
Lemma~\ref{lem:appendix-negative-moments}, and H\"older's inequality, this
proves the finite-time bound.  For the adjoint cutoff family, the drift changes only the
signs of the local chain links.  The signed matrix
$\widetilde A^\delta(z)$ is therefore still uniformly invertible, the Gram
matrix $Q=\int_0^1G(u)G(u)^\top du$ is unchanged, the cutoff derivative bounds
and scaled Jacobian estimates are identical, and the constant $p$- and
$\theta$-diffusion fields enter only through the same uniformly controlled
remainder estimates.  Thus all constants can be chosen common to the forward
and adjoint families.
\end{proof}

\bibliographystyle{alpha}
\bibliography{bibtex}

\end{document}